\documentclass[a4paper,12pt,reqno,oneside]{amsart} 

\usepackage{setspace} 
\usepackage{typearea} 
\usepackage{amscd,amssymb,verbatim} 

\usepackage{xr-hyper}
\usepackage{cite}
\usepackage[colorlinks,backref,final,bookmarksnumbered,bookmarks]{hyperref}
\usepackage[backrefs]{amsrefs}

\usepackage{ifthen}
\newcommand{\arxiv}[2][]{\ifthenelse{\equal{#1}{}}
{\href{http://arxiv.org/abs/#2}{\tt arXiv:#2}}
{\href{http://arxiv.org/abs/math/#2}{\tt arXiv:math.#1/#2}}}

\theoremstyle{plain}
\newtheorem{theorem}{Theorem}[section]
\newtheorem{lemma}[theorem]{Lemma}

\newtheorem{corollary}[theorem]{Corollary}
\newtheorem{proposition}[theorem]{Proposition}
\newtheorem{problem}[theorem]{Problem}
\newtheorem{addendum}[theorem]{Addendum}
\newtheorem{maintheorem}{Theorem}
\newtheorem*{addendum*}{Addendum}

\theoremstyle{definition}
\newtheorem{example}[theorem]{Example}

\newtheoremstyle{remark}
{}{}{}{}{\itshape}{}{ }{\thmname{#1}\thmnumber{ \itshape #2.}}
\theoremstyle{remark}
\newtheorem{remark}[theorem]{Remark}

\DeclareMathOperator{\id}{id}
\DeclareMathOperator{\Int}{Int}
\DeclareMathOperator{\Hom}{Hom}
\DeclareMathOperator{\lk}{lk}
\DeclareMathOperator{\st}{st}
\DeclareMathOperator{\coker}{coker}

\def\x{\times}
\def\but{\setminus}
\def\emb{\hookrightarrow}

\def\eps{\epsilon}
\def\phi{\varphi}
\def\tl{\tilde}
\def\ccirc{\circledcirc}
\def\emptyset{\varnothing}
\renewcommand{\:}{\colon}
\def\xr#1{\xrightarrow{#1}}

\def\R{\mathbb{R}}
\def\Z{\mathbb{Z}}

\def\P{\mathcal{P}}
\def\T{\mathcal{T}}
\def\J{\mathcal{J}}

\begin{document}

\title{Lifting generic maps to embeddings. The double point obstruction}
\author{Sergey A. Melikhov}
\address{Steklov Mathematical Institute of Russian Academy of Sciences,
ul.\ Gubkina 8, Moscow, 119991 Russia}
\email{melikhov@mi-ras.ru}


\begin{abstract}
Given a generic PL map or a generic smooth fold map $f\:N^n\to M^m$, where 
$m\ge n$ and $2(m+k)\ge 3(n+1)$, we prove that $f$ lifts to a PL or smooth embedding 
$N\to M\x\R^k$ if and only if its double point locus $\{(x,y)\in N\x N\mid f(x)=f(y),\,x\ne y\}$ 
admits an equivariant map to $S^{k-1}$.
As a corollary we answer a 1990 question of P.~ Petersen and obtain some other applications.

We also discuss several criteria for lifting a non-degenerate PL map or a $C^0$-stable smooth map
$f\:N^n\to M^m$, where $m\ge n$, to an embedding in $M\x\R$, elaborating on V.~ Po\'enaru's observations.
In particular, the existence of such a lift is determined by the equivariant homotopy type of
the diagram consisting of the three projections from the triple point locus 
$\{(x,y,z)\in N\x N\x N\mid f(x)=f(y)=f(z),\,x\ne y\ne z\ne x\}$ to the double point locus.

The three Appendices, which can be read independently of the rest of the paper, are devoted to stable and 
generic maps.
Appendix \ref{pl-maps} introduces an elementary theory of stable PL maps.
Appendix \ref{sec-xx} extends the 2-multi-0-jet transversality theorem over the usual compactification of 
$M\x M\but\Delta_M$.
\end{abstract}

\maketitle

\section{Introduction} \label{sec1}

A continuous/PL/smooth map $f\:N\to M$ is called a topological/PL/smooth {\it $k$-prem} (from 
``{\bf pr}ojected {\bf em}bedding'') if there exists a map $g\:N\to\R^k$ such that 
$f\x g\:N\to M\x\R^k$ is a topological/PL/smooth embedding.
When the choice of a category is irrelevant, we will speak simply of ``$k$-prems''.
A discussion of this definition, including a comparison of the three categories, as well as a number of
references on $k$-prems can be found in the companion paper \cite{partI}.

The main objective of the present paper is to determine algebraically, under some reasonable hypotheses, 
whether a given map $f$ is a $k$-prem.
The main theorem of the present paper is applied in \cite{AM} and was originally motivated by that application.

For a space $N$ let $\Delta_N=\{(x,x)\in N\x N\}$ and $\tl N=N\x N\but\Delta_N$.
We endow $\tl N$ with the factor exchanging involution, and we also endow each sphere $S^n$ 
with the antipodal involution.
By a ``map'' we will mean a continuous map.
Given a map $f\:N\to M$, let $\Delta_f=\{(x,y)\in\tl N\mid f(x)=f(y)\}$.

A necessary condition for $f\:N\to M$ to be a $k$-prem is the existence of an equivariant map 
$\tl g\:\Delta_f\to S^{k-1}$.
Namely, $\tl g(x,y)=\frac{g(y)-g(x)}{||g(y)-g(x)||}$, where $g\:N\to\R^k$ is a map such that
$f\x g\:N\to M\x\R^k$ is an embedding.
In general, this condition is not sufficient.

\begin{example} 
The $3$-fold covering $f\:S^1\to S^1$ is not a $1$-prem, but there exists an equivariant map $\Delta_f\to S^0$.
Indeed, $\Delta_f$ is homeomorphic to $S^1\sqcup S^1$, where the involution permutes the two components.

Also there exists a non-degenerate%
\footnote{A PL map is called {\it non-degenerate} if it has no point-inverses of dimension $>0$.}
PL map $f$ from $S^1$ to the triod (=the cone over the three-point set) which is not a $1$-prem even though
there exists an equivariant map $\Delta_f\to S^0$ \cite{Go}*{Example 5}.
Namely, $f$ is the composition of the $3$-fold covering $S^1\to S^1$ and the unique simplicial surjection
from the boundary of the hexagon onto the triod.

However, if $f$ is a stable%
\footnote{Stable PL maps are defined in Appendix \ref{pl-maps}. 
A PL map $f$ of a graph in $\R^1$ is stable if and only if every point-inverse $f^{-1}(y)$
contains at most one point $x$ such that $f$ does not restrict to a homeomorphism between
a neighborhood of $x$ and a neighborhood of $y$.}
PL map of a graph to $\R^1$, then $f$ is a $1$-prem if and only if 
there exists an equivariant map $\Delta_f\to S^0$ \cite{Go}*{Theorem 11} (this is deduced in \cite{Go} from 
the result of \cite{FK}; in connection with the latter see also \cite{Sk}, \cite{AFT}).
Let us note that there exists a stable PL map of a trivalent {\it tree} to $\R^1$ which is 
not a $1$-prem \cite{Si} (see also \cite{ARS}*{\S3.1}).
\end{example}

More examples will be discussed toward the end of this introduction.

\begin{maintheorem}\label{th1} Suppose that $m\ge n$ and $2(m+k)\ge 3(n+1)$.
Let $f\:N^n\to M^m$, where $N$ is compact, be one of the following:

(a) a generic PL map of a polyhedron to a PL manifold;

(b) a generic smooth fold map%
\footnote{A {\it smooth fold map} is a smooth map whose only singularities are of the fold type.
In particular, every smooth immersion is trivially a smooth fold map.}
between smooth manifolds;

(c) a generic smooth map between smooth manifolds, where $3n-2m\le k$.

\noindent
If there exists an equivariant map $\Delta_f\to S^{k-1}$, then $f$ is a PL $k$-prem in the case of (a) and 
a smooth $k$-prem in the cases of (b) and (c).
\end{maintheorem}

In this statement and in what follows we assume the following conventions.

1) The dimension of a manifold or a polyhedron is often indicated by a superscript.

2) ``Generic'' is understood as ``satisfying a certain fixed set of conditions that determine an open and 
dense set of maps'' (see \S\ref{generic0} for the details).
This is a slight variation of Gromov's terminology \cite{Gr}*{1.3.2(B)} (what he calls ``generic''
will be called ``weakly generic'' in Appendix \ref{sec-xx}).

In order to be able to apply Theorem \ref{th1} to specific maps, one may want to know explicitly 
the dense open sets whose existence it asserts.
It is not hard to see (see Propositions \ref{stability} and \ref{stability-pl}) that parts (a) and (b) 
are equivalent to their versions with ``generic'' replaced by ``stable''.%
\footnote{Stable smooth maps are discussed in a number of textbooks, including \cite{GG}, \cite{AGV}.
See also Appendix~ \ref{stable maps}.}
Also, it is easy to see that part (c) implies its version with ``generic'' replaced by ``stable'' 
(see the proof of Proposition \ref{stability}).
The converse implication holds in a wide range of dimensions (in particular, for all $n\le 8$), since 
by Mather's theorem \cite{MaVI} generic smooth maps $N^n\to M^m$, where $m\ge n$, are stable if 
either $6m\ge 7n-7$, or $6m\ge 7n-8$ and $m\le n+3$.

\begin{corollary} \label{stable}
Theorem \ref{th1} is true if ``generic'' is replaced by ``stable''.
\end{corollary}

\begin{remark} \label{constructive}
In Theorem \ref{th1}(a) one may also replace ``generic'' by another explicit condition, which is less restrictive
and easier to check than stability (see Remark \ref{gen-pos}, which uses the notation introduced in 
the first few lines of \S\ref{fold maps}).

As for part (c) of Theorem \ref{th1}, it is reduced to a relative version of (b) by a half-page argument 
(see \S\ref{1-mod-main}), and we leave it to the interested reader to formulate an explicit list of genericity 
conditions that are needed for that argument.
That list is not going to be short, although it might be possible to make it shorter by extra work.
However, the author is not convinced that part (c) is a definitive result; it seems more likely that
the real theorem in this direction has yet to be proved (see Remark \ref{improving}).
\end{remark} 

\begin{remark} \label{obvious}
Theorem \ref{th1} is obvious for immersions without triple points.
In this case, the projection $\pi\:N\x N\to N$ embeds $\Delta_f$.
Then any extension of the composition 
$\pi(\Delta_f)\xr{\pi^{-1}}\Delta_f\xr{\phi}S^{k-1}\subset\R^k$ to a PL or smooth map 
$g\:N\to\R^k$ yields a PL or smooth embedding $f\x g\:N\emb M\x\R^k$.

It is not hard to elaborate on this construction so as to prove Theorem \ref{th1} for generic maps 
without triple points (see Proposition \ref{3pt-free} for the PL case).
Let us note that a generic smooth map without triple points is a smooth fold map.
\end{remark}

\begin{remark} If $k=1$, then generic maps $N^n\to M^m$, where $2(m+k)\ge 3(n+1)$, have no triple points.
Thus the case $k=1$ in Theorem \ref{th1} is also quite easy. 
\end{remark}

\begin{remark} 
A natural approach to proving Theorem \ref{th1} is by trying to adapt Haefliger's generalized 
Whitney trick (see \cite{Ad}*{\S VII.4}; see also \cite{Sz1} for an alternative proof).
In fact, our proof below can be used to embed a core part of Haefliger's 
``standard model'' into a generic lift of $f$ in a way that agrees with 
the projection $M\x\R^k\to M$.
However, it is far from clear how one could possibly construct appropriate global 
Haefliger-style framings especially when $\Delta_f$ is not immersed in $N$.

We prove Theorem \ref{th1} by what can be called a new kind of generalized Whitney trick. 
It contrasts with the Whitney--Haefliger(--Koschorke) approach in that ours describes the desired homotopy 
essentially by an explicit formula.
(The formula is only aware of pairs of points, and may overlook triple-point information; because of this,
the homotopy given by the formula has to be slightly perturbed.)
Haefliger's construction is less explicit in that it depends on the choice of an embedding of 
the ``standard model'', whose existence is proved using obstruction theory.
However, it is difficult to compare the two versions of generalized Whitney trick directly because they 
apply under incompatible hypotheses.
\end{remark}

\begin{remark} \label{improving}
Theorem \ref{th1}(c) falls short of the announcements in \cite{M1}*{third remark after Theorem 5} 
and \cite{M3}*{\S1}, where it was conjectured to hold under the weaker restriction $4n-3m\le k$.
(See Remark \ref{announcement} below for an explanation.)
The main obstacle in extending Theorem \ref{th1}(b) to maps with more general singularities
is that already in the presence of cusps (i.e., singularities of type $\Sigma^{1,1,0}$) it is no longer easy 
to maintain the connection between the geometric maps and the configuration space data throughout the proof.
\end{remark}

\begin{proposition} \label{trivial range} 
Every stable smooth (PL) map $f\:N^n\to M^{2n+1-k}$ is a smooth (PL) $k$-prem.
\end{proposition}

\begin{proof} We will only discuss the smooth case here; the PL case is proved similarly.
Let $g\:N\to M\x\R^k$ be a smooth embedding that is sufficiently $C^\infty$-close to
the composition $N\xr{f} M\xr{\iota} M\x\R^k$, where $\iota$ is the inclusion onto $M\x\{0\}$.
Since $f$ is $C^\infty$-stable, we may assume that the composition $N\xr{g} M\x\R^k\xr{\pi}M$,
where $\pi$ is the projection, is $C^\infty$-left-right-equivalent to $f$.
But then clearly $g$ itself lifts to a smooth embedding.
\end{proof}

In the PL case, one can actually say more:

\begin{maintheorem} \label{nondeg} Every non-degenerate PL map of compact $n$-polyhedra is a PL $(n+1)$-prem.
\end{maintheorem}

\begin{proof} The given PL map is triangulated by a simplicial map $f\:K\to L$.
This in turn yields a simplicial map $K'\to L'$ between the barycentric subdivisions.
For each $i$-simplex $\tau$ of $L$ let us embed the preimage $f^{-1}(\hat\tau)$ of its
barycenter $\hat\tau$ into the $(i+1)$st coordinate axis $\left<e_{i+1}\right>$ of $\R^{n+1}$. 
Let us define $g\:K\to\R^{n+1}$ by these embeddings on the vertices of $K'$ and by extending 
linearly to each simplex of $K'$.
It remains to show that $f\x g\:|K|\to |L|\x\R^{n+1}$ is injective.

Given distinct points $x,y\in|K|$ such that $f(x)=f(y)$, they lie in the interiors of some 
distinct simplexes $s$, $t$ of $K'$ such that $f(s)=f(t)$.
Now $s$ is a join $\hat\sigma^{i_0}*\dots*\hat\sigma^{i_m}$, where 
$\sigma^{i_0}\subsetneqq\dots\subsetneqq\sigma^{i_m}$ are simplexes of $K$.
Moreover, each $f(\sigma^{i_k})=f(\tau^{i_k})$, where $\tau^{i_k}$ is the corresponding factor
of the similar decomposition $t=\hat\tau^{i_0}*\dots*\hat\tau^{i_m}$.
Since $s\ne t$, we have $\sigma^{i_m}\ne\tau^{i_m}$ and consequently 
$\hat\sigma^{i_m}\ne\hat\tau^{i_m}$.
Then also $g(\hat\sigma^{i_m})\ne g(\hat\tau^{i_m})$.

Writing $r=f(s)=f(t)$, we similarly have $r=\hat\rho^{i_0}*\dots*\hat\rho^{i_m}$.
Thinking of $r$ as a subset of $\R^m$, each line $L_k:=\hat\rho^{i_k}\x\left<e_{i_k+1}\right>$ 
is skew with respect to the affine subspace $A_k$ of $\R^m\x\R^{n+1}$ spanned by 
$L_0,\dots,L_{k-1}$ (that is, $L_k$ and $A_k$ are disjoint and contain no parallel vectors).
Therefore the join $L_0*\dots*L_m$ is contained in $\R^m\x\R^{n+1}$, and in fact in $r\x\R^{n+1}$.
Since $f\x g$ sends $\hat\sigma^{i_m}$ and $\hat\tau^{i_m}$ into distinct points of $L_m$,
it must also send $x$ and $y$ to distinct points of $L_0*\dots*L_m\subset r\x\R^{n+1}$.
\end{proof}

\begin{remark} Theorem \ref{nondeg} yields a simple explicit construction of an embedding of a given 
$n$-polyhedron in $\R^{2n+1}$ (see \cite{MZ}*{1.2} for another simple explicit construction).
Indeed, the barycentric subdivision $K'$ of a given $n$-dimensional simplicial complex $K$ admits a
non-degenerate simplicial map onto the $n$-simplex $\Delta^n$ (by sending the barycenter of each 
$i$-simplex of $K$ to the $(i+1)$st vertex of $\Delta^n$ and extending linearly).
\end{remark}

The following application of Theorem \ref{th1} is deduced from it in \S\ref{sec5}
by elaborating on some arguments due to P. M. Akhmetiev.

\begin{maintheorem} \label{akhmetiev}
Let $f\:N\to M$ be a stable smooth map between stably parallelizable smooth $n$-manifolds,
where $n\ne 1,2,3,7$.
Then $f$ is an $n$-prem.
\end{maintheorem}

Akhmetiev himself proved that every $f$ as above is {\it $n$-realizable} in the sense that the composition 
$N\xr{f}M\xr{\iota} M\x\R^n$ with the inclusion $\iota$ onto $M\x\{0\}$ is $C^0$-approximable by smooth embeddings
\cite{A1}, \cite{A2} (see also \cite{M1}*{\S5} and \cite{A3} for alternative proofs).
It is possible to deduce Theorem \ref{akhmetiev} from the statement of this original result of Akhmetiev
(although the direct proof of Theorem \ref{akhmetiev} given in \S\ref{sec5} is simpler than the one
given by this deduction). 
Namely, it is proved in \cite{AM} using Theorem \ref{th1} of the present paper that if $f\:N\to M$ is 
a stable smooth map between smooth $n$-manifolds (resp.\ a stable PL map from an $n$-polyhedron to 
a PL $n$-manifold), where $n\ge 3$, then $f$ is $n$-realizable if and only if it is a smooth 
(resp.\ PL) $n$-prem.
In fact, the equivalence of $k$-realizability and the $k$-prem property is proved in \cite{AM}
under much weaker hypotheses, which are however strictly stronger than those of Theorem \ref{th1}.
Moreover, there are explicit examples in \cite{AM} showing that a map $f$ as in Theorem \ref{th1}
can be $k$-realizable without being a $k$-prem.

Theorem \ref{akhmetiev} fails for $n=1,3,7$ by considering e.g.\ the double cover $S^n\to\R P^n$
(cf.\ \cite{M1}*{Example 3}).
It is unknown if it holds for $n=2$. 
Strong partial results on this problem were obtained by L. Funar and P. Pagotto \cite{FP}; 
see also \cite{M3} for another approach.

Here is another application of Theorem \ref{th1} (the proof is in \S\ref{sec5}).

\begin{maintheorem}\label{th3} Let $N$ be a PL (smooth) $\Z/2$-homology $n$-sphere, $M$ an orientable 
PL (smooth) $n$-manifold, $n>2$.
A stable PL (smooth) map $f\:N\to M$ is a PL (smooth) $n$-prem if either

(a) $\deg(f)$ is zero or odd; or

(b) $f_*\:\pi_1(N)\to\pi_1(M)$ is onto.
\end{maintheorem}

By a PL $\Z/2$-homology $n$-sphere we mean a PL manifold (not just a polyhedral homology manifold) 
with the same $\bmod 2$ homology as $S^n$.

Theorem \ref{th3}(b) implies, in particular, the following:

\begin{corollary} \label{spheres} Stable smooth maps $S^n\to S^n$ are smooth $n$-prems for $n>2$.
\end{corollary}

For $n=2$ this is also true \cite{M1}*{Proof of the Yamamoto--Akhmetiev Theorem}.
In the cases $n=3,7$ Corollary \ref{spheres} is a rather delicate result, in the sense that some closely related 
assertions are false:
\begin{itemize}
\item the universal covering map of the Poincar\'e homology sphere is not a $3$-prem \cite{M1}*{Example 3}; 
\item a certain stable self-map of the Poincar\'e homology sphere is not a $3$-prem \cite{M1}*{Remark to Theorem 4};
\item for $n=3,7$ there exists an equivariant map $F\:S^n\x S^n\to S^n\x S^n$ such that $F|_{\tilde S^n}$
is transverse to $\Delta_{S^n}$ and $(F|_{\tilde S^n})^{-1}(\Delta_{S^n})$ admits no equivariant map to 
$S^{n-1}$ \cite{A4} (see also Lemma \ref{5.1} below).%
\footnote{It is also asserted in the main result of \cite{A4} that such equivariant maps $S^n\x S^n\to S^n\x S^n$
do not exist for $n\ne 1,3,7$.
A proof of this assertion can be found in the present paper, in the proof of Theorem \ref{akhmetiev} in \S\ref{sec5};
see also \cite{A3} for a different approach.}
\end{itemize}
Corollary \ref{spheres} is also a consequence of the author's previous result that stable smooth maps $S^n\to S^n$
are $n$-realizable for $n>2$ \cite{M1} and the aforementioned result of \cite{AM}, proved using 
Theorem \ref{th1} of the present paper.

Theorem \ref{th3}(a) yields a negative solution to Petersen's problem \cite{Pe}*{end of \S3}: 
does there exist a $3$-dimensional lens space $L(p,q)$ with $p$ odd whose
universal covering is not a $3$-prem?
(The condition ``with $p$ odd'' is not explicitly stated in Petersen's question, but is implied 
by the preceding discussion and certainly by his observation on the preceding page that no even 
degree covering $S^3\to M^3$ is a $3$-prem.)

\begin{corollary} The universal covering of every $3$-dimensional lens space $L(p,q)$ with $p$ odd
is a $3$-prem.
\end{corollary}

The following example shows that the restriction $2(m+k)\ge 3(n+1)$ in Theorem \ref{th1} cannot be relaxed.

\begin{example} \label{Boy}
It is easy to construct a self-transverse immersion $f$ of the M\"obius band in $\R^3$ such that $\Delta_f$ is 
homeomorphic to $S^1$ with the antipodal involution.
On the other hand, if $f$ is a self-transverse immersion of an orientable $2$-manifold $N$ in $\R^3$, 
it is easy to see that $\Delta_f$ admits an equivariant map to $S^0$.
Nevertheless, there exist self-transverse immersions of orientable $2$-manifolds in $\R^3$ that are not $1$-prems
\cite{Gi}*{\S2}, \cite{A0} (see also \cite{CS}, \cite{Sa}, \cite{ARS}*{\S3.3}).
Below we describe a variant of this construction, which is particularly simple in that it has only
one triple point in the case of a surface with boundary, and simultaneously extend it to higher dimensions.
\bigskip

{\bf (a)} Let us first describe an immersion $\phi\:G\to H$ between graphs with only one triple point
such that $\Delta_\phi$ admits an equivariant map to $S^0$ but $\phi$ is not a $1$-prem.

Our graphs will have loops and multiple edges.
Let $H$ consist of one vertex $v$ (of degree 6) and three edges $A$, $B$, $C$.
Let $G$ consist of three vertices $a$, $b$, $c$ (each of degree $4$) and the following six edges:
\begin{itemize}
\item $A'$ and $B''$ between $a$ and $b$;
\item $B'$ and $C''$ between $b$ and $c$;
\item $C'$ and $A''$ between $c$ and $a$.
\end{itemize}
We define $\phi$ by $\phi(a)=\phi(b)=\phi(c)=v$, $\phi(A')=\phi(A'')=A$, $\phi(B')=\phi(B'')=B$ and 
$\phi(C')=\phi(C'')=C$.%
\footnote{Let us note that if $g\:\R P^2\to\R^3$ is the Boy's surface, then the restriction of $g$ 
to the image of $\Delta_g$ in $\R P^2$ can be described similarly, except that $A''$ will connect $c$ 
to itself, $B''$ will connect $a$ to itself, and $C''$ will connect $b$ to itself.}
It is easy to see that $\Delta_\phi$ admits an equivariant map to $S^0$.
However, the pairs $(a,b)$, $(b,c)$ and $(c,a)$ lie in the same component of $\Delta_\phi$.%
\footnote{In fact, this remains true even if the edges $C$, $C'$ and $C''$ are deleted from $G$ and $H$.}
Hence if $g\:G\to\R$ is a map such that $\phi\x g\:G\to H\x\R$ is an embedding, then $g(a)-g(b)$, 
$g(b)-g(c)$ and $g(c)-g(a)$ must all have the same sign.
However, this contradicts the fact that their sum is $0$.%
\footnote{In fact, $\{(a,b,c)\}$ is the simplest example of what is called a Penrose staircase
in Theorem \ref{th4}.}
Thus $\phi$ is not a $1$-prem.
\bigskip

{\bf (b)} Next we construct an orientable $2n$-manifold $M$ with boundary and a self-transverse immersion 
$f\:M\to\R^{3n}$ with only one triple point such that $\Delta_f$ admits an equivariant map to $S^0$
but $f$ is not a $1$-prem.

This $f$ will be modeled on $\phi$, in the sense that $G$ and $H$ will appear as the images of 
$\Delta_f$ in $M$ and in $\R^3$, respectively, and $\phi$ will appear as the restriction of $f$.

Let $I=[-1,1]$, and let us consider three $2n$-disks $D_{xy}=I^n\x I^n\x 0$, 
$D_{yz}=0\x I^n\x I^n$ and $D_{zx}=I^n\x 0\x I^n$ in the three coordinate $2n$-planes in $\R^{3n}$.
(These disks are going to be neighborhoods of the vertices $a$, $b$, $c$ respectively, or rather 
their images under the immersion.)
Let $p_\pm=(\pm 1,0,\dots,0)\in I^n$ and let $x_\pm$, $y_\pm$ and $z_\pm$ be the images of $p_\pm$
in $I^n\x 0\x 0$, $0\x I^n\x 0$ and $0\x 0\x I^n$.
Let us consider three arcs with interiors outside $I^{3n}$: an arc $J_{xy}$ connecting $x_+$ with $y_-$; 
an arc $J_{yz}$ connecting $y_+$ with $z_-$; and an arc $J_{zx}$ connecting $z_+$ with $x_-$.
(These arcs are going to lie in the edges $A$, $B$, $C$ respectively.)
Let $e_i$ be the standard $i$th basis vector in $\R^n$, and let $v_i^x$, $v_i^y$ and $v_i^z$ be its images
in $\R^n\x 0\x 0$, $0\x\R^n\x 0$ and $0\x 0\x\R^n$.

We may assume that $J_{xy}$ lies in the $2$-plane through $0$ spanned by $v_1^x$ and $v_1^y$, and similarly
for $J_{yz}$ and $J_{zx}$.
The normal bundle $\lambda$ of $J_{xy}$ in this plane is a line subbundle in the normal bundle $\nu$ of $J_{xy}$
in $\R^{3n}$.
Let us note that $\lambda$ is spanned by (a translate of) $v_1^y$ over $x_+$ and by $v_1^x$ over $y_-$.
Another line subbundle $\lambda'$ in $\nu$ is spanned by $v_1^z$.
The $2$-bundle $\lambda\oplus\lambda'$ can also be represented as $\lambda_1\oplus\lambda_1'$, where
$\lambda_1$ and $\lambda_1'$ coincide with $\lambda$ and $\lambda'$ over $x_+$, and then start rotating 
away from them, eventually (i.e.\ over $y_-$) rotating through an angle of $\pi/2$.
If we do the rotation counterclockwise, then the frame $(v_1^y,v_1^z)$ over $x_+$ is transported along
$\lambda_1\oplus\lambda_1'$ into the frame $(v_1^z,-v_1^ x)$ over $y_-$.

For each $i>1$ the vectors $v_i^x$, $v_i^y$ and $v_i^z$ also span line subbundles $\mu_x$, $\mu_y$ 
and $\mu_z$ of $\nu$.
The $3$-bundle $\mu_x\oplus\mu_y\oplus\mu_z$ can also be represented as $\mu_i\oplus\mu_i'\oplus\mu_i''$, 
where $\mu_i$, $\mu_i'$ and $\mu_i''$ coincide with $\mu_x$, $\mu_y$ and $\mu_z$ over $x_+$, and then start 
rotating around the vector $v_i^x+v_i^y+v_i^z$, so that eventually the three vectors exchange cyclically 
according to the permutation $(xyz)$.
In other words the frame $(v_i^x,v_i^y,v_i^z)$ over $x_+$ is transported along $\mu_i\oplus\mu_i'\oplus\mu_i''$
into the frame $(v_i^y,v_i^z,v_i^x)$ over $y_-$.

Let us now connect $D_{xy}$ with $D_{yz}$ by the total space of the $(2n-1)$-disk bundle over $J_{xy}$ associated with
$\lambda_1\oplus\mu_2\dots\oplus\mu_n\oplus\mu_2'\oplus\dots\oplus\mu_n'$.
(The union of this total space with $D_{xy}$ and $D_{yz}$ is going to be a neighborhood of the edge $A'$,
or rather its image under the immersion.)
We also connect $D_{zx}$ with $D_{xy}$ by the total space of the $(2n-1)$-disk bundle over $J_{xy}$ associated with
$\lambda_1'\oplus\mu_2\oplus\dots\oplus\mu_n\oplus\mu_2''\dots\oplus\mu_n''$.
(This is going to be an immersed neighborhood of $A''$.)
We similarly connect $D_{xy}$ with $D_{yz}$ and $D_{yz}$ with $D_{zx}$ along $J_{yz}$.
(These are going to be immersed neighborhoods of $B'$ and $B''$, respectively.)
Also we connect $D_{yz}$ with $D_{zx}$ and $D_{zx}$ with $D_{xy}$ along $J_{zx}$.
(These are going to be immersed neighborhoods of $C'$ and $C''$, respectively.)

Thus we obtain an orientable manifold with boundary $M^{2n}$ and a self-transverse immersion $f\:M\to\R^{3n}$ 
with one triple point which extends $\phi$ so that $\Delta_f=\Delta_\phi$.
Hence $\Delta_f$ admits an equivariant map to $S^0$ but $f$ is not a $1$-prem.
\bigskip

{\bf (c)} Finally we show that the previous construction can be improved so as to get rid of the boundary.
Thus we will construct a closed orientable $2n$-manifold $W$ and a self-transverse immersion $F\:W\to\R^{3n}$ 
(with several triple points) such that $\Delta_F$ admits an equivariant map to $S^0$ but $F$ is not a $1$-prem.

First we note that in the previous construction, $M$ is a $2n$-dimensional thickening of the graph $G$, 
that is, a handlebody with three $0$-handles and six $1$-handles; and $f$ factors through an immersion 
$f\:(M,\partial M)\to (N,\partial N)$, where $N$ is a $3n$-dimensional thickening of the graph $H$, 
that is, a handlebody with one $0$-handle $h$ and three $1$-handles $h_i$.

Let $B$ be obtained from $N$ by attaching three $2$-handles $h'_i$ so that they cancel geometrically 
the $1$-handles $h_i$.
(Thus $B$ is homeomorphic to the $3n$-ball.)
The $2$-handles are attached to $N$ along thin tubular neighborhoods $T_i$ of three embedded circles 
$C_i\subset\partial N$, and it is not hard to ensure that each $C_i\cap h_i$ is an arc disjoint from
$f(\partial M)$, and each $C_i\cap h$ is an arc which meets $f(\partial M)$ transversely (if at all).
Thus $f(\partial M)$ lies entirely in $\partial B$ and in the $(3n-1)$-disks $Q_i:=T_i\cap h$.
Moreover, each $f(\partial M)\cap Q_i$ consists of parallel copies of an unknotted $(2n-1)$-disk
in the $(3n-1)$-ball $Q_i$ (since the tubular neighborhood $T_i$ is thin and the intersections of 
its core circle $C_i$ with $f(\partial M)$ are transversal).

Now the double $2M:=M\cup_{\partial M=\partial M} M$ is immersed by $2f$ into 
$S^{3n}=B\cup_{\partial B=\partial B}B$, but still has boundary, which is embedded by $2f$ into 
the annuli $2Q_i\cong D^1\x S^{3n-2}$ which in turn lie in the tori $2T_i\cong S^1\x S^{3n-2}$.
These tori bound the solid tori $2h'_i\cong D^2\x S^{3n-2}$; and to these solid tori there are 
attached the $(3n-1)$-handles $2h_i\cong D^1\x D^{3n-1}$.
The unions $U_i:=2h'_i\cup 2h_i\cong D^1\x D^{3n-1}$ can also be viewed as $(3n-1)$-handles, whose
bases are our annuli $2Q_i\cong D^1\x S^{3n-2}$ containing the image of $\partial(2M)$.
Now the image of $\partial(2M)$ in each $2Q_i$ consists of parallel copies of an embedded 
$(2n-1)$-sphere, which is unknotted in the $(3n-1)$-sphere $\partial U_i$; consequently they
bound parallel $2n$-disks in the $3n$-ball $U_i$.
The desired immersion $F$ is the extension of $2f$ by means of these embedded disks.
Although the interior of each of these disks intersects the image of $2f(M)$, the intersection
corresponds to two new components of $\Delta_F$ which are interchanged by the involution,
so the equivariant map $\Delta_{2f}\to S^0$ yields an equivariant map $\Delta_F\to S^0$.
\end{example}

For a space $X$ let $\Delta^{(3)}_X=\{(x,x,x)\in X^3\}$ and 
$X^{(3)}=\{(x,y,z)\in X^3\mid x\ne y\ne z\ne x\}$, and 
for a map $f\:X\to Y$ let $\Delta^{(3)}_f=\{(x,y,z)\in X^{(3)}\mid f(x)=f(y)=f(z)\}$.

\begin{maintheorem}\label{th4} Let $f\:X\to Y$ be a non-degenerate PL map between compact polyhedra.
The following are equivalent:
\begin{enumerate}
\item $f$ is a PL $1$-prem;
\item there exist an equivariant map $\phi\:\Delta_f\to S^0=\{1,-1\}$ and for each $y\in Y$ 
a total ordering of $f^{-1}(y)$ such that $x<x'$ if and only if $\phi(x,x')=1$;
\item there exists an equivariant map $\phi\:\Delta_f\to S^0$ such that $\phi^3\:(\Delta_f)^3\to (S^0)^3$ 
sends every triple of the form $\big((x,y),(y,z),(z,x)\big)$ into the complement of $\Delta^{(3)}_{S^0}$;
\item $f$ has no Penrose staircases, where a {\rm Penrose staircase} for $f$ is a finite subset 
$F\subset\Delta^{(3)}_f$ such that for every equivariant map $\phi\:\Delta_f\to S^0$ there exists an 
$(x,y,z)\in F$ such that that $\phi(x,y)=\phi(y,z)=\phi(z,x)$;
\item there exist an $S_2$-equivariant map $\Phi\:\Delta_f\to\tilde\R$ and an $S_3$-equivariant%
\footnote{$S_i$ denotes the symmetric group.}
map $\Psi\:\Delta^{(3)}_f\to\R^{(3)}$ that, given any choice of 
an embedding $j\:\{1,2\}\to\{1,2,3\}$, commute with the corresponding projections 
$\pi_j\:\Delta^{(3)}_f\to\Delta_f$ and $\Pi_j\:\R^{(3)}\to\tilde\R$.
\end{enumerate}
\end{maintheorem}

Theorem \ref{th4} elaborates on some results found in the literature, most notably in Po\'enaru's 
paper \cite{Po2}.
The implication (3)$\Rightarrow$(2) is implicit in the proof of \cite{Po2}*{Lemme 1.5}.
The implication (2)$\Rightarrow$(1) for generic smooth immersions between manifolds is implicit 
in \cite{Po2}*{Lemme 1.4} (whose proof is left to the reader); it is proved in \cite{Gi}*{Proposition 4} 
for generic immersions of orientable $2$-manifolds in $\R^3$ and in \cite{Go}*{Theorem 2} in the general case.
See also \cite{Fa} for a related result.

Yet another condition equivalent to those of Theorem \ref{th4} is contained in A. Gorelov's recent paper 
\cite{Go}*{Theorem 4}.
His approach is somewhat different (as compared to conditions (4) and (5)), and can be regarded as 
a natural correction of the approach of \cite{Po2} (whose main result is incorrect, as shown in 
\cite{Go}*{Theorem 5}).

\begin{proof} The implications (1)$\Rightarrow$(2)$\Rightarrow$(3)$\Rightarrow$(4) are trivial.

(2)$\Rightarrow$(1).
Let $K$ and $L$ triangulate $X$ and $Y$ so that $f$ is triangulated by a simplicial map $K\to L$.
For each vertex $v$ of $L$ let us pick a monotone embedding of $f^{-1}(v)$ in $\R$.
This defines a map $g\:X\to\R$ on the vertices of $K$, and we extend it linearly to each simplex of $K$.
If $(x,y)\in\Delta_f$, then $x=\sum_i\lambda_iv_i$ and $y=\sum_i\lambda_iw_i$ for some $\lambda_i>0$ 
such that $\sum_i\lambda_i=1$ and some vertices $v_i$ and $w_i$ of $K$ such that $f(v_i)=f(w_i)$ for each $i$,
and $v_i\ne w_i$ for at least one $i$.
Clearly, each pair $(v_i,w_i)$ lies either in $\Delta_X$ or in the same component of $\Delta_f$ as $(x,y)$.
In the latter case $\phi(v_i,w_i)=\phi(x,y)$ for each $i$, and therefore each $g(v_i)-g(w_i)$ is of the same sign as 
$\phi(x,y)\in S^0=\{1,-1\}$.
Then $g(x)-g(y)=\sum_i\lambda_i\big(g(v_i)-g(w_i)\big)$ is also of that sign, and in particular $g(x)\ne g(y)$.
Hence $f\x g\:X\to Y\x\R$ is injective.

(3)$\Rightarrow$(2).
Given an $y\in Y$, let us define a binary relation $<$ on $f^{-1}(y)$ by $x<x'$ if $x\ne x'$ and $\phi(x,x')=1$.
By (3) $<$ is transitive. 
Hence it is a strict total order.

(4)$\Rightarrow$(3).
Clearly, $\Delta_f$ has only finitely many connected components.
Hence there exist only finitely many distinct equivariant maps $\Delta_f\to S^0$.
Let us denote them by $\phi_1,\dots,\phi_n$.
If (3) does not hold, then each $\phi_i$ sends some triple $\big((x_i,y_i),(y_i,z_i),(z_i,x_i)\big)$ into
$\Delta^{(3)}_{S^0}$.
Then the triples $(x_i,y_i,z_i)$, $i=1,\dots,n$, form a Penrose staircase.

(5)$\Rightarrow$(3).
The natural embedding $e\:X^{(3)}\to(\tilde X)^3$,
$e(x,y,z)=\big((y,z),(z,x),(x,y)\big)$, yields embeddings
$e_f\:\Delta^{(3)}_f\to(\Delta_f)^3$ and $e_\R\:\R^{(3)}\to(\tilde\R)^3$.
Clearly, $\Phi$ and $\Psi$ commute with the projections $\pi_j$ and $\Pi_j$ for each $j$ 
if and only if $\Phi^3$ and $\Psi$ commute with $e_f$ and $e_\R$.
On the other hand, if $h\:\tilde\R\to S^0$ is a homotopy equivalence, then 
$h^3\:(\tilde\R)^3\to(S^0)^3$ sends the image of $e_\R$ onto the complement of 
$\Delta^{(3)}_{S^0}$.

(3)$\Rightarrow$(5).
In the notation of the proof of (5)$\Rightarrow$(3), let us note that $e$ is equivariant 
with respect to the action of $S_3$ on $(\tilde X)^3$ given by 
$\sigma(p_1,p_2,p_3)=(t^\sigma p_{\sigma(1)},t^\sigma p_{\sigma(2)},t^\sigma p_{\sigma(3)})$,
where $t^\sigma\:\tilde X\to\tilde X$ is the factor exchanging involution if $\sigma$ is odd
and the identity if $\sigma$ is even.
Hence the composition $\Delta^{(3)}_f\xr{e_f}(\Delta_f)^3\xr{\phi^3}(S^0)^3\but\Delta^{(3)}_{S^0}$
is equivariant with respect to the free action of $S_3$ on $(S^0)^3\but\Delta^{(3)}_{S^0}$
given by $\sigma(\eps_1,\eps_2,\eps_3)=\big((-1)^\sigma\eps_{\sigma(1)},
(-1)^\sigma\eps_{\sigma(2)},(-1)^\sigma\eps_{\sigma(3)}\big)$.
Also, the composition $\R^{(3)}\xr{e_\R}(\tilde\R)^3\xr{h^3}(S^0)^3\but\Delta^{(3)}_{S^0}$ is
an $S_3$-equivariant homotopy equivalence.
Hence we obtain the desired $S_3$-map $\Psi$.
\end{proof}

\begin{corollary} Let $f\:X^n\to Y^m$ be a $C^0$-stable smooth map between compact smooth manifolds, where $n\le m$.
Then the assertion of Theorem \ref{th4} with ``PL $1$-prem'' replaced by ``smooth $1$-prem'' holds for $f$.
\end{corollary}

\begin{proof} By Verona's theorem (see Theorem \ref{Shiota}) $f$ is $C^0$-left-right equivalent to a PL map $f'$.
Since $n\le m$, both maps are non-degenerate.
By \cite{partI}*{Theorem B} $f$ is a smooth $1$-prem if and only if it is a topological
$1$-prem.
The latter is equivalent to $f'$ being a topological $1$-prem.
Now the assertion follows from Theorem \ref{th4}.
\end{proof}

\begin{problem} Let $f\:X\to Y$ be a finite-to-one map between compacta.
Does the assertion of Theorem \ref{th4} with ``PL $1$-prem'' replaced by ``topological $1$-prem'' 
hold for $f$?
\end{problem}

\section{Generic maps} \label{generic0}

\subsection{Generic smooth maps} \label{generic}
In the present paper we adhere to the following understanding of the term {\it generic}.
It can be used only in assertions of certain types.
First we consider the smooth case.
Let $N$ and $M$ be smooth manifolds, where $N$ is compact.%
\footnote{The case of non-compact $N$ is discussed in Appendix \ref{sec-xx} but is not needed in the main part of
the present paper.}
The assertion ``every generic smooth map $N\to M$ satisfies property $P$'' (or any logically equivalent 
assertion) means ``$C^\infty(N,M)$ contains a dense open subset whose elements satisfy property $P$''.
The assertion ``every generic smooth fold map $N\to M$ satisfies property $P$'' (or any logically equivalent 
assertion) means ``the subspace of fold maps in $C^\infty(N,M)$ contains a dense open subset whose elements 
satisfy property $P$''.
Here ``fold'' can be replaced by any other adjective or adjective phrase, and one can similarly
use adjectives or adjective phrases in the PL case, which will be discussed below.

Let us note that, according to our convention, the word ``generic'' is not 
an adjective, but a more complex part of speech (which modifies not only a noun phrase but the whole sentence, 
or even several sentences).
One can also understand ``generic'' as an adjective, whose meaning e.g.\ in case of ``generic smooth maps'' 
would be ``satisfying a certain fixed set of conditions that determine an open and dense set of maps in 
$C^\infty(N,M)$''.
A drawback of this reading is that the scope of the implicit existential quantifier in ``a certain fixed 
set'' is not clearly specified, which may lead to ambiguity.

\begin{proposition} \label{stability}
Theorem \ref{th1}(b) is equivalent to its modified version where ``generic'' is replaced by ``stable''.
\end{proposition}

\begin{proof} Let $F,S,A\subset C^\infty(N,M)$ consist of all fold maps, of all stable maps and of all corank one
maps, respectively.
Since $S$ is open, $S\cap F$ is open in $F$.
Since $S\cap A$ is dense in $A$ (see Theorem \ref{Morin}) and $F$ is open, $S\cap F$ 
is dense in $F$.
Thus the modified version of Theorem \ref{th1}(b) implies the original version.

Conversely, let $G$ be the dense open subset of $F$ given by the statement of Theorem \ref{th1}(b), 
and let $f\:N\to M$ be a stable smooth fold map.
Since $G$ is dense in $F$, there exists a $g\in G$ that is $C^\infty$-left-right equivalent to $f$.
Then $f$ is a smooth $k$-prem if and only if $g$ is a smooth $k$-prem; also, there exists
an equivariant map $\Delta_f\to S^{k-1}$ if and only if there exists an equivariant map $\Delta_g\to S^{k-1}$.
Hence Theorem \ref{th1}(b) implies its modified version.
\end{proof}

\subsection{Generic PL maps} \label{generic-pl}
Here we assume familiarity with Appendix \ref{pl-maps}.

Let $N$ and $M$ be polyhedra, where $N$ is compact.
The assertion ``every generic PL map $N\to M$ satisfies property $P$'' (or any logically equivalent 
assertion) means ``for each triangulations $K$, $L$ of $N$, $M$ and every stratification map $\phi$ between
their face posets, $C(\phi)$ contains an open dense subset whose elements satisfy property $P$''.

The proof of Theorem \ref{th1} will also utilize the following definition. 
Let $N_0$ be a closed subpolyhedron of $N$.
The assertion ``every PL map $N\to M$, generic on $N\but N_0$, satisfies property $P$'' (or any logically 
equivalent assertion) means ``for each triangulations $(K,K_0)$ of $(N,N_0)$ and $L$ of $M$ and every 
stratification map $\phi$ between the face posets of $K$ and $L$, the space $C(\phi)$ contains a subset 
whose elements satisfy property $P$ and whose intersection with each point-inverse $F$ of the restriction map 
$C(\phi)\to C(\phi_0)$ is dense and open in $F$, where $\phi_0$ denotes the restriction of $\phi$
to the face poset of $K_0$''.

\begin{proposition} \label{stability-pl}
Theorem \ref{th1}(a) is equivalent to its modified version where ``generic'' is replaced by ``stable''.
\end{proposition}

\begin{proof} Theorem \ref{PL stability} implies that generic PL maps $N\to M$ are stable.
(In more detail, let $\phi$ be a stratification map between the face posets of some triangulations
$K$, $L$ of $N$, $M$.
By Theorem \ref{PL stability} $S(\phi)$ is open and dense in $C(\phi)$.
But every element of $S(\phi)$ is a stable PL map, by using the same triangulations.)
Hence the modified version of Theorem \ref{th1}(a) implies the original version.

Conversely, let $f\:N\to M$ be a stable PL map.
Then there exist triangulations $K$, $L$ of $P$, $Q$ and a stratification map $\phi$ between the face posets 
of $K$ and $L$ such that $f\in S(\phi)$.
Assuming the assertion of Theorem \ref{th1}(a), we get an open dense subset $G$ of $C(\phi)$ 
such that if $g\in G$ and there exists an equivariant map $\Delta_g\to S^{k-1}$, then $g$ is a PL $k$-prem.
Since $G$ is dense and $f$ is PL-left-right equivalent to some $f'\in S(\phi)$, some $g\in G$ is 
PL-left-right equivalent to $f'$ and hence to $f$.
Then $f$ is a PL $k$-prem if and only if $g$ is a PL $k$-prem; also, there exists an equivariant map 
$\Delta_f\to S^{k-1}$ if and only if there exists an equivariant map $\Delta_g\to S^{k-1}$.
Hence Theorem \ref{th1}(a) implies its modified version.
\end{proof}

\section{Proof of Theorem \ref{th1} modulo Main Lemma} \label{thm-mod-lemma}

\subsection{Proof of Theorem 1 modulo Main Theorem} \label{1-mod-main}
The following theorem will be proved later in this section modulo a certain Main Lemma.

\begin{theorem}[Main Theorem]\label{mainthm} Suppose that $m\ge n$ and $2(m+k)\ge 3(n+1)$.

\smallskip
(a) Let $M$ be a PL $m$-manifold and $N$ a compact $n$-polyhedron, and let $M_\star$, $N_\star$ be their 
closed subpolyhedra.
Let $f\:N\to M$ be a PL map such that $f^{-1}(M_\star)=N_\star$ and $f$ is generic on $N\but N_\star$, and 
let $f_\star=f|_{N_\star}$.

Suppose that $e_\star\:N_\star\to\R^k$ is a PL map such that $f_\star\x e_\star\:N_\star\to M_\star\x\R^k$ is 
an embedding and $\tl e_\star\:\Delta_{f_\star}\to S^{k-1}$ extends to an equivariant map
$\alpha\:\Delta_f\to S^{k-1}$.

Then $e_\star$ extends to a PL map $e\:N\to\R^k$ such that $f\x e\:N\to M\x\R^k$ is 
an embedding and $\tl e\:\Delta_f\to S^{k-1}$ is equivariantly homotopic to $\alpha$ 
keeping $\Delta_{f_\star}$ fixed.
\smallskip

(b) Let $M$ be a smooth $m$-manifold and $N$ a compact smooth $n$-manifold, and let $M_\star$ and
$N_\star$ be codimension zero closed submanifolds of $M$ and $N$ respectively.
Let $f\:N\to M$ be a generic smooth fold map such that $f^{-1}(M_\star)=N_\star$, and let 
$f_\star=f|_{N_\star}$.

Suppose that $e_\star\:N_\star\to\R^k$ is a map such that $f_\star\x e_\star\:N_\star\to M_\star\x\R^k$ 
is a smooth embedding and $\tl e_\star\:\Delta_{f_\star}\to S^{k-1}$ extends to an equivariant map
$\alpha\:\Delta_f\to S^{k-1}$.

Then $e_\star$ extends to a map $e\:N\to\R^k$ such that $f\x e\:N\to M\x\R^k$ is a smooth embedding 
and $\tl e\:\Delta_f\to S^{k-1}$ is equivariantly homotopic to $\alpha$ keeping $\Delta_{f_\star}$ fixed.
\end{theorem}

Theorem \ref{th1}(a) is an immediate consequence of the case $M_\star=\emptyset$ of Theorem \ref{mainthm}(a).
The case $M_\star\ne\emptyset$ will arise from the induction step in the proof of the case $M_\star=\emptyset$. 

Theorem \ref{th1}(b) is an immediate consequence of the case $M_\star=\emptyset$ of Theorem \ref{mainthm}(b).
The case $M_\star\ne\emptyset$ is needed to deduce Theorem \ref{th1}(c).
Before we go into this deduction, let us introduce some notation.

\subsection{Immersions and fold maps}\label{fold maps}
For a space $N$ let $\Delta^{(3)}_N=\{(x,x,x)\in N^3\}$ and for a map $f\:N\to M$ let 
\[\Delta^{(3)}_f=\{(x,y,z)\in N^3\mid f(x)=f(y)=f(z),\,x\ne y,\,y\ne z,\,z\ne x\}.\]
Let $\bar\Delta_f$ be the closure of $\Delta_f$ in $N^2$ and let $\bar\Delta^{(3)}_f$ be the closure 
of $\Delta^{(3)}_f$ in $N^3$.
Let $\Sigma_f$ be the set of points $p\in N$ such that every neighborhood of $p$ contains distinct points 
$x$, $y$ satisfying $f(x)=f(y)$.
Let $\Sigma^{(3)}_f$ be the set of points $p\in N$ such that every neighborhood of $p$ contains 
pairwise distinct points $x$, $y$, $z$ satisfying $f(x)=f(y)=f(z)$.
Thus $\Delta_{\Sigma_f}=\bar\Delta_f\cap\Delta_N$ and 
$\Delta^{(3)}_{\Sigma^{(3)}_f}=\bar\Delta^{(3)}_f\cap\Delta^{(3)}_N$. 

A map $f\:N\to M$ is called a (smooth) {\it immersion} if every point of $N$ has a neighborhood that 
is (smoothly) embedded by $f$.
Clearly, $f$ is an immersion if and only if $\Sigma_f=\emptyset$.
By a {\it fold map} we mean a map $f\:N\to M$ such that $\Sigma_f^{(3)}=\emptyset$;
this includes {\it smooth fold maps}, whose only singularities are of the fold type
(= of type $\Sigma^{1,0}$, cf.\ \cite{GG}).

\begin{proof}[Proof of Theorem \ref{th1}(c)]
This proof essentially uses Appendix \ref{sec-xx}.
Let $X=f^{-1}\Big(f\big(\Sigma_f^{(3)}\big)\Big)$.
Since $f$ is generic, we may assume that $X$ is a compact polyhedron (see Theorem \ref{Shiota}) 
and that $\dim X=\dim\Sigma_f^{(3)}\le 3n-2m-2$.
Let $g\:N\to\R^k$ be a generic smooth map and let $h=f\x g\:N\to M\x\R^k$.
Since $k\ge\dim X+1$, by Theorem \ref{taking off} we may assume that $\Sigma_h$ is disjoint 
from $X$.
Thus $h$ embeds some closed neighborhood $N_\star$ of $X$ in $N$ such that $N_\star$ is a manifold with 
boundary, $N_\star$ deformation retracts onto $X$, and $N_\star=f^{-1}(M_\star)$ for some 
codimension zero submanifold $M_\star\subset M$.
Let $f_\star=f|_{N_\star}$ and $e_\star=g|_{N_\star}$.
We may further assume that $\Delta_{f_\star}$ equivariantly deformation retracts onto $Y:=\Delta_{f|_X}$.

We have $\dim Y\le\dim X\le 3n-2m-2$. 
Since $3n-2m-2<k-1$, the map $\tl g|_Y\:Y\to S^{k-1}$ is equivariantly homotopic 
to $\phi|_Y$, where $\phi\:\Delta_f\to S^{k-1}$ is the given equivariant map.
It follows that $\tl e_\star\:\Delta_{f_\star}\to S^{k-1}$ is equivariantly homotopic to 
$\phi|_{\Delta_{f_\star}}$, and consequently extends to an equivariant map
$\alpha\:\Delta_f\to S^{k-1}$ which is homotopic to $\phi$.
Let $N'$ be a codimension zero closed submanifold of $N$ containing $N\but N_\star$ and disjoint from $X$.
(If $N_\star\cap\partial N=\emptyset$, we may take $N'$ to be simply the closure of $N\but N_\star$;
otherwise $N'$ has to be slightly larger in order for $\partial N'$ to be a smooth manifold.)
Since $f$ is a fold map on $N'$, we can now apply Theorem \ref{mainthm}(b) to obtain an extension of 
$e_\star$ to a map $e\:N\to\R^k$ such that $f\x e\:N\to M\x\R^k$ is a smooth embedding (and, in fact, 
$\tl e$ is equivariantly homotopic to $\phi$). 
\end{proof}

\subsection{Simple fold maps}\label{simple fold maps}
By a {\it (smooth) simple fold map} we mean a (smooth) fold map $f\:N\to M$ that does not intersect its own folds, 
that is, $(N\x\Sigma_f)\cap\Delta_f=\emptyset$.
Clearly, (smooth) immersions are (smooth) simple fold maps, and generic PL (smooth) triple point free maps are 
(smooth) simple fold maps.
Generic PL (smooth) maps $N^n\to M^m$
\begin{itemize}
\item have no triple points when $2m\ge 3n+1$; 
\item are simple (smooth) fold maps when $2m\ge 3n$;
\item are (smooth) fold maps when $2m\ge 3n-1$, provided that $M$ and $N$ are PL manifolds rather than polyhedra 
in the PL case.
\end{itemize}

Let $E_f$ be the set of distinct pairs $(p,q)\in N\x N$ such that every neighborhood of $(p,q)$ contains pairs 
$(x,y)$ and $(x,z)$ such that $x$, $y$, $z$ are pairwise distinct and $f(x)=f(y)=f(z)$.
Let $\bar E_f$ be the closure of $E_f$ in $N^2$.
Thus $\bar\Delta^{(3)}_f\but\Delta^{(3)}_f$ is the union of the images of $\bar E_f$ under the three embeddings 
of $N\x N$ onto the thick diagonals of $N^3$.
Clearly, $\bar E_f\but E_f=\Delta_{\Sigma_f^{(3)}}$ and $\bar E_f\subset N\x\Sigma_f$. 
It is easy to see that if $f$ is a generic (PL or smooth) map, then $E_f=(N\x\Sigma_f)\cap\Delta_f$.
Thus a generic PL (smooth) map is a (smooth) simple fold map if and only if $\bar E_f=\emptyset$.

\begin{remark} Smooth simple fold maps have been studied by A. Sz\H ucs.
He calls them ``maps of singular multiplicity 2'' in \cite{Sz2} and ``$\Sigma^{1_1}$-maps'' 
(not to be confused with ``$\Sigma^1$-maps'') in \cite{Sz4} (but ``$\Sigma^{1_1}$-maps''
have a different meaning in a paper by R. Rim\'anyi and A. Sz\H ucs published in the same year 
and in later papers by A. Sz\H ucs). 
Also they would be called ``simple $\Sigma^1$-singular maps'' in direct analogy with 
the terminology of \cite{Sz3}.
\end{remark}

\begin{lemma}[Main Lemma]\label{mainlemma}
(a) Theorem \ref{mainthm}(a) holds if $f|_{N\but N_\star}$ is a simple fold map.

(b) Theorem \ref{mainthm}(b) holds if $f$ is a smooth simple fold map.
\end{lemma}

Lemma \ref{mainlemma}(a) is proved in \S\ref{sec2} and Lemma \ref{mainlemma}(b) is proved in \S\ref{sec3}
by building on the proof of Lemma \ref{mainlemma}(a).
A sketch of the proof, which applies to both PL and smooth cases, is given in \S\ref{sketch}.

\begin{remark}
The full strength of Lemma \ref{mainlemma} will be needed to prove Theorem \ref{mainthm}(b).
To prove Theorem \ref{mainthm}(a), we need Lemma \ref{mainlemma}(a) only in the cases where $f|_{N\but N_\star}$ is 
either (i) a triple-point free map or (ii) a trivial covering.
It is not hard to prove the case (i) directly, which we will do (see Proposition \ref{3pt-free}) since it helps
to understand the general position assumptions that are really needed for Theorem \ref{mainthm}(a)
(see Remark \ref{gen-pos}).
However, the case (ii) is not so far from the general case (of simple fold maps), at least from 
the conceptual viewpoint.
From the technical viewpoint, the verification of ``being not so far'' does take some efforts (which are not 
in vain since they are anyway needed for the smooth case).
One reason to be careful about this verification is that $f$ may have complicated behavior near $N_\star$ 
(since it is assumed to be a simple fold map only over $N\but N_\star$), but that is also an issue in the case (ii) 
(since in that case $f$ is assumed to be a covering only over $N\but N_\star$) and so in the case (ii) one would 
anyway need to be careful about some verification of this sort.
\end{remark}

\subsection{A lemma on isovariant maps}
A triangulation of a polyhedron $P$ with a PL action of a finite group $G$ is called {\it equivariant} if 
the action is by simplicial maps and the stabilizer of each simplex equals the stabilizer of each interior point 
of this simplex; this implies, in particular, that the fixed point set $P^G$ is triangulated by a subcomplex.

An {\it isovariant} map is an equivariant map $f\:X\to Y$ which does not increase stabilizers of points, i.e.\
the stabilizer of each $x\in X$ contains the stabilizer of $f(x)$.
(Let us note that the reverse inclusion also holds since $f$ is equivariant.)
In particular, an equivariant map $\phi$ of a $\Z/2$-space $X$ into $\R^k$ with $\Z/2$ acting by the central symmetry
$x\mapsto -x$ is isovariant if and only if $\phi^{-1}(0)$ coincides with the fixed point set $X^{\Z/2}$.

\begin{lemma} \label{isovariant PL approximation}
Let $P$ be a compact $\Z/2$-polyhedron and let $Q$ be an invariant closed subpolyhedron of $P$.
Let $\beta\:P\to\R^k$ be an isovariant map such that $\beta|_Q$ is a PL map.
Then $\beta$ is isovariantly homotopic keeping $Q$ fixed to a PL map.
\end{lemma}

\begin{proof}
Let $A$ be an equivariant triangulation of $P$ such that $Q$ is triangulated by a subcomplex of $A$.
Let $B$ be the union of all closed simplexes of $A$ that do not intersect $P^{\Z/2}$.
Let $B'$ be an equivariant subdivision of $B$ into sufficiently small simplexes, and let us 
extend it to an equivariant subdivision $A'$ of $A$ without adding new vertices.
Let $C$ be the simplicial neighborhood of $P^{\Z/2}$ in $A'$.
Let us define $\phi$ to coincide with $\beta$ on the vertices of $A'$ and by extending
linearly to the simplexes of $A'$.
Then it is easy to see that $\phi$ is isovariant.
(Every simplex of $C$ is a join of a simplex $S$ of $P^{\Z/2}$ and a simplex $T$ of $B'$.
Then $\phi(S)=0$ and $\phi(T)\subset\R^k\but 0$ by our choice of $B'$. 
Hence $\phi^{-1}(0)\cap (S*T)=S$.)
Also, $\phi$ is ``conical'' on $C$ in the sense that it is the composition of the quotient map 
$C\to C/P^{\Z/2}$ and a conical map of the cone $C/P^{\Z/2}$ over $\partial C$ into the infinite cone $\R^k$.

Now $\beta|_B$ is isovariantly homotopic to $\phi|_B$ by the linear homotopy $(1-t)\beta|_B+t\phi|_B$, which keeps $Q$ fixed.
Therefore $\beta$ is isovariantly homotopic keeping $Q$ fixed to a map $\gamma$ such that $\gamma|_B=\phi|_B$.
Finally, similarly to the Alexander trick, $\gamma$ is isovariantly homotopic to $\phi$ by a homotopy $\gamma_t$,
$t\in [0,1]$, keeping $B\cup Q$ fixed and such that each $\gamma_t$ coincides with $\gamma$ 
on $C\but tC$ for a certain smaller neighborhood $tC$ of $P^{\Z/2}$, and is ``conical'' on $tC$.
\end{proof}

\subsection{The triple point free case}
In the case where $f|_{N\but N_\star}$ is a simple fold map {\it without triple points}, Theorem \ref{mainthm}(a) 
can be proved by a much simpler construction, which works without any hypotheses on dimensions and genericity:

\begin{proposition} \label{3pt-free}
Let $M^m$ be a PL manifold and $N^n$ a compact polyhedron, and let $M_\star$, $N_\star$ be closed subpolyhedra
of $M$ and $N$ respectively.
Let $f\:N\to M$ be a PL map such that $f^{-1}(M_\star)=N_\star$ and $f|_{N\but N_\star}$ is a simple 
fold map without triple points, and let $f_\star=f|_{N_\star}$.

Suppose that $e_\star\:N_\star\to\R^k$ is a PL map such that $f_\star\x e_\star\:N_\star\to M_\star\x\R^k$ is 
an embedding and $\tl e_\star\:\Delta_{f_\star}\to S^{k-1}$ extends to an equivariant map
$\alpha\:\Delta_f\to S^{k-1}$.

Then $e_\star$ extends to a PL map $e\:N\to\R^k$ such that $f\x e\:N\to M\x\R^k$ is 
an embedding and $\tl e\:\Delta_f\to S^{k-1}$ is equivariantly homotopic to $\alpha$ 
keeping $\Delta_{f_\star}$ fixed.
\end{proposition}

Let us emphasize that $f$ is not assumed to be generic here, and $m$, $n$ and $k$ are arbitrary.
Let us note that a generic PL map without triple points is automatically a simple fold map.

\begin{proof} We are given $e_\star$ and $\alpha$ and we need to construct $e$.
Let $f_\circ=f|_{N\but N_\star}$.
Let us note that $\Delta_{f_\star}=\Delta_f\but\Delta_{f_\circ}$, but 
$\bar\Delta_f^\star:=\bar\Delta_f\but\bar\Delta_{f_\circ}$ may be larger than $\bar\Delta_{f_\star}$. 
Clearly, $\bar\Delta_f^\star=\bar\Delta_f\cap (N_\star\x N_\star)$.

Let $S_\star\:\bar\Delta_f^\star\to\R^k$ be defined by $S_\star(x,y)=\frac{e_\star(x)+e_\star(y)}2$, and let
$S\:\bar\Delta_f\to\R^k$ be an arbitrary PL map extending $S_\star$ and satisfying $S(x,y)=S(y,x)$.
Let $A_\star\:\bar\Delta_f^\star\to\R^k$ and $a_\star\:\bar\Delta_f^\star\to[0,\infty)$ be defined by 
$A_\star(x,y)=\frac{e_\star(y)-e_\star(x)}2$ and $a_\star(x,y)=||A_\star(x,y)||$, and let
$a\:\bar\Delta_f\to[0,\infty)$ be an arbitrary PL map extending $a_\star$ and satisfying $a(x,y)=a(y,x)$ and
$a(x,y)=0\Leftrightarrow x=y$.
Let us define $A\:\bar\Delta_f\to\R^k$ by $A(x,y)=\alpha(x,y)\cdot a(x,y)$ for $x\ne y$ and by $A(x,x)=0$.
Then $A$ is isovariant and 
$A(x,y)=\frac{e_\star(y)-e_\star(x)}{||e_\star(y)-e_\star(x)||}\cdot\frac{||e_\star(y)-e_\star(x)||}2=A_\star(x,y)$ 
for all $(x,y)\in\Delta_{f_\star}$.
Since $A(x,x)=0$, we have $A|_{\bar\Delta_f^\star}=A_\star$.
Since $A_\star$ is a PL map, by Lemma \ref{isovariant PL approximation} $A$ is isovariantly homotopic keeping 
$\bar\Delta_f^\star$ fixed to a PL map $A'\:\bar\Delta_f\to\R^k$.

Let us define $\phi\:\bar\Delta_f\to\R^k$ by $\phi(x,y)=S(x,y)+A'(x,y)$.
Then $\phi$ is a PL map and $\phi(x,y)=S_\star(x,y)+A_\star(x,y)=e_\star(y)$ for $(x,y)\in\bar\Delta_f^\star$.
Let $\pi\:N\x N\to N$ be the projection onto the second factor.
Since $f_\circ$ is a simple fold map without triple points, $\pi$ embeds $\bar\Delta_{f_\circ}$. 
Let us define $e_1\:N_\star\cup\pi(\bar\Delta_f)\to\R^k$ by $e_1(y)=e_\star(y)$ for $y\in N_\star$ and by
$e_1(y)=\phi(x,y)$ for $(x,y)\in\bar\Delta_f$.
Let us note that the latter definition agrees with the former when $(x,y)\in\bar\Delta_f^\star$,
and it follows that $e_1$ is well-defined.
Let $e\:N\to\R^k$ be any PL map extending $e_1$.
Since $f\x e\:N\to M\x\R^k$ embeds $N_\star$ and $\pi(\Delta_{f_\circ})$, it is an embedding.
We have \[\tl e(x,y)=\frac{e(y)-e(x)}{||e(y)-e(x)||}=\frac{\phi(x,y)-\phi(y,x)}{||\phi(x,y)-\phi(y,x)||}
=\frac{A'(x,y)-A'(y,x)}{||A'(x,y)-A'(y,x)||}\]
and
\[\frac{A(x,y)-A(y,x)}{||A(x,y)-A(y,x)||}=\frac{\alpha(x,y)-\alpha(y,x)}{||\alpha(x,y)-\alpha(y,x)||}=
\frac{2\alpha(x,y)}{||2\alpha(x,y)||}=\alpha(x,y)\]
for all $(x,y)\in\Delta_f$.
Since $A$ is isovariantly homotopic to $A'$ keeping $\bar\Delta_f^\star$ fixed, it follows that
$\tl e$ is equivariantly homotopic to $\alpha$ keeping $\Delta_{f_\star}$ fixed.
\end{proof}

\subsection{Proof of Main Theorem modulo Main Lemma: PL case}

\begin{proof}[Proof of Theorem \ref{mainthm}(a)] Let $K$ and $L$ be triangulations of $N$ and $M$ such that 
$f\:K\to L$ is simplicial and $N_\star$ and $M_\star$ are triangulated by subcomplexes.
Let $\bar D_3$ be the image of the composition $\bar\Delta^{(3)}_f\subset N^3\to N$ of the inclusion and 
the projection onto the first factor.
Let $E$ be the image of the composition $(\Sigma_f\x N)\cap\Delta_f\subset N^2\to N$ of the inclusion and
the projection onto the first factor.
Since $f$ is generic on $N\but N_\star$, we may assume that $f|_{N\but N_\star}$ is non-degenerate, that
$\bar D_3\cap (N\but N_\star)$ is of dimension $\le 3n-2m$ and that
$E\cap (N\but N_\star)$ is of dimension $\le 3n-2m-1$ (in fact, $\le 3n-2m$ would already suffice for our 
purposes).

Let $\sigma_1,\dots,\sigma_r$ be the simplexes of $f(\bar D_3\cup E)$ not contained in $M_\star$, arranged 
in an order of increasing dimension.
Let $L_i=M_\star\cup\sigma_1\cup\dots\cup\sigma_i$, and let $K_i=f^{-1}(L_i)$.
Let $f_i\:K_i\to L_i$ be the restriction of $f$. 
Thus $f_0=f_\star$.

Let $e_0=e_\star$.
Suppose more generally that $e_i\:K_i\to\R^k$ is a PL map such that $f_i\x e_i\:K_i\to L_i\x\R^k$ is 
an embedding and $\tl e_i\:\Delta_{f_i}\to S^{k-1}$ is equivariantly homotopic to 
$\alpha|_{\Delta_{f_i}}$ by a homotopy $h_i\:\Delta_{f_i}\x I\to S^{k-1}$ that keeps 
$\Delta_{f_\star}$ fixed.

Let $(Q,Q_\star)=f^{-1}(\sigma_{i+1},\partial\sigma_{i+1})$ and let $(F,F_\star)=f|_{(Q,Q_\star)}$ and 
$E_\star=e_i|_{Q_\star}$.
Since $f|_{N\but N_\star}$ is non-degenerate, $f$ sends every simplex of $K$ not contained in $N_\star$ 
homeomorphically onto some simplex of $L$.
Hence the preimage of every open $k$-simplex of $L$ not contained in $M_\star$ is 
a disjoint union of open $k$-simplexes of $K$.
In particular, $F$ restricts to a trivial finite-fold (possibly 0-fold) covering
$Q\but Q_\star=K_{i+1}\but K_i\to\sigma_{i+1}\but\partial\sigma_{i+1}=L_{i+1}\but L_i$.
The map $F_\star\x E_\star\:Q_\star\to\partial\sigma_{i+1}\x\R^k$ is an embedding, and
$\tl E_\star\:\Delta_{F_\star}\to S^{k-1}$ extends to an equivariant map $\Phi\:\Delta_F\to S^{k-1}$
given by $\alpha|_{\Delta_F}$ along with $H_\star:=h_i|_{\Delta_{F_\star}\x I}$, using that the inclusion
$\Delta_{F_\star}\subset\Delta_F$ is a cofibration.%
\footnote{If $F\:Q\to X$ is a PL map, and $R$ is a closed subpolyhedron of $Q$, then $\tl R$ is
a closed subpolyhedron of $\tl Q$, and consequently $\Delta_{F|_R}=(F\x F)^{-1}(\Delta_X)\cap\tl R$ 
is a closed subpolyhedron of $\Delta_F=(F\x F)^{-1}(\Delta_X)\cap\tl Q$.}
Since every covering is an immersion, and hence a simple fold map, and since every $s$ such that $s\le 3n-2m$ 
satisfies $s+k\ge\frac32(s+1)$, by Lemma \ref{mainlemma}(a) $E_\star$ extends to a PL map $E\:Q\to\R^k$ such that 
$F\x E\:Q\to\sigma_{i+1}\x\R^k$ is an embedding and $\tl E\:\Delta_F\to S^{k-1}$ is 
equivariantly homotopic to $\Phi$ keeping $\Delta_{F_\star}$ fixed.
Then $\tl E$ is equivariantly homotopic to $\alpha|_{\Delta_F}$ by a homotopy $H$
that extends $H_\star$.

Clearly, $\Delta_{f_{i+1}}=\Delta_{f_i}\cup_{\Delta_{F_\star}}\Delta_F$.
Thus $e_{i+1}:=e_i\cup_{E_\star} E$ is a PL map $K_{i+1}\to\R^k$ such that 
$f_{i+1}\x e_{i+1}\:K_{i+1}\to L_{i+1}\x\R^k$ 
is an embedding and $\tl e_{i+1}\:\Delta_{f_{i+1}}\to S^{k-1}$ is equivariantly homotopic to 
$\alpha|_{\Delta_{f_{i+1}}}$ by the homotopy $h_{i+1}:=h_i\cup_{H_\star}H$ which obviously keeps 
$\Delta_{f_\star}$ fixed. 

Finally, we have $L_r=M_\star\cup f(\bar D_3\cup E)$ and consequently $K_r\supset N_\star\cup\bar D_3\cup E$.
Hence $f|_{N\but K_r}$ is a simple fold map without triple points, and we can apply Proposition \ref{3pt-free}.
\end{proof}

\begin{remark} \label{gen-pos}
It is clear from the proof of Theorem \ref{mainthm}(a) that the hypothesis ``$f$ is generic''
in Theorem \ref{th1}(a) can be replaced by the conjunction of the following conditions:
\begin{enumerate}
\item $f$ is non-degenerate;
\item $\dim\bar\Delta^{(3)}_f\le 3n-2m$;
\item $\dim(\Sigma_f\x N)\cap\Delta_f\le 3n-2m-1$.
\end{enumerate}
The same conditions but with $f|_{N\but N_\star}$ in place of $f$ work to replace the hypothesis 
``$f$ is generic on $N\but N_\star$'' in Theorem \ref{mainthm}(a).
\end{remark}

\begin{remark} The proof of Theorem \ref{mainthm}(a) modulo Lemma \ref{mainlemma}(a) almost looks like
an induction on Thom--Boardman (or at least Morin) strata (which gets stuck at the penultimate step).
Nevertheless, it remains a challenge to make such an induction work in the smooth case.
\end{remark}

\subsection{Proof of Main Theorem modulo Main Lemma: Smooth case}
In order to prove the smooth case of the Main Theorem, we will actually need a slightly strengthened 
version of the smooth case of the Main Lemma.
It is also convenient to state the analogous strengthening of the PL case.

\begin{addendum}[to Main Lemma] \label{mainlemma'}  
(a) Let $M$, $N$, $f$, $M_\star$, $N_\star$ and $f_\star$ be as in Theorem \ref{mainthm}(a).
Let $X$ be the union of $\Delta_{f_\star}$ and $(\Sigma_f\x N\cup N\x\Sigma_f)\cap\Delta_f$. 
Suppose that $g\:N\to\R^k$ is a PL map, generic on $N\but N_\star$ and such that 
$f_\star\x (g|_{N_\star})\:N_\star\to M_\star\x\R^k$ is an embedding, $\Delta_{f\x g}$ is disjoint from $X$, 
and the map $\tl g\:X\to S^{k-1}$ defined by $\tl g(x,y)=\frac{g(y)-g(x)}{||g(y)-g(x)||}$ extends to 
an equivariant map $\alpha\:\Delta_f\to S^{k-1}$.
Then $g$ is homotopic keeping $N_\star$ fixed to a PL map $e\:N\to\R^k$ such that $f\x e\:N\to M\x\R^k$ is 
an embedding and $\tl e\:\Delta_f\to S^{k-1}$ is equivariantly homotopic to $\alpha$ keeping 
$\Delta_{f_\star}$ fixed.
\smallskip

(b) Let $M$, $N$, $f$, $M_\star$, $N_\star$ and $f_\star$ be as in Theorem \ref{mainthm}(b).
Let $X$ be the union of $\Delta_{f_\star}$ and $(\Sigma_f\x N\cup N\x\Sigma_f)\cap\Delta_f$. 
Suppose that $g\:N\to\R^k$ is a generic smooth map such that 
$f_\star\x (g|_{N_\star})\:N_\star\to M_\star\x\R^k$ is a smooth embedding,
$\Delta_{f\x g}$ is disjoint from $X$, and the map $\tl g\:X\to S^{k-1}$ defined by 
$\tl g(x,y)=\frac{g(y)-g(x)}{||g(y)-g(x)||}$ extends to an equivariant map $\alpha\:\Delta_f\to S^{k-1}$.
Then $g$ is homotopic keeping $N_\star$ fixed to a map $e\:N\to\R^k$ such that $f\x e\:N\to M\x\R^k$ is 
a smooth embedding and $\tl e\:\Delta_f\to S^{k-1}$ is equivariantly homotopic to $\alpha$ 
keeping $\Delta_{f_\star}$ fixed.
\end{addendum}

Let us note that Lemma \ref{mainlemma} is a trivial consequence of Addendum \ref{mainlemma'}:
if $f$ is a simple fold map, then $(\Sigma_f\x N)\cap\Delta_f=\emptyset$, so $X=\Delta_{f_\star}$.

\begin{proof}[Proof of Theorem \ref{mainthm}(b)]
Since $f$ is a fold map, $S:=(\Sigma_f\x N)\cap\Delta_f$ is a compact subset of $N\x N\but\Delta_N$.
Let $\pi$ be the projection of $N\x N$ to the first factor and $t$ be the factor exchanging involution
of $N\x N$.
Let $N_0=\pi\big(S\cup t(S)\big)\cup N_\star$ and $f_0=f|_{N_0}$.

Since $f$ is generic, $\dim S\le 3m-2n$ (in fact, $\dim S\le 3m-2n-1$, but we do not need this).
Then it follows from Theorem \ref{mainthm}(a) and the triangulation theorem (see Theorem \ref{Shiota}) that 
the given smooth map $e_\star\:N_\star\to\R^k$ extends to a continuous map $e_0\:N_0\to\R^k$ such that 
$f_0\x e_0\:N_0\to M\x\R^k$ is a topological embedding and $\tl e_0\:\Delta_{f_0}\to S^{k-1}$ is 
equivariantly homotopic to $\alpha|_{\Delta_{f_0}}$ keeping $\Delta_{f_\star}$ fixed.

Let $g\:N\to\R^k$ be a generic smooth map such that $g|_{N_\star}=e_\star$ and $g|_{N_0}$ 
is sufficiently $C^0$-close to $e_0$, and let $h=f\x g\:N\to M\x\R^k$.
Since $e_0$ is injective and $S$ is a compact subset of $N\x N\but\Delta_N$, we may assume that 
$\bar\Delta_h$ is disjoint from $X:=S\cup t(S)\cup\Delta_{f_\star}$ and that $\tl g\:X\to S^{k-1}$ 
is equivariantly homotopic to $\alpha|_X$ keeping $\Delta_{f_\star}$ fixed.
Then $\tl g$ extends to an equivariant map $\alpha'\:\Delta_f\to S^{k-1}$ which is equivariantly 
homotopic to $\alpha$ keeping $\Delta_{f_\star}$ fixed.

Then by Addendum \ref{mainlemma'}(b) $g$ is homotopic keeping $N_\star$ fixed to a map $e\:N\to\R^k$ 
such that $f\x e\:N\to M\x\R^k$ is a smooth embedding and $\tl e\:\Delta_f\to S^{k-1}$ is equivariantly 
homotopic to $\alpha$ keeping $\Delta_{f_\star}$ fixed.
\end{proof}

\subsection{Proof of Main Lemma: Sketch} \label{sketch}
Here is an outline of the proof of Lemma \ref{mainlemma}.
In this outline, we will assume that $M_\star=\emptyset$ in order to simplify notation.

Let us pick a generic lift $h=f\x g\:N\to M\x\R^k$ of the given map $f\:N\to M$.
Let us define $\bar g\:\Delta_f\to\R^k$ by $(x,y)\mapsto g(y)-g(x)$.
Using the hypothesis, it is not hard to construct a generic equivariant homotopy $\Phi$ from
$\bar g$ to a map into $\R^k\but\{0\}$.
Let $D_\ccirc$ be the image of $\Phi^{-1}(0)\subset\Delta_f\x I$ under the projection $\Delta_f\x I\to\Delta_f$ 
and let $D$ be the closure of $D_\ccirc$ in $\bar\Delta_f$.
Then $D$ is of a small dimension (namely, $2n+1-m-k$, which is less than $n/2$) and therefore the projection
$\pi\:\bar\Delta_f\subset N\x N\to N$ embeds $D$.
Also, since $f$ is a simple fold map, $\pi$ is an immersion, and hence embeds a closed
invariant neighborhood $U$ of $D$.
Our plan is to homotope $h$ to an embedding by a homotopy lying over $f$, with support in $\pi(U)$,
using the configuration level homotopy $\Phi$ (restricted over $U$) and the homeomorphism $\pi|_U$ between 
$U$ and $\pi(U)$.
A preliminary version $\xi_t=f\x e_t$ of the desired homotopy is given by an explicit formula involving 
$\xi_0=h|_{\pi(U)}$ and $\Phi$, and does eliminate the existing double points of $h$; 
thus $\xi_1$ embeds $\pi(U)$.
But, of course, this formula is unaware that $f$ may have triple and 10-tuple points, and 
because of this $\xi_t$ may accidentally create new double points between $\pi(U)$ and $N\but\pi(U)$.
However, since $D$ is of small dimension, by a slight perturbation of $\xi_t$ we can at least 
ensure that new double points do not occur between $\pi(D)$ and $N\but\pi(D)$.
Then they also do not occur between a small neighborhood $\pi(V^\sharp)$ of $\pi(D)$
in $\pi(U)$ and $N\but\pi(V^\sharp)$, and (using again that $f$ is a simple fold map) 
even between a neighborhood $W^\sharp$ of $\pi(D)$ in $N$ and $N\but W^\sharp$.
With enough care this remains true (even with the same $V^\sharp$ and $W^\sharp$, so that we don't get 
a circularity!) after we amend $\Phi$ so that it has support in $V^\sharp$.

\begin{remark} One naturally wonders if it is not excessive to first define a preliminary version $\xi_t$ of 
the homotopy and then amend%
\footnote{Here we are concerned only with amendments that are global, at least on some scale;
perturbations into general position do not count as ``amendments'' here.}
it (along with $\Phi$): isn't it possible to define $\xi_t$ just once, after an appropriate preparation?
The amendment is based on the fact that the images of $\pi(D)$ (of dimension $<n/2$) and $N\but\pi(D)$ 
(of dimension $n$) in $M\x\R^k$ (of dimension $m+k\ge 3(n+1)/2$) do not intersect in general position
(while lying over $f$). 
If they happen to remain disjoint in general position even after projecting to $M$, then we could indeed define 
$\xi_t$ just once; but this assumption (namely, $m\ge 3n/2$) would be too restrictive.

We could actually do as follows: first define $\xi_t$ on $\pi(D)$ (by explicit formula), then perturb it into
general position (along with $\Phi$), and finally extend it over $\pi(U)$ (by explicit formula).
While this seems to be the most logical order of actions, it would require analyzing properties of 
the construction by explicit formula two times (separately for $\pi(D)$ and for $\pi(U)$), even though 
the formula is the same.
So it is really in order to avoid this repeated analysis that we define $\xi_t$ twice on the entire $\pi(U)$. 
\end{remark}

\section{Proof of Main Lemma: PL case} \label{sec2}

Here is the proof of Lemma \ref{mainlemma}(a), split into a number of steps.

\subsection{Construction of $D$}
Let us write $N_\circ=N\but N_\star$, $M_\circ=M\but M_\star$ and $f_\circ=f|_{N_\circ}$.
Since $f$ is generic on $N_\circ$, we may assume that $\dim\Delta_{f_\circ}\le 2n-m$.
Let $d=2n-m-k+1$.
Let $g\:N\to\R^k$ be a PL map extending $e_\star$ and generic on $N_\circ$, and let
$h=f\x g\:N\to M\x\R^k$.
Since $2n-m<k+d$, $g$ lifts (with respect to the projection $\R^{k+d}\to\R^d$) to a PL map $G\:N\to\R^{k+d}$
extending the composition $N_\star\xr{e_\star}\R^k\subset\R^{k+d}$ and such that 
$f\x G\:N\to M\x\R^{k+d}$ is an embedding.
Since $g$ is generic on $N_\circ$ and $g|_{N_\star}=e_\star$, we may assume that $\Delta_h$ is a subpolyhedron 
of codimension $\ge k$ in $\Delta_{f_\circ}$.

Since $f$ is PL, $\bar\Delta_f$ is a closed subpolyhedron of $N\x N$ and $\Sigma_f$ is 
a closed subpolyhedron of $N$.
Let us note that $\bar\Delta_f^\star:=\bar\Delta_f\but\bar\Delta_{f_\circ}$ may be 
larger than $\bar\Delta_{f_\star}$, even though $\Delta_{f_\star}=\Delta_f\but\Delta_{f_\circ}$.
(Clearly, $\bar\Delta_f^\star=\bar\Delta_f\cap (N_\star\x N_\star)$.)
Let us define an isovariant map $\beta\:\bar\Delta_f\to\R^k$ by sending 
$\Delta_{\Sigma_f}$ to $0$ and by $(x,y)\mapsto\alpha(x,y)\cdot ||G(y)-G(x)||$ for $x\ne y$.
Let us note that $\beta(x,y)=\frac{g(y)-g(x)}{||g(y)-g(x)||}\cdot||g(y)-g(x)||=g(y)-g(x)$ 
for $(x,y)\in\bar\Delta_f^\star$.
In particular, $\beta|_{\bar\Delta_f^\star}$ is a PL map.

By Lemma \ref{isovariant PL approximation} $\beta$ is isovariantly homotopic keeping $\bar\Delta_f^\star$ 
fixed to a PL map $\phi\:\bar\Delta_f\to\R^k$.
By equivariantly perturbing the images of the vertices of an equivariant triangulation of $\bar\Delta_f$
such that $\phi$ is linear on its simplexes we may assume that $\phi$ is generic on $\bar\Delta_{f_\circ}$
(as an equivariant map).
The equivariant map $\bar\Delta_f\to\R^k$ defined by $(x,y)\mapsto g(y)-g(x)$ is equivariantly homotopic to 
$\phi$ by the linear homotopy $\Phi\:\bar\Delta_f\x I\to\R^k$, which keeps $\bar\Delta_f^\star$ fixed.
Thus we have $\Phi(x,y,t)=g(y)-g(x)$ for all 
$(x,y,t)\in\bar\Delta_f\x\{0\}\cup\bar\Delta_f^\star\x I$ and 
$\Phi(\Delta_f\x\{1\}\cup\Delta_{f_\star}\x I)\subset\R^k\but\{0\}$.
Hence $\nabla_\ccirc:=\Phi^{-1}(0)\but(\Delta_N\x I)$ lies in $\Delta_{f_\circ}\x [0,1)$, also 
$\nabla_\ccirc\cap(\Delta_{f_\circ}\x\{0\})=\Delta_h$.
Let $\nabla_\circ$ be the closure of $\nabla_\ccirc$ in $\bar\Delta_{f_\circ}\x I$, and let
$\nabla$ be its closure in $\bar\Delta_f\x I$.
Let $D_\ccirc$, $D_\circ$ and $D$ be the images of $\nabla_\ccirc$, $\nabla_\circ$ and $\nabla$
under the projection $\bar\Delta_f\x I\to\bar\Delta_f$.
Since $\Phi$ is PL, $\nabla$ is a closed subpolyhedron of $\bar\Delta_f\x I$ and 
consequently $D$ is a closed subpolyhedron of $\bar\Delta_f$.
Since $\Phi$ is generic on $\bar\Delta_{f_\circ}\x I$, we may assume that 
$\nabla_\ccirc$ is of codimension $\ge k$ in $\Delta_{f_\circ}\x I$, hence of dimension $\le d$.
Consequently $\nabla$ and $D$ are also of dimension $\le d$.

Let $\pi\:\bar\Delta_f\subset N\x N\to N$ be the composition of the inclusion and the projection onto 
the first factor and let $\pi_\circ=\pi|_{\bar\Delta_{f_\circ}}$ and $\pi_\ccirc=\pi|_{\Delta_{f_\circ}}$.

\begin{lemma} \label{pi-circ}
(a) $\Delta_{\pi_\ccirc}$ is homeomorphic to $\Delta^{(3)}_{f_\circ}$.

(b) $\Delta_{\pi_\circ}\but\Delta_{\pi_\ccirc}$ is a union of two homeomorphic copies of 
$(\Sigma_{f_\circ}\x N)\cap\Delta_{f_\circ}$.

(c) $\Sigma_{\pi_\ccirc}$ coincides with $(N\x\Sigma_{f_\circ})\cap\Delta_{f_\circ}$ as long as $f$ 
is generic on $N_\circ$.

(d) $\Sigma_{\pi_\circ}\but\Sigma_{\pi_\ccirc}$ coincides with $\Delta_{\Sigma^{(3)}_{f_\circ}}$ 
as long as $f$ is generic on $N_\circ$.
\end{lemma}

\begin{proof}[Proof. (a)]
Every point of $\Delta_{\pi_\ccirc}$ is of the form $\big((x,y),(x,z)\big)$, 
where $(x,y)$ and $(x,z)$ belong to $\Delta_{f_\circ}$ and $y\ne z$.
Thus the projection $N^4\to N^3$, $\big((x,y),(x,z)\big)\mapsto (x,y,z)$, sends $\Delta_{\pi_\ccirc}$ 
homeomorphically onto $\Delta^{(3)}_{f_\circ}$.
\end{proof}

\begin{proof}[(b)]
Clearly, $\Delta_{\pi_\circ}\but\Delta_{\pi_\ccirc}$ is the union of the images of
$(\Sigma_{f_\circ}\x N)\cap\Delta_{f_\circ}$ under the two embeddings $N^2\to N^4$ given by 
$(x,y)\mapsto\big((x,x),(x,y)\big)$ and $(x,y)\mapsto\big((x,y),(x,x)\big)$.
\end{proof}

\begin{proof}[(c)]
It is easy to see that $\Sigma_{\pi_\ccirc}\subset (N\x\Sigma_{f_\circ})\cap\Delta_{f_\circ}$. 
In more detail, by definition, $\Sigma_{\pi_\ccirc}$ consists of pairs $(p,q)\in\Delta_{f_\circ}$ 
whose arbitrary neighborhood in $\Delta_{f_\circ}$ contains distinct pairs $(x,y)$ and $(x,z)$.
Here $f_\circ(y)=f_\circ(x)=f_\circ(z)$, so $q\in\Sigma_{f_\circ}$.
Thus $\Sigma_{\pi_\ccirc}\subset (N\x\Sigma_{f_\circ})\cap\Delta_{f_\circ}$.

The reverse inclusion $(N\x\Sigma_{f_\circ})\cap\Delta_{f_\circ}\subset \Sigma_{\pi_\ccirc}$ 
needs the hypothesis that $f$ is generic on $N_\circ$.
Assuming this hypothesis, we actually have
$\Sigma_{\pi_\ccirc}=E_{f_\circ}=(N\x\Sigma_{f_\circ})\cap\Delta_{f_\circ}$ (see \S\ref{simple fold maps}
concerning $E_{f_\circ}$).
\end{proof}

\begin{proof}[(d)]
Let $\Sigma'_{f_\circ}$ denote the set of points $p\in N_\circ$ such that every neighborhood of $p$ contains 
distinct points $x$ and $y$ such that $x\in\Sigma_{f_\circ}$ and $f_\circ(x)=f_\circ(y)$; thus 
$\Delta_{\Sigma'_{f_\circ}}$ is the intersection of $\Delta_N$ with the closure of 
$(N\x\Sigma_{f_\circ})\cap\Delta_{f_\circ}$.

It is easy to see that $\Sigma_{\pi_\circ}\but\Sigma_{\pi_\ccirc}$ contains $\Delta_{\Sigma^{(3)}_{f_\circ}}$ 
and is contained in $\Delta_{\Sigma^{(3)}_{f_\circ}}\cup\Delta_{\Sigma'_{f_\circ}}$; actually, all three sets 
coincide as long as $f$ is generic on $N_\circ$.

In more detail, clearly $\Sigma_{\pi_\circ}\but\Sigma_{\pi_\ccirc}$ coincides with $\Sigma_{\pi_\circ}\cap\Delta_N$.
By definition, the latter consists of pairs $(p,p)\in\bar\Delta_{f_\circ}\cap\Delta_N$ whose arbitrary 
neighborhood in $\bar\Delta_{f_\circ}$ contains distinct pairs $(x,y)$ and $(x,z)$.
The case where $x$, $y$, $z$ can be chosen to be pairwise distinct amounts to $p\in\Sigma^{(3)}_{f_\circ}$,
and the case where $x$ and $y$ can be chosen to be equal amounts to $p\in\Sigma'_{f_\circ}$.
This proves the two inclusions.

Finally, as long as $f$ is generic on $N_\circ$, $(x,y)\in (N\x\Sigma_{f_\circ})\cap\Delta_{f_\circ}$ implies 
$(x,y,y)\in\bar\Delta^{(3)}_{f_\circ}$ and therefore $\Sigma'_{f_\circ}\subset\Sigma^{(3)}_{f_\circ}$.
\end{proof}

\subsection{$\pi$ embeds $D$}
Since $f$ is generic on $N_\circ$, we may assume that $\dim\Delta^{(3)}_{f_\circ}\le 3n-2m$ and that
$\dim(\Sigma_{f_\circ}\x N)\cap\Delta_{f_\circ}\le 3n-2m-1$.
Hence by Lemma \ref{pi-circ}(a,b) $\bar\Delta_{\pi_\circ}$ is of dimension at most $3n-2m$.
Therefore so is its closure $\bar\Delta_{\pi_\circ}$.

Let $K$ be an equivariant triangulation of $\bar\Delta_f$ and $L$ a triangulation of $N$ 
such that $\pi\:K\to L$ is simplicial and $\bar\Delta_f^\star$ and $N_\star$ are triangulated by subcomplexes.
Since $f_\circ$ is non-degenerate, so is $\pi_\circ$.
Therefore if $\sigma_1$, $\sigma_2$ are simplexes of $K$ not in $\bar\Delta_f^\star$ such that 
$\pi(\sigma_1)=\pi(\sigma_2)$, then they have the same dimension $s$.
Moreover, $s\le\dim\bar\Delta_{\pi_\circ}\le 3n-2m$.
Since $\Phi$ is generic on $\bar\Delta_{f_\circ}\x I$, we may assume that $\nabla\cap\sigma_i\x I$ 
is of dimension at most $s-(k-1)$, and therefore so is $D\cap\sigma_i$.
Since $2\big(s-(k-1)\big)-s\le 3n+2-2m-2k<0$, we may assume that $\pi$ embeds $D$.

\subsection{Construction of $U$}
Since $f_\circ$ is a simple fold map, we conclude from Lemma \ref{pi-circ}(c,d) that $\Sigma_{\pi_\circ}=\emptyset$.
Thus $\pi_\circ\:\bar\Delta_{f_\circ}\to N$ is an immersion.

Since $\pi_\circ$ immerses $\bar\Delta_{f_\circ}$ and embeds $D_\circ$, it must in fact embed some $\Z/2$-invariant 
closed neighborhood $U_\circ$ of $D_\circ$ in $\bar\Delta_{f_\circ}$.%
\footnote{Indeed, since $D_\circ\x D_\circ$ and $\bar\Delta_{\pi_\circ}$ are disjoint closed subsets of 
$\bar\Delta_{f_\circ}\x\bar\Delta_{f_\circ}$, there exists an open neighborhood $U'_\circ$ of $D_\circ$ 
such that $U'_\circ\x U'_\circ$ is disjoint from $\bar\Delta_{\pi_\circ}$.
Since $D_\circ$ is invariant, $U'_\circ$ contains an invariant closed neighborhood $U_\circ$ of $D_\circ$.}
It is not hard to choose $U_\circ$ so that its closure $U$ in $\bar\Delta_f$ is a closed subpolyhedron of 
$N\x N$, and $U_\star:=U\but U_\circ$ coincides with $D\but D_\circ$.
Since $\pi$ embeds $D$ and $U_\circ$, and $\pi(U_\star)\cap\pi(U_\circ)\subset N_\star\cap N_\circ=\emptyset$, 
we get that $\pi$ embeds $U$.
(It would in fact suffice for our purposes to know that it embeds $U_\circ$, but to stay within
the PL category it helps to deal with compact polyhedra.)

\subsection{Construction of $e_t$}
Let us construct a PL homotopy $e_t\:\pi(U)\to\R^k$ keeping $\pi(U_\star)$ fixed
and such that $e_0=g|_{\pi(U)}$ and $e_t(y)-e_t(x)=\Phi(x,y,t)$ for all $(x,y)\in U$.
(We do not need to worry if $e_t$ keeps the entire frontier of $\pi(U)$ fixed, because we will eventually 
use only the restriction of $e_t$ to a small neighborhood of $\pi(D)$ in $\pi(U)$.)

For each $x\in\pi(U)$ there is a unique $y=y(x)\in\pi(U)$ such that 
$(x,y)\in U$.
The vector $e_t(y)-e_t(x):=\Phi(x,y,t)$ is given for each $t$, and we have some freedom
in choosing its endpoints $e_t(x)$ and $e_t(y)$.
We may, for instance, endow every point $x\in\pi(U)$ with a ``mass'' continuously 
depending on $x$ and choose the endpoints $e_t(x)$ and $e_t(y)$ so that their ``center 
of gravity'' does not depend on $t$.
For our purposes, it suffices to consider the constant mass function, so that the center
of gravity is the midpoint of the vector.
The requirement that this midpoint be fixed under the homotopy can be expressed by
\[\frac{e_t(x)+e_t(y)}2=\frac{g(x)+g(y)}2,\]
where the left hand side can be rewritten as
\[e_t(x)+\tfrac12\big(e_t(y)-e_t(x)\big)=e_t(x)+\tfrac12\Phi(x,y,t).\]
Thus we define $e_t\:\pi(U)\to\R^k$ by 
\begin{align*}
e_t(x)=&\ \tfrac12\big(g(x)+g(y)-\Phi(x,y,t)\big)\\
=&\ g(x)+\tfrac12\big(g(y)-g(x)-\Phi(x,y,t)\big),
\end{align*}
where $y\in\pi(U)$ is the unique point such that $(x,y)\in U$.
Clearly, $e_t$ is piecewise linear, $e_0$ is the restriction of $g$, and 
$e_t(y)-e_t(x)=\Phi(x,y,t)$ for all $(x,y)\in U$; the latter can also be verified 
directly:
\[e_t(y)-e_t(x)=g(y)-g(x)-(\tfrac12+\tfrac12)\big(g(y)-g(x)-\Phi(x,y,t)\big)=\Phi(x,y,t).\]
Finally, since $U_\star\subset\bar\Delta_f^\star$ and $\Phi(x,y,t)=g(y)-g(x)$
for $(x,y)\in\bar\Delta_f^\star$, we have $e_t(x)=g(x)$ for all $x\in\pi(U_\star)$. 
Let us note that since $e_t(y)-e_t(x)=\Phi(x,y,t)$ for $(x,y)\in U$ and 
$\Phi(\Delta_f\x\{1\})\subset\R^k\but\{0\}$, $e_1(x)\ne e_1(y)$ for $(x,y)\in U$.

\begin{remark}\label{announcement}
By considering a non-constant mass function, it is not hard to generalize the above construction of $e_t$ to 
the case where $f$ embeds $\Sigma_f$ but is not necessarily a simple fold map.
However, in the case where $f$ is a simple fold map, below we extend $e_t$ to a neighborhood of $\pi(U)$ in $N$ 
without creating new double points; this seems to be very difficult to achieve when $f$ is not a simple fold map.
The latter issue was overlooked by the author at the time of the announcements of Theorem \ref{th1}(c) in
\cite{M1}*{third remark after Theorem 5} and \cite{M3}*{\S1}, which is why those announcements contain 
the restriction $4n-3m\le k$ (instead of $3n-2m\le k$). 
\end{remark}

\subsection{Perturbation of $e_t$}
The constraint $e_t(y)-e_t(x)=\Phi(x,y,t)$ for $(x,y)\in U$ can be used to 
reconstruct $\Phi$ from $e_t$.
More precisely, if we amend $e_t$ into a PL homotopy $g'_t$ by an amendment with support
in $U$, then this constraint yields a new isovariant PL homotopy 
$\Phi'\:\bar\Delta_f\x I\to\R^k$ which coincides with $\Phi$ outside $U$.
Moreover, if $g'_t$ satisfies $g'_0=g|_{\pi(U)}$ and $g'_1(x)\ne g'_1(y)$ 
for $(x,y)\in U$, then $\Phi'$ satisfies 
$\Phi'(x,y,0)=\Phi(x,y,0)$ and $\Phi'(\Delta_f\x\{1\})\subset\R^k\but\{0\}$.
Also, as long as the amendment preserves the midpoints, that is,
$\frac12\big(g'_t(x)+g'_t(y)\big)=\frac12\big(g(x)+g(y)\big)$ for all $(x,y)\in U$, 
then it is compatible with the definition of $e_t$; that is, if we repeat the definition 
of $e_t$ with $\Phi'$ in place of $\Phi$, we will get nothing but $g'_t$.
Thus we are free to perturb $e_t$ keeping the midpoints fixed.
Then by arguments similar to equivariant general position we may assume $e_t$ to be generic
on $\pi(U_\circ)\x(0,1]$. 

Let us define $\Xi\:\pi(U)\x I\to M\x\R^k$ by $\Xi(x,t)=\big(f(x),\,e_t(x)\big)$.
We have $\Xi(x,0)=h(x)$ for $x\in\pi(U)$, where $h=f\x g\:N\to M\x\R^k$ is 
our original generic lift of $f$.
Since $e_1(x)\ne e_1(y)$ for $(x,y)\in U$, we get that $\Xi$ embeds 
$\pi(U)\x\{1\}$.
However, since $f$ may have triple points, $\Xi\big(\pi(U)\x\{1\}\big)$ may 
intersect $h\big(N\but\pi(U)\big)$.

\begin{lemma} \label{perturbation}
By perturbing $e_t$ we may make $\Xi\big(\pi(D_\circ)\x I\big)$ 
disjoint from $h\big(N_\circ\but\pi(D_\circ)\big)$.
\end{lemma}

\begin{proof}
Let $T$ and $T^*$ be the images of $(D_\circ\x N_\circ)\cap\bar\Delta_{f_\circ}^{(3)}$ 
under the projections of $D_\circ\x N_\circ$ to its factors.
Then $\pi(T)$ and $T^*$ are closed subsets of $N_\circ$ and $\dim T=\dim T^*\le d+n-m=3n-2m-k+1<k-1$.
Let us note that since $f_\circ$ is a simple fold map and $f$ is generic on $N_\circ$, we have
$\bar\Delta^{(3)}_{f_\circ}=\Delta^{(3)}_{f_\circ}$; consequently $T\subset D_\ccirc$.
Since $\pi$ embeds $D_\circ$, $T^*$ is disjoint from $\pi(D_\circ)$.
Then, since $\Delta_h\subset D_\circ$, $h(T^*)$ is disjoint from $h\big(\pi(D_\circ)\big)$, 
and in particular from $h\big(\pi(T)\big)$.
Since $f|_{\pi(T)\cup T^*}$ is non-degenerate and $e_t$ is generic on $\pi(U_\circ)\x(0,1]$, we may assume that 
$\Xi\big(\pi(T)\x I\big)$ is disjoint from $h(T^*)$ in $M\x\R^k$.

Indeed, let $Q$ and $R$ be triangulations of $N$ and $M$ such that $f\:Q\to R$ is simplicial, and let
$\sigma\subset\pi(T)$ and $\tau\subset T^*$ be simplexes of $Q$.
If $\Xi(x\x I)$ is not disjoint from $h(y)$, then $f(x)=f(y)$.
Hence if $x\in\sigma$ and $y\in\tau$, then $f(\sigma)=f(\tau)$ and $\sigma$, $\tau$ have the same 
dimension $s<k-1$.
Since $(s+1)+s<s+k$, we may assume $\Xi(\sigma\x I)$ to be disjoint from $h(\tau)$ in $f(\sigma)\x\R^k$.
Indeed, it is easy to see that a generic $(s+1)$-plane in $\R^{s+k}$, where $k>s+1$, is disjoint from 
$\R^s\x\{0\}$ and at the same time maps surjectively onto $\R^s$ under the vertical projection.

Therefore $\Xi\big(\pi(D_\circ)\x I\big)$ is disjoint from $h\big(N_\circ\but\pi(D_\circ)\big)$.
\end{proof}

\begin{lemma} There is a closed neighborhood $W^\sharp_\circ$ of $\pi(D_\circ)$ in $N_\circ$ 
such that $\Delta_{f|_{W^\sharp_\circ}}\subset U$.
\end{lemma}

It is clear from the proof that the closure $W^\sharp$ of $W^\sharp_\circ$ in $N$ may be assumed
to be a subpolyhedron of $N$.

\begin{proof}
Indeed, let $D_\circ^*=\pi^{-1}\big(\pi(D_\circ)\big)\but D_\circ\subset\bar\Delta_f$.
Since $\pi(D_\circ)\subset N_\circ$, we have $D_\circ^*\subset\bar\Delta_{f_\circ}$, 
and since $f_\circ$ is a simple fold map, $D_\circ^*\subset\Delta_{f_\circ}$.
Let $T$ and $T^*$ be as in the proof of Lemma \ref{perturbation}.
Then $\pi(D_\circ^*)=\pi(T)$ and $\pi^*(D_\circ^*)=T^*$, where $\pi^*\:\bar\Delta_f\subset N\x N\to N$ 
is the composition of the inclusion and the projection onto the second factor.
Let $O$ and $O^*$ be disjoint open neighborhoods of $\pi^*(D_\circ)=\pi(D_\circ)$ and of $T^*$ in $N_\circ$.
We may assume that their closures in $N$ are subpolyhedra of $N$.
Clearly, $(\pi^*)^{-1}(O)$ and $(\pi^*)^{-1}(O^*)$ are disjoint open neighborhoods of
$D_\circ$ and $D_\circ^*$ in $\bar\Delta_{f_\circ}$.
Then $J:=\big((\pi^*)^{-1}(O)\cap\Int U\big)\cup(\pi^*)^{-1}(O^*)$ is an open 
neighborhood of $D_\circ\cup D_\circ^*=\pi^{-1}\big(\pi(D_\circ)\big)$ in $\bar\Delta_{f_\circ}$.
Then $J$ is also open in $\bar\Delta_f$, so $\bar\Delta_f\but J$ is compact, and consequently 
so is its image $\pi(\bar\Delta_f\but J)$.
Since $\bar\Delta_{f_\circ}\but J$ is disjoint from $\pi^{-1}\big(\pi(D_\circ)\big)$, so 
is $\bar\Delta_f\but J$, and consequently $\pi(\bar\Delta_f\but J)$ is disjoint from 
$\pi(D_\circ)$.
Hence $O':=N_\circ\but\pi(\bar\Delta_f\but J)$ is an open neighborhood of 
$\pi(D_\circ)$ in $N_\circ$ such that $\pi^{-1}(O')\subset J$.
Let $W^\sharp_\circ$ be a closed neighborhood of $\pi(D_\circ)$ in $O\cap O'$.
Then $W^\sharp_\circ$ is disjoint from $O^*$ and 
$\pi^{-1}(W^\sharp_\circ)\subset J\subset U\cup(\pi^*)^{-1}(O^*)$.
Hence $\pi^{-1}(W^\sharp_\circ)\cap(\pi^*)^{-1}(W^\sharp_\circ)\subset U$; in other words, 
$\Delta_{f|_{W^\sharp_\circ}}\subset U$.
\end{proof}

\subsection{Construction of $W$ and $W^\flat$} 
Let $Z^\sharp_\circ$ be the closure of $N_\circ\but W^\sharp_\circ$ in $N_\circ$.
Since $\Xi\big(\pi(D_\circ)\x I\big)$ is disjoint from $h\big(N_\circ\but\pi(D_\circ)\big)$,
it is disjoint from $h(Z^\sharp_\circ)$.
Let us recall that $\Xi(x,t)=\big(f(x),\tfrac12g(x)+\tfrac12g(y)-\tfrac12\Phi(x,y,t)\big)$,
where $y\in\pi(U)$ is the unique point such that $(x,y)\in U$.
Since $\Xi\big(\pi(D_\circ)\x I\big)$ and $h(Z^\sharp_\circ)$ are disjoint, where $\pi(D_\circ)$ and 
$Z^\sharp_\circ$ are disjoint closed subsets of $N_\circ$ whose closures in $N$ are subpolyhedra of $N$, 
there exists a PL function $\eps\:N\to[0,\infty)$ such that 
$\eps^{-1}(0)=N_\star$ and for each $x\in\pi(D_\circ)$ and $t\in I$ we have 
$\big(f(x'),\tfrac12g(x')+\tfrac12g(y')-\tfrac12\Phi(x'',y'',t)\big)\notin h(Z^\sharp_\circ)$ whenever 
$x'$, $x''$ are $\eps(x)$-close to $x$ and $y'$, $y''$ are $\eps(y)$-close to the unique point $y\in\pi(U)$
such that $(x,y)\in U$.

Let $W_\circ$ be a closed neighborhood of $\pi(D_\circ)$ in $N_\circ$ which is contained in 
the open $\eps$-neighborhood of $\pi(D_\circ)$ (that is, the union of the open balls $B_{\eps(x)}(x)$ 
for all $x\in \pi(D_\circ)$).
We may assume that $W_\circ\subset\Int W^\sharp_\circ$ and that the closure $W$ of $W_\circ$ in $N$ 
is a subpolyhedron of $N$.
Let $W^\flat_\circ$ be a closed neighborhood of $\pi(D_\circ)$ in $\Int W_\circ$.
We may assume that the closure $W^\flat$ of $W^\flat_\circ$ in $N$ is a subpolyhedron of $N$.

Let $V^\flat=\Delta_{f|_{W^\flat}}$ and $V=\Delta_{f|_{W}}$, and let $\varLambda^\flat$ and 
$\varLambda$ be the closures of $\Delta_f\but V^\flat$ and $\Delta_f\but V$ in $\Delta_f$.
Similarly, let $\bar V^\flat=\bar\Delta_{f|_{W^\flat}}$ and $\bar V=\bar\Delta_{f|_{W}}$, 
and let $\bar\varLambda^\flat$ and $\bar\varLambda$ be the closures of $\bar\Delta_f\but\bar V^\flat$ 
and $\bar\Delta_f\but\bar V$ in $\bar\Delta_f$.
Let us note that $\pi(\bar V)\subset W$, and by our choice of $\eps$ (using $x''=x'$ and $y''=y'$), 
$\Xi\big(\pi(\bar V)\x I\big)$ is disjoint from $h(Z^\sharp_\circ)$.

\subsection{Construction of $\Phi^+$}
Let $\Phi^+$ be the composition
\[\bar\Delta_f\x I\xr{r}\bar\Delta_f\x\{0\}\cup(\bar V^\flat\cup\bar\varLambda)\x I\xr{\Psi}\R^k,\]
where $r$ is an isovariant PL retraction such that $r(x,y,t)=(x',y',t')$ implies $d(x,x')<\eps(x')$
and $d(y,y')<\eps(y')$ and $\Psi$ is defined by \[\Psi(x,y,t)=\begin{cases}
\Phi(x,y,t),&\text{ if }(x,y)\in\bar V^\flat;\\
g(y)-g(x),&\text{ if }(x,y)\in\bar \varLambda;\\
g(y)-g(x),&\text{ if }t=0.
\end{cases}\]
Then $\Phi^+$ is an isovariant PL homotopy satisfying $\Phi^+|_{\bar V^\flat\x I}=\Phi|_{\bar V^\flat\x I}$ 
and $\Phi^+(x,y,t)=g(y)-g(x)$ for all $(x,y,t)\in\bar\Delta_f\x\{0\}\cup\bar \varLambda\x I$.
Also, since $\Phi^+(V^\flat\x\{1\})\subset\Phi(\Delta_f\x\{1\})\subset\R^k\but\{0\}$ and 
$\Phi^+(\varLambda^\flat\x I)\subset\Phi(\Delta_f\x I\but D_\ccirc\x I)\subset\R^k\but\{0\}$, we have 
$\Phi^+(\Delta_f\x\{1\})\subset\R^k\but\{0\}$.

\subsection{Construction of $e_t^+$}
Let $e^+_t\:\pi(U)\to\R^k$ be defined similarly to $e_t$ but using $\Phi^+$ in place of $\Phi$:
\[e^+_t(x)=g(x)+\tfrac12\big(g(y)-g(x)-\Phi^+(x,y,t)\big),\]
where $y\in\pi(U)$ is the unique point such that $(x,y)\in U$.
Then like before, $e_t^+$ keeps $\pi(U_\star)$ fixed, $e^+_0=g|_{\pi(U)}$ and $e^+_t(y)-e^+_t(x)=\Phi^+(x,y,t)$ 
for all $(x,y)\in U$.
Also, $e_t^+(x)=e_t(x)$ for all $x\in\bar V^\flat$ and $e_t^+(x)=g(x)$ for all 
$x\in\pi(U\cap\bar\varLambda)$.

Let us define $\Xi^+\:\pi(U)\x I\to M\x\R^k$ by $\Xi^+(x,t)=\big(f(x),\,e^+_t(x)\big)$.
Then $\Xi^+(x,t)=h(x)$ for all $(x,t)\in\pi(U)\x\{0\}\cup\pi(U\cap\bar\varLambda)\x I$ and 
$\Xi^+(x,t)=\Xi(x,t)$ for all $x\in\pi(\bar V^\flat)$.
Also, since $e^+_t(x)-e^+_t(y)=\Phi^+(x,y,t)$ for all $(x,y)\in U$ and
$\Phi^+(\Delta_f\x\{1\})\subset\R^k\but\{0\}$, we get that $\Xi^+$ embeds $\pi(U)\x\{1\}$.
Since $\Xi^+(x,t)=\Xi(x,t)$ for all $x\in\pi(\bar V^\flat)$, where $\bar V^\flat\subset\bar V$ and
$\Xi(\pi(\bar V)\x I)$ is disjoint from $h(Z^\sharp_\circ)$, we also get that
$\Xi^+\big(\pi(\bar V^\flat)\x I\big)$ is disjoint from $h(Z^\sharp_\circ)$.
Moreover, in fact, $\Xi^+\big(\pi(\bar V)\x I\big)$ is disjoint from $h(Z^\sharp_\circ)$ due to 
our choice of $\eps$ (this time using $x''\ne x'$ and $y''\ne y'$).

\subsection{Construction of $h_t$}
Let $Z$ be the closure of $N\but W_\circ$ in $N$.
Let $\Gamma$ be the composition
\[N\x I\xr{R}N\x\{0\}\cup\big(\pi(U)\cup Z\big)\x I\xr{\Theta}\R^k,\]
where $R$ is a PL retraction such that $r(x,t)=(x',t')$ implies $d(x,x')<\eps(x')$ and $\Theta$ is defined by 
\[\Theta(x,t)=\begin{cases}
e_t^+(x,t),&\text{ if }(x,y)\in\pi(U);\\
g(x),&\text{ if }(x,y)\in Z;\\
g(x),&\text{ if }t=0.
\end{cases}\]
Let us note that the three cases agree on overlaps, including 
$\pi(U)\cap Z=\pi(U\cap\bar\varLambda)$.

Clearly, $f\x\Gamma\:N\x I\to M\x\R^k$ is an extension of $\Xi^+$.
Let us define $g_t\:N\to\R^k$ by $g_t(x)=\Gamma(x,t)$, and let $h_t=f\x g_t\:N\to M\x\R^k$.
Then $h_0=h$, $h_t|_Z=h|_Z$ and $h_t(x)=\Xi(x,t)$ for all $x\in\pi(\bar V^\flat)$.
Also, $h_1$ embeds $\pi(U)$, and $h_t\big(\pi(\bar V)\big)$ is disjoint from $h(Z^\sharp_\circ)$ 
for each $t\in I$.
Moreover, due to our choice of $\eps$, in fact, $h_t(W)$ is disjoint from $h(Z^\sharp_\circ)$ for each $t\in I$.
In other words, $h_t(W_\circ)$ is disjoint from $h(Z^\sharp)$ for each $t\in I$, where $Z^\sharp$ is 
the closure of $N\but W^\sharp_\circ$ in $N$.

\subsection{Verification}
Since $\pi(\Delta_h)\subset\pi(D_\ccirc)\subset\pi(D)$ and $h\big(\pi(D)\big)$ is disjoint from $h(Z)$, 
$h$ embeds $Z$.
Since $h_t|_Z=h|_Z$, so does $h_t$ for each $t\in I$.
In particular, $h_t$ embeds $Z^\sharp$, and $h_t(Z\but Z^\sharp)$ is disjoint from $h_t(Z^\sharp)$.
On the other hand, since $N\but Z\subset W_\circ$, where $h_t(W_\circ)$ is disjoint from 
$h(Z^\sharp)$ and $h|_{Z^\sharp}=h_t|_{Z^\sharp}$ for each $t\in I$, we get that $h_t(N\but Z)$ is 
disjoint from $h_t(Z^\sharp)$.
Thus $h_t(N\but Z^\sharp)$ is disjoint from $h_t(Z^\sharp)$.
Since $h_t$ also embeds $Z^\sharp$, we obtain that $\Delta_{h_t}\cap (N\x Z^\sharp\cup Z^\sharp\x N)=\emptyset$ 
for each $t\in I$.
In particular, $\Delta_{h_t}\subset W^\sharp\x W^\sharp$.
Finally, since $h_1$ embeds $\pi(\bar V^\sharp)$, where $\bar V^\sharp=\bar\Delta_{f|_{W^\sharp}}$, 
it also embeds $W^\sharp$.
Consequently, $h_1$ is an embedding.

It remains to verify that $\tl g_1$ is equivariantly homotopic to the given map $\alpha\:\Delta_f\to S^{k-1}$.
Let $V^\sharp=\Delta_{f|_{W^\sharp}}$ and let $\varLambda^\sharp$ be the closure of $\Delta_f\but V^\sharp$ 
in $\Delta_f$.
Since $\Delta_{h_t}\cap (N\x Z^\sharp\cup Z^\sharp\x N)=\emptyset$ for each $t\in I$, so in particular 
$\Delta_{h_t}\cap\varLambda^\sharp=\emptyset$, there is an equivariant homotopy 
$\psi_t\:\varLambda^\sharp\to S^{k-1}$, defined by 
$\psi_t(x,y)=\frac{g_t(y)-g_t(x)}{||g_t(y)-g_t(x)||}$, such that $\psi_1=\tilde g_1|_{\varLambda^\sharp}$.
Since $\psi_t$ keeps $U\cap\varLambda^\sharp$ fixed (where $U\cap\varLambda^\sharp$ contains $U\but\Int U$), 
it extends to an equivariant homotopy $\psi^+_t\:\Delta_f\to S^{k-1}$ between 
$\psi^+_0:=(\tilde g_1|_{V^\sharp})\cup\psi_0$ and $\psi^+_1:=\tilde g_1$.
But it is easy to see (using that $V\subset V^\sharp\subset U$) that $\psi^+_0$ coincides with the composition 
$\Delta_f\xr{\Phi^+|_{\Delta_f\x\{1\}}}\R^k\but\{0\}\to S^{k-1}$.
It follows from the definition of $\Phi^+$ that the latter is equivariantly homotopic to
the composition $\Delta_f\xr{\Phi|_{\Delta_f\x\{1\}}}\R^k\but\{0\}\to S^{k-1}$.
But the latter is in turn equivariantly homotopic to $\alpha$ by the construction of $\Phi$. 
\qed

\subsection{Proof of Addendum} To prove Addendum \ref{mainlemma'}(a), we proceed as in 
the above proof of Lemma \ref{mainlemma}(a), with the following modifications.

The map $g$ is now not chosen at random but given. 
By the hypothesis $\Delta_h$ is disjoint from $X$.
Also, since $\alpha$ extends $\tilde g$, we may assume that $\Phi$ keeps $X$ fixed.
Then $D$ will be disjoint from $X$.
Then we may also choose $U$ to be disjoint from $X$.
Then, although $f_\circ$ is no longer assumed to be a simple fold map, the same constructions still work. 
\qed

\section{Proof of Main Lemma: Smooth case} \label{sec3}

In this section we will assume familiarity with Appendix \ref{sec-xx}
(with the exception of Theorem \ref{projective differential} and its consequences, which are not needed now).

The proof of Lemma \ref{mainlemma}(b) and Addendum \ref{mainlemma'}(b) is generally similar to 
the proof of Lemma \ref{mainlemma}(a) and Addendum \ref{mainlemma'}(a), with many straightforward 
modifications.
We will discuss only substantial modifications.

\subsection{Boundary constraints}
Compared to the PL case, $f$ is assumed to be generic not just on $N\but N_\star$, but on the entire $N$.
This means, in particular, that $f$ may be assumed to be transverse to $\partial M_\star$.
Then whatever is given on $N_\star$ can be extended over a small neighborhood of $N_\star$ in $N$.
Due to this, $\pi(D)$ will be entirely contained in $N\but N_\star$, and so we no longer need to care 
about the intersection of $\pi(U)$ with $N_\star$.
(The solution we used in the PL case, to keep this intersection to a minimum, would actually not suffice 
for the smooth case, had we still assumed $f$ to be generic only on $N\but N_\star$.)
Due to this simplification, we will no longer discuss the behaviour at the boundary in what follows.

\subsection{Making $\phi$ smooth}
The construction of the isovariant map $\phi\:\bar\Delta_f\to\R^k$ is modified as follows.
Since the manifold with boundary $\check\Delta_f$ is equivariantly homotopy equivalent 
to its interior $\Delta_f$, the given map $\alpha\:\Delta_f\to S^{k-1}$ is equivariantly 
homotopic to the restriction of a smooth map $\check\alpha\:\check\Delta_f\to S^{k-1}$.
On the other hand, let $\check G\:\check\Delta_f\to S^{k+d-1}\x[0,\infty)$ be as in Lemma \ref{3.1}, 
and let $\kappa\:\check\Delta_f\to [0,\infty)$ be the composition of $\check G$ and 
the projection onto $[0,\infty)$.
Since $\kappa$ is smooth, so is $\check\phi:=\check\alpha\x\kappa\:\check\Delta_f\to S^{k-1}\x[0,\infty)$.

Let $\mathring\R^k$ be the blowup of $\R^k$ at $0$; thus the map $S^{k-1}\x[0,\infty)\to\R^k$, $(x,s)\mapsto sx$, 
factors as a composition $S^{k-1}\x[0,\infty)\xr{Q}\mathring\R^k\xr{R}\R^k$.
Also let $\hat\Delta_f$ be the image of $\check\Delta_f$ in the blowup of $N\x N$ 
along $\Delta_N$.
Then the projection $\check\Delta_f\to\bar\Delta_f$ factors as a composition
$\check\Delta_f\xr{q}\hat\Delta_f\xr{r}\bar\Delta_f$.
Since $\check\phi$ is equivariant, it descends to a map $\hat\phi\:\hat\Delta_f\to\mathring\R^k$,
which is easily seen to be smooth.
Since $f$ is a corank one map, $r$ is a diffeomorphism, and consequently $\check\phi$
descends to a smooth map $\phi\:\bar\Delta_f\to\R^k$, as illustrated in the diagram:
\[\begin{CD}
\check\Delta_f@>\check\phi>>S^{k-1}\x[0,\infty)\\
@VVq V@VVQ V\\
\hat\Delta_f@>\hat\phi>>\mathring\R^k\\
@VVr V@VVR V\\
\bar\Delta_f@>\phi>>\R^k.
\end{CD}\]

\subsection{Analyzing $\phi$ at $\Sigma_f$}
Let us compute $d\phi_{(x,x)}(v)$ for each $x\in\Sigma_f$ and each unit vector 
$v\in T_x N$ in the kernel of $df_x$, where the tangent bundle $\T_N\:TN\to N$ is 
identified with the normal bundle $\nu$ of $\Delta_N\subset N^2$.
By Lemma \ref{smoothcurve} there exists a smooth curve $\delta\:\R\to N$ such that 
$\delta(0)=x$, $\delta'(0)=v$ and $f\big(\delta(t)\big)=f\big(\delta(-t)\big)$ for each $t\in\R$.
Let $\gamma=\delta\x\delta\:\R\to N\x N$ and let $\check\gamma\:\R\to\check N$ be the lift of $\gamma$.
Then
\begin{multline*}
d\phi_{(x,x)}(v)=d\phi_{(x,x)}\big(\gamma'(0)\big)=(\phi\gamma)'(0)
=\lim_{t\to 0^+\!\!\!\!}\,\,\frac{(\phi\gamma)(t)-0}{t}
=\lim_{t\to 0}\,\frac{\check\phi\big(\delta(t),\delta(-t)\big)}{t}\\
=\lim_{t\to 0^+\!\!\!\!}\,\frac{\big|\big|G\delta(-t)-G\delta(t)\big|\big|
\ \check\alpha\big(\delta(t),\delta(-t)\big)}{t}
=\Bigg|\Bigg|\lim_{t\to 0}\frac{G\delta(-t)-G\delta(t)}{t}\Bigg|\Bigg|\ 
\check\alpha\Big(\lim_{t\to 0^+\!\!\!\!}\ \gamma(t)\Big)\\
=\big|\big|{-2}(G\delta)'(0)\big|\big|\ \check\alpha\big(\check\gamma(0)\big)
=-2\big|\big|dG_x\big(\delta'(0)\big)\big|\big|\ 
\check\alpha(x,v)=-2\big|\big|dG_x(v)\big|\big|\check\alpha(x,v).
\end{multline*}
Here $dG_x(v)\ne 0$ since $G$ is a smooth embedding.
Hence $d\phi_{(x,x)}(v)\ne 0$.

\subsection{Perturbing $\Phi$, $\Phi^+$ and $g_t$}
When perturbing $\phi$ along with the map $e_1\:\pi(U)\to\R^k$ so that
$\phi(x,y)=e_1(y)-e_1(x)$ for $(x,y)\in U$, the following conditions must be preserved:
$\phi$ is smooth and isovariant, and the restriction of $d\phi$ to $\ker df$ is 
a monomorphism of bundles.
It is not hard to ensure that $\phi^+:=\Phi^+|_{\bar\Delta_f\x\{1\}}$ also satisfies these conditions 
by slightly perturbing it keeping $\bar V^\flat\cup\bar\varLambda$ fixed.
Also, it is not hard to ensure that $g_1$ is smooth by slightly perturbing it keeping
$\pi(U)\cup Z$ fixed.

\subsection{Injectivity of the differential}
Let us compute $d(g_1)_x(v)$, where $x\in\Sigma_f\cap\pi(U)$ and $v\in T_x N$ is 
a unit vector in the kernel of $df_x$.
We may assume that $\gamma\big((0,1)\big)\subset U$ and so 
$\delta\big((-1,1)\big)\subset\pi(U)$, so that we can use the constraint
$g_1(y)-g_1(x)=\phi^+(x,y)$ for $(x,y)\in U$.
\begin{multline*}
d(g_1)_x(v)=d(g_1)_x\big(\delta'(0)\big)=(g_1\delta)'(0)
=\lim_{t\to 0}\,\frac{g_1\delta(t)-g_1\delta(-t)}{2t}
=-\lim_{t\to 0}\,\frac{\phi^+\big(\delta(t),\delta(-t)\big)}{2t}\\
=-\frac12\,\lim_{t\to 0^+\!\!\!\!}\,\,\frac{\phi^+\gamma(t)-0}{t}
=-\frac12 (\phi^+\gamma)'(0)=-\frac12 d\phi^+_{(x,x)}\big(\gamma'(0)\big)
=-\frac12 d\phi^+_{(x,x)}(v).
\end{multline*}
But $d\phi^+_{(x,x)}(v)\ne 0$ since the restriction of $d\phi^+$ to $\ker df$ is a monomorphism.
Hence $d(g_1)_x(v)\ne 0$.
Thus $f\x g_1$ is a smooth embedding. 
\qed

\section{Proofs of Theorems \ref{akhmetiev} and \ref{th3}}\label{sec5}

\begin{proof}[Proof of Theorem \ref{akhmetiev}] Since $N$ is stably parallelizable, the tangent bundle of 
$\tilde N$ is isomorphic to $n\lambda$, where $\lambda$ is the line bundle associated with the two-fold 
cover $\tilde N\to\tilde N/t$ (see Lemma \ref{s-parallelizable}).
On the other hand, since $M$ is stably parallelizable, the normal bundle of $\Delta_f/t$ in $\tilde N/t$
is stably isomorphic to $n\lambda_f$, where $\lambda_f$ is the restriction of $\lambda$ over $\Delta_f/t$ 
(see Lemma \ref{s-parallelizable2}).
Therefore $\Delta_f/t$ is stably parallelizable.

By Proposition \ref{trivial range} $f$ lifts to an embedding $g\:M\to N\x\R^{n+1}$.
This yields an equivariant map $\tilde g\:\Delta_f\to S^n$.
We may assume that $\tilde g$ is transverse to the equatorial $S^{n-1}$.
Let $F$ be the composition $M\xr{g} N\x\R^{n+1}\xr{\pi}N\x\R$, where $\pi$ is the projection onto 
the vertical line.
Then $\Delta_F=\tilde g^{-1}(S^{n-1})$.
Consequently $\Delta_F/t=\bar g^{-1}(\R P^{n-1})$, where $\bar g\:\Delta_f/t\to\R P^n$ is determined
by $\tilde g$ and therefore is the classifying map of the bundle $\lambda_f$.
Hence $\lambda_f$ is the pullback of the tautological line bundle $\gamma$ over $\R P^n$.
Since $[\R P^{n-1}]$ is Poincar\'e dual to $w_1(\gamma)$, we obtain that $[\Delta_F/t]$ is
Poincar\'e dual to $w_1(\lambda_f)$.

On the other hand, since the normal bundle of $\R P^{n-1}$ in $\R P^n$ is isomorphic to $\gamma$,
the normal bundle of $\Delta_F/t$ in $\Delta_f/t$ is isomorphic to $\lambda$.
Since $\Delta_f/t$ is stably parallelizable, we get that $\tau\oplus\lambda$ is stably trivial,
where $\tau$ is the tangent bundle of $\Delta_F/t$.
Hence by \cite{KM}*{Lemma 3.5} $\tau\oplus\lambda$ is trivial.
Then by the Hirsch theorem (see \cite{RS2}, \cite{RS1}) $\Delta_F/t$ immerses in $\R^n$.
Since $n\ne 1,3,7$, this implies that $w_1^{n-1}(\Delta_F/t)=0$ (see \cite{M1}*{Lemma 9}).
Also we have $w_1(\Delta_F/t)=w_1(\lambda_F)$, where $\lambda_F$ is $\lambda$ restricted
over $\Delta_F/t$, so $w_1^{n-1}(\lambda_F)=0$.
Since $[\Delta_F/t]$ is Poincar\'e dual to $w_1(\lambda_f)$, it follows that $w_1(\lambda_f)^n=0$.
Then $\Delta_f$ admits an equivariant map to $S^{n-1}$ (see Lemma \ref{5.1}).
\end{proof}

\begin{lemma} \label{s-parallelizable}
If $N$ is a stably parallelizable $n$-manifold, then the tangent bundle of $\tilde N/t$
is stably isomorphic to $n\lambda$, where $\lambda$ is the line bundle associated with the double cover 
$\tilde N\to\tilde N/t$.
\end{lemma}

The case $N=S^n$ is discussed in \cite{A3}*{Lemma 3.1}.

\begin{proof}
Since $N$ is stably parallelizable, $Q:=N\x\R$ is parallelizable \cite{KM}*{proof of Lemma 3.4}.
The total space $T_{\tilde Q/t}$ of the tangent bundle $\tau_{\tilde Q/t}$ is naturally identified
with a submanifold of $\widetilde{T_Q}/t$.
The points of the latter are of the form $\{(x,y,u,v),(y,x,v,u)\}$, where $u\in T_xQ$ and $v\in T_yQ$.
If $e_0,\dots,e_n$ are the basis vectors of a framing of $\tau_Q$, then $\tau_{\tilde Q/t}$
is the Whitney sum $\eps_0\oplus\dots\oplus\eps_n\oplus\lambda_0\oplus\dots\oplus\lambda_n$ of 
its line subbundles $\eps_i$ with basis vectors $\{(x,y,e_i,e_i),(y,x,e_i,e_i)\}$ and line subbundles 
$\lambda_i$ with basis vectors $\{(x,y,e_i,-e_i),(y,x,-e_i,e_i)\}$.
Clearly, each $\eps_i$ is trivial and each $\lambda_i$ is isomorphic to the line bundle associated 
with the double cover $\tilde Q\to\tilde Q/t$.
It is also easy to see, by a similar construction, that the normal bundle of $\tilde N/t$ in 
$\tilde Q/t$ is isomorphic to $\eps\oplus\lambda$.
Hence $\tau_{\tilde N/t}\oplus\eps\oplus\lambda\simeq(n+1)\eps\oplus(n+1)\lambda$.
By adding a bundle $\lambda^\perp$ such that $\lambda\oplus\lambda^\perp$ is trivial
to both sides of the equation, we conclude that $\tau_{\tilde N/t}$ is stably isomorphic to $n\lambda$.
\end{proof}

\begin{lemma} \label{s-parallelizable2} 
Let $M$ be a stably parallelizable smooth $n$-manifold, $N$ a smooth $n$-manifold
and $f\:N\to M$ a self-transverse map.
Then the normal bundle $\nu_f$ of $\Delta_f/t$ in $\tilde N$ is stably isomorphic to $n\lambda$, where 
$\lambda$ is the line bundle associated with the double cover $\Delta_f\to\Delta_f/t$.
\end{lemma}

The case $M=\R^n$ is discussed in \cite{A3}*{\S3} (see also \cite{AM}*{\S4}).

\begin{proof} Since $N$ is stably parallelizable, $Q:=N\x\R$ is parallelizable \cite{KM}*{proof of Lemma 3.4}.
Let $R=N\x\R$ and let $F=f\x\id_\R\:R\to Q$.
Let $e_0,\dots,e_n$ be the basis vectors of a framing of the normal bundle of $\Delta_Q$ in $Q\x Q$ such that
the differential of the factor exchanging involution of $Q\x Q$ sends each $e_i$ to $-e_i$.
Then by pulling back the $e_i$ we obtain an equivariant framing of the normal bundle $\nu_F$ of $\Delta_F$ in 
$\tilde R$; that is, an isomorphism $\nu_F\simeq(n+1)\eps$ which is equivariant with respect to the differential 
of the factor exchanging involution on $\tilde R$ and the sign involution $v\mapsto -v$ on $(n+1)\eps$.
It follows that the normal bundle of $\Delta_F/t$ in $\tilde R/t$ splits as $\lambda_0\oplus\dots\oplus\lambda_n$, 
where each $\lambda_i$ is isomorphic to the line bundle associated with the double cover $\tilde Q\to\tilde Q/t$.
On the other hand, it is easy to see, by a similar construction, that the normal bundle of $\Delta_f/t$ in
$\Delta_F/t$ is isomorphic to $\lambda$. 
Hence $\nu_f\oplus\lambda\simeq(n+1)\lambda$.
By adding a bundle $\lambda^\perp$ such that $\lambda\oplus\lambda^\perp$ is trivial to both sides of the equation, 
we conclude that $\nu_f$ is stably isomorphic to $n\lambda$.
\end{proof}

Let $\Delta$ be a space with a free involution $t$.
Its {\it Yang index} is the maximal $k$ such that $w_1(\lambda)^k\ne 0$,
where $\lambda$ is the line bundle associated with the $2$-covering
$\Delta\to\Delta/t$.

\begin{lemma}\label{5.1} If $\Delta$ is an $n$-manifold with a free involution
$t$, it admits an equivariant map to $S^{n-1}$ with the antipodal involution
if and only if it has Yang index $\,<n$.
\end{lemma}

\begin{proof} The ``only if'' part is trivial, since $S^{n-1}$ has Yang
index $<n$.

The first obstruction to the existence of an equivariant map
$\Delta\to S^{n-1}$ is
$e(\lambda)^n\in H^n(\Delta/t;\Z_\lambda^{\otimes n})$, where $\Z_\lambda$ is
the integral local coefficient system corresponding to $\lambda$, and
$e(\lambda)\in H^1(\Delta/t;\Z_\lambda)$ is the Euler class of $\lambda$,
which is induced from the generator of $H^1(\R P^\infty;\Z_\lambda)\simeq\Z/2$
under a classifying map of $\lambda$; this obstruction is complete
(see \cite{M2}*{\S2}).
Since $\Delta$ is a manifold, $H^n(C;\Z_\lambda^{\otimes n})$ is either $0$ or
$\Z$ or $\Z/2$ for each component $C$ of $\Delta$.
So $H^n(\Delta/t;\Z_\lambda^{\otimes n})$ contains no elements of order $4$.
Hence $e(\lambda)^n=0$ iff its $\bmod 2$ reduction $w_1(\lambda)^n=0$. 
\end{proof}

\begin{lemma}\label{5.2} \cite{M1}*{proof of Corollary to Theorem 5}
Let $N$ be a $\Z/2$-homology $n$-sphere and $\Delta\subset\tl N$ an
$n$-dimensional closed $\Z/2$-invariant submanifold.
$\Delta$ has Yang index $\,<n$ if and only if every (compact) $\Z/2$-invariant 
component of $\Delta$ projects with an odd degree to the first factor of $N\x N$.
\end{lemma}

All results of \cite{M1} referred to in this section are proved there in
the smooth category, but remain true in the PL category by the same arguments, 
without substantial modifications (apart from using PL transversality in place of 
smooth transversality).

\begin{proof}[Proof of Theorem \ref{th3}] By Corollary \ref{stable} and 
Lemmas \ref{5.1} and \ref{5.2}, we only need to show that every $\Z/2$-invariant component of 
$\Delta_f$ projects with an even degree to the first factor of $N\x N$.
Under the assumptions of part (a), this follows immediately from
\cite{M1}*{Theorem 2} and \cite{M1}*{diagram ($*$) in the introduction}.

For (b), let $f_0\:N_0\to M_0$ be the map whose mapping cylinder is
the universal covering of the mapping cylinder of $f$.
(So $\pi_1(M_0)=1$ and $\pi_1(N_0)=\ker(f_*)$.)

Let us first consider the case where $\pi_1(M)$ is finite.
Each component $C$ of $\Delta_f$ is covered by some component $C_0$ of
$\Delta_{f_0}$ with finite degree.
By \cite{M1}*{Theorem 3}, $C_0$ projects with zero degree to the (first, say)
factor of $N_0\x N_0$.
Hence $C$ projects with zero degree to the first factor of $N\x N$.

Now if $\pi_1(M)$ is infinite, let $C$ be a component of $\Delta_f$.
Let $p$ be a regular point of $f$ and let $b$ be some point of
$S:=\widetilde{f^{-1}(p)}\cap C$.
If $\ell\subset C$ is a path with endpoints in $S$, let $\alpha_\ell\in\pi_1(M,p)$
be the class of its image under the composition
$\Delta_f\subset N\x N\xr{\pi}N\xr{f}M$.
Let $S_0$ be the set of those points of $S$ that can be joined with $b$ by
a path $\ell\subset C$ such that $\alpha_\ell=1$.
If $g_1,\dots,g_i\in\pi_1(M,p)$ and $S_i\subset S$ are defined and
$S_i\ne S$, pick a point $q\in S\but S_i$ and a path $\ell\subset C$ joining $q$
with $b$; set $g_{i+1}=\alpha_{\ell}$ and define $S_{i+1}$ to consist
of those points of $S$ that can be joined with $b$ by a path $\ell$ such that
$\alpha_\ell\in\{1,g_1,\dots,g_{i+1}\}$.
This process terminates at some finite stage $r\le |S|$.

The outcome is that $S=\widetilde{f^{-1}(p)}\cap C$ is now in bijection with
$\bigsqcup_{i\le r}\widetilde{f_0^{-1}(g_i\hat p)}\cap C_0$, where $\hat p$
is a lift of $p$, and $\hat C_0$ is the component of $\Delta_{f_0}$ that
covers $C$ and contains the point $\hat b$ that corresponds to $b$ under
the identification $f^{-1}(p)=f_0^{-1}(\hat p)$.
Indeed, the points of $g_{i+1}(\hat S_{i+1}\but\hat S_i)$ are in $\hat C_0$
and those of $g_{i+1}(\hat S\but\hat S_{i+1})$ are not in $\hat C_0$ by
the construction of $S_{i+1}$.
If a point $g_{i+1}(\hat q)$ of $g_{i+1}(\hat S_i)$ is in $\hat C_0$, then
$q$ can be joined with $b$ by a path $\ell_{qb}\subset C$ such that
$\alpha_{\ell_{qb}}=g_{i+1}$.
Hence every point of $S_{i+1}\but S_i$ can be joined with $q$ by a path
$\ell$ going first to $b$, then to $q$ via the reverse of $\ell_{qb}$, and
finally back to $b$ so that $\alpha_\ell=g_{i+1}g_{i+1}^{-1}g_j$ for some
$j\le i$, which is a contradiction.

As a consequence, the degree of the composition $C\to N\to M$, as computed at
$b$, equals $r$ times the degree of the projection $C_0\to N_0\to M_0$, as
computed at $\hat b,g_1\hat b,\dots,g_r\hat b$.
Here we understand that if $C_0$ and $M_0$ are non-compact, but
the composition is a proper map, the degree is defined via cohomology with
compact support (or locally-finite homology) and so may be nonzero.
\cite{M1}*{Theorem 3} was stated in the case where $N_0$ is compact, but
the proof works for proper maps of non-compact manifolds as well.
Thus $C_0\to N_0$ has degree $0$, and therefore so does either $C\to N$ or
$N\to M$.
In the former case the proof is completed similarly to the case of finite
$\pi_1(M)$, and in the latter we refer to part (a). 
\end{proof}

\appendix

\section{Stable smooth maps} \label{stable maps}

Continuous maps $f,g\:N\to M$ between $C^r$-manifolds are called {\it $C^r$-left-right equivalent}, where 
$r\in\{0,1,\dots,\infty\}$, if there exist $C^r$-self-homeomorphisms $\phi$ of $N$ and $\psi$ of $M$ such that 
the following diagram commutes:
\[\begin{CD}
N@>\phi>>N\\
@VfVV@VgVV\\
M@>\psi>>M.
\end{CD}\]
A smooth (i.e.\ $C^\infty$) map $f\:N\to M$ between smooth (i.e.\ $C^\infty$) manifolds is called 
{\it $C^r$-stable} if it has a neighborhood in $C^\infty(N,M)$ whose every member is $C^r$-left-right 
equivalent to $f$.
By a stable smooth map we mean a $C^\infty$-stable one. 

\begin{theorem}[Triangulation Theorem] \label{Shiota} 
Let $M$ and $N$ be smooth manifolds, where $N$ is compact.
Then there exists a dense open set $S\subset C^\infty(N,M)$ such that every $f\in S$ is $C^0$-stable and 
$C^0$-left-right equivalent to a PL map (with respect to some smooth triangulations of $M$ and $N$).
\end{theorem}

R. Thom and J. Mather proved that the set of $C^0$-stable smooth maps $N\to M$ contains a dense open subset of 
$C^\infty(N,M)$ (see \cite{GWPL}), and A. Verona proved that $C^0$-stable smooth maps are triangulable \cite{Ve}.
Here is another approach:

\begin{proof} By Shiota's theorem \cite{Shi} a smooth map $f\:N\to M$ is $C^0$-left-right equivalent to 
a PL map if it is Thom stratified.
By \cite{GWPL}*{IV.3.3} $f$ is Thom stratified if it belongs to the set $S$ of smooth maps that are 
multi-transverse to a certain Whitney stratification of a suitable jet space.
By \cite{GWPL}*{IV.1.1 and IV.4.1} $S$ is open and dense in $C^\infty(N,M)$.
Also, all members of $S$ are $C^0$-stable \cite{GWPL}*{IV.4.4}.
\end{proof}

A smooth map $f\:N^n\to M^m$, $n\le m$, is called a {\it corank one} map, if $\dim(\ker df_x)\le 1$ at 
every point $x\in N$.
In particular, every smooth fold map is a corank one map.
The set of corank one maps is open in $C^\infty(N,M)$.%
\footnote{Indeed, $f$ is a corank one map if and only if $j^1f(N)$ is disjoint from the closed subset 
$\bigcup_{i\ge 2}\Sigma^i$ of $J^1(N,M)$ (see \cite{GG}).}
If $2m\ge 3(n-1)$, corank one maps are also dense in $C^\infty(N,M)$ (see \cite{GG}*{VI.5.2}).

\begin{theorem}[Corank One Stability Theorem] \label{Morin}
Let $M^m$ and $N^n$ be smooth manifolds, where $N$ is compact, $m\ge n$.
Let $A$ be the set of all corank one maps $N\to M$ and $S$ be the set of all $C^\infty$-stable maps $N\to M$.
Then $S\cap A$ is open and dense in $A$.
\end{theorem}

This result is well-known (see \cite{Ka}*{\S2.1}, \cite{GG}*{VII.6.4}), but I did not find a conclusive 
writeup of the proof in the literature. 

\begin{proof} Let $T$ be the set of all Thom--Boardman maps $N\to M$ that are self-transverse 
(=``have normal crossings'' in the terminology of \cite{GG}).
Then $T$ is dense in $C^\infty(N,M)$ \cite{GG}*{VI.5.2}.
Since $A$ is open, $T\cap A$ is dense in $A$.
For each $f\in T\cap A$ we have $\Sigma_f^{i_1,\dots,i_k}=\emptyset$ if $i_1>1$, and hence
(see \cite{Bo}*{2.18}) also if some $i_j>1$.
Hence each $x\in N$ belongs to some $\Sigma_f^{1,\dots,1,0}$ (which includes $\Sigma_f^0$).
Then by Morin's theorem \cite{Mo} $f$ has stable germs at all $x\in N$.
In particular, they are infinitesimally stable (see definition in \cite{GG}).
Since $f$ is also self-transverse, it has infinitesimally stable multi-germs at all $y\in M$
\cite{MaIV}*{1.6}.
Hence $f$ is infinitesimally stable (see e.g.\ \cite{GG}*{V.1.5 and V.1.6}) and therefore stable 
(see \cite{GG}).
Thus $T\cap A\subset S$.
Since $T\cap A$ is dense in $A$, so is $S\cap A$.
Clearly, $S$ is open in $C^\infty(N,M)$, so $S\cap A$ is open in $A$.
\end{proof}

\section{Stable PL maps} \label{pl-maps}

\subsection{PL transversality}
A subpolyhedron $Y$ of a polyhedron $X$ is said to be {\it collared} in $X$ if some neighborhood of $Y$ in $X$ 
is homeomorphic to $Y\x [0,1]$ by a PL homeomorphism that extends $\id\:Y\to Y\x\{0\}$.

A PL map $f\:P\to Q$ between polyhedra is said to be {\it PL transverse} to a triangulation $L$ of $Q$ if 
$f^{-1}(\partial\sigma)$ is collared in $f^{-1}(\sigma)$ for each simplex $\sigma$ of $L$.
The map $f$ is called {\it PL transverse} to a subpolyhedron $R$ of $Q$ if $f$ is PL transverse to 
some triangulation $L$ of $Q$ such that $R$ is triangulated by a subcomplex of $L$.

Let $K$ and $L$ be simplicial complexes.
A {\it semi-linear map} $f\:K\to L$ is a PL map $|K|\to |L|$ between their underlying polyhedra that sends 
every simplex of $K$ into some simplex of $L$ by an affine map.
Every semi-linear map $f\:K\to L$ determines a monotone map $[f]$ between the face posets%
\footnote{By ``face poset'' we mean the poset of all nonempty faces.}
of $K$ and $L$, defined by sending every simplex $\sigma$ of $K$ to the minimal simplex of $L$ 
containing $f(\sigma)$.
Two semi-linear maps $f,g\:K\to L$ will be called {\it combinatorially equivalent} if $[f]=[g]$, or in 
other words if $f^{-1}(\sigma)=g^{-1}(\sigma)$ for every simplex $\sigma$ of $L$.

If $f,g\:K\to L$ are combinatorially equivalent semi-linear maps, then $f$ is PL transverse to $L$ if and only if 
$g$ is PL transverse to $L$.
In this case the monotone map $[f]$ between the face posets of $K$ and $L$ is called a {\it stratification map}.
If $L'$ is a simplicial subdivision of a simplicial complex $L$, and $\id\:|L'|\to|L|$ is regarded as
a semi-linear map $s\:L'\to L$, then the monotone map $[s]$ is a stratification map.

Since composition of stratification maps is a stratification map \cite{M4}*{13.4} (see also
\cite{BRS}*{``Amalgamation'' on p.\ 23 and ``Extension to polyhedra'' on p.\ 35}), a PL map $P\to|L|$
that is transverse to $L'$ must also be transverse to $L$.
Conversely, if a PL map $f\:P\to |L|$ is transverse to $L$, it is PL-left-right equivalent%
\footnote{The definition of PL-left-right equivalence repeats that of $C^r$-left-right equivalence (see Appendix \ref{stable maps}), 
with ``$C^r$'' replaced by ``PL'' throughout.}
to a PL map that is transverse to $L'$ (see \cite{BRS}*{Theorem II.2.1 and ``Extension to polyhedra'' on p.\ 35});
moreover, the equivalence is via PL homeomorphisms $|L|\to |L|$ and $P\to P$ that preserve the simplexes of $L$ and
their preimages.
Using a generalization of these results from simplicial complexes to cone complexes (see \cite{BRS}, \cite{M4})
and a result of M. M. Cohen%
\footnote{If $f\:K_1\to K_2$ is a simplicial map between simplicial complexes, then $|f|\:|K_1|\to|K_2|$ is transverse to
the dual cone complex $K_2^*$ (see \cite{M4}*{16.2}).}
it is not hard to show (see \cite{BRS}*{\S II.4}) that every PL map $f\:P\to |L|$ is PL-left-right equivalent to a PL map 
that is transverse to $L$.

\subsection{Stable PL maps} \label{stable-pl}
If $K$ is a simplicial complex, a {\it linear map} $f\:K\to\R^m$ is a PL map $|K|\to\R^m$ whose restriction
to every simplex of $K$ is the restriction of an affine map.
Let $C(K,\R^m)$ be the subspace of $C^0(|K|,\R^m)$ consisting of all linear maps.
Let $S(K,\R^m)$ be the set of all linear maps $f\:K\to\R^m$ such that $f$ has a neighborhood in $C(K,\R^m)$ 
whose every member is PL-left-right equivalent to $f$.

More generally, given simplicial complexes $K$ and $L$ and a monotone map $\phi$ between their face posets, 
let $C(\phi)$ be the subspace of $C^0(|K|,|L|)$ consisting of all semi-linear maps $f\:K\to L$ such that 
$[f]=\phi$.
Let $S(\phi)$ be the set of all semi-linear maps $f\:K\to L$ such that $f$ has a neighborhood 
in $C(\phi)$ whose every member is PL-left-right equivalent to $f$.
If $L$ triangulates $\R^m$ and $\phi$ is a constant map onto some $m$-simplex, then $C(\phi)$ is 
an open subspace of $C(K,\R^m)$, and $S(\phi)=S(K,\R^m)\cap C(\phi)$.

\begin{theorem} \label{PL stability}
Let $K$ and $L$ be simplicial complexes, where $K$ is finite, and $\phi$ be a monotone map between their 
face posets.
Then $S(\phi)$ is open and dense in $C(\phi)$.
\end{theorem}

\begin{proof} The definition of $S(\phi)$ implies that it is open in $C(\phi)$.

Let us call a map $\nu$ from a finite set $F$ to an affine space $V$ a {\it general position map} if
for each $G\subset F$ the affine subspace of $V$ spanned by $\nu(G)$ is of dimension $\min(\#G-1,\dim V)$.
In other words, $\nu$ is required to be injective, unless $\dim V=0$; not to send any three points into
the same affine line, unless $\dim V\le 1$; not to send any four points into the same affine plane, 
unless $\dim V\le 2$; and so on.
Each of these conditions determines an open and dense subset of $C^0(F,V)$, and hence their intersection,
which is the set of all general position maps $F\to V$, is also open and dense in $C^0(F,V)$.

Let $G(\phi)$ be the subset of $C(\phi)$ consisting of all semi-linear maps $f\:K\to L$ such that for each 
simplex $\sigma$ of $L$ the restriction of $f$ to set of vertices of the subcomplex $f^{-1}(\sigma)$ of $K$ 
is a general position map into the affine space spanned by $\sigma$.
Since $K$ is finite, it is easy to see that $G(\phi)$ is open and dense in $C(\phi)$.

For every semi-linear map $f\:K\to L$ there is a standard construction yielding subdivisions $K'_f$, $L'_f$ 
of $K$, $L$ with respect to which $f$ is simplicial (see \cite{Ze}).
It is not hard to see that if $f\in G(\phi)$, then there is a neighborhood $U$ of $f$ in $C(\phi)$ 
such that $K'_f$ and $L'_f$ are isomorphic (as simplicial complexes) to $K'_g$ and $L'_g$ for every $g\in U$.
Then $g$ is PL-left-right equivalent to $f$.
Hence $G(\phi)\subset S(\phi)$, and therefore $S(\phi)$ is dense in $C(\phi)$. 
\end{proof}

\begin{corollary} If $K$ is a simplicial complex, $S(K,\R^m)$ is open and dense in $C(K,\R^m)$.
\end{corollary}

A PL map $f\:P\to Q$ between polyhedra will be called {\it stable} if there exist triangulations $K$, $L$ of 
$P$, $Q$ and a stratification map $\phi$ between their face posets such that $f$ is PL-left-right
equivalent to a member of $S(\phi)$.
In particular, stable PL maps $|K|\to\R^m$ include all members of $S(K,\R^m)$.

\begin{remark} (a) There is an alternative approach to stable PL maps.
By using the $C^1$ topology on semi-linear maps (see \cite{Mu}) one can do without fixing 
a triangulation of the domain.
However, a triangulation of the target still needs to be fixed, and hence PL transversality still needs 
to be used in this approach.

(b) A classical approach to general position arguments for PL maps from a compact polyhedron
to a PL manifold $M$ is to cover $M$ by coordinate charts, and achieve desired general position properties
separately in each chart (see \cite{Ze}).
Since the transition maps are not linear but PL, it seems to be difficult to formulate this approach in
invariant terms (such as stability), even for a fixed atlas. 
\end{remark}

\subsection{Examples of stable PL maps}

\begin{proposition} \label{emb-stable}
Every PL embedding $f\:P\to Q$ between polyhedra, where $P$ is compact, is stable.
\end{proposition}

\begin{proof} Upon replacing $f$ by a PL-left-right equivalent embedding we may assume that it is PL transverse
to some triangulation $L$ of $Q$.
Let $K$ be a triangulation of $P$ such that $f$ is a semi-linear map $K\to L$.
Let $g\:K\to L$ be a semi-linear map with $[g]=[f]$ that is $\eps$-close to $f$ in the sup metric.
Since $K$ is finite and $f$ sends disjoint simplexes of $K$ to disjoint subsets of $Q$, so does $g$,
as long as $\eps$ is sufficiently small.
But if $g(p)=g(q)$, where the minimal simplexes $\sigma$, $\tau$ of $K$ containing $p$ and $q$ share
a common face $\rho=\sigma\cap\tau$, then we have $\sigma=\rho*\sigma'$ and $\tau=\rho*\tau'$ and it is 
easy to see that either $g(\sigma)$ meets $g(\tau')$ or $g(\sigma')$ meets $g(\tau)$.
Thus $g$ is an embedding.
By a similar argument, the linear homotopy between $f$ and $g$ is an isotopy.
Moreover, if $h_i$ denotes the semi-linear map $K\to L$ which agrees with $g$ on the first $i$ vertices of $K$
and with $f$ on the remaining ones, so that $h_0=f$ and $h_k=g$, where $k$ is the number of vertices of $K$,
then by similar arguments each $h_i$ is an embedding, and the linear homotopy between $h_i$ and $h_{i+1}$
is an isotopy.
Each $[h_i]=[f]$, hence every such isotopy is covered by a PL ambient isotopy (even if $Q$ is not a manifold).
Thus $g$ is PL-left-right equivalent to $f$.
\end{proof}

\begin{example} Let $N^n$ and $M^m$ be PL manifolds, where $N$ is compact and $m\ge n$, and let 
$f\:N\to M$ a PL map.
As long as $f$ is non-degenerate, for each $x\in N$ there is a PL map $\lk(x,f)\:\lk(x,N)\to\lk\big(f(x),M\big)$, 
which is well-defined up to PL-left-right equivalence; and for each $y\in M$ there is a PL map 
$\lk(y,f)\:\bigsqcup_{f(x)=y}\lk(x,N)\to\lk(y,M)$, which is also well-defined up to PL-left-right equivalence.

(a) If $m\ge 2n+1$, it follows from Proposition \ref{emb-stable} that $f$ is stable if and only if it is
an embedding.

(b) If $m=2n$, it follows by similar arguments that $f$ is stable if and only if it is an immersion 
(i.e., locally injective) with a finite set $\Delta$ of transverse double points (i.e.\ points $y\in M$ such 
that $\lk(y,f)$ is PL-left-right equivalent, not necessarily preserving the orientations, to the Hopf link 
$\partial I^n\x\{0\}\sqcup\{0\}\x\partial I^n\subset\partial(I^n\x I^n)$, where $I=[-1,1]$).
Let us note that stable maps $N^2\to M^4$ may be locally knotted at finitely many points of 
$N\but f^{-1}(\Delta)$.
Nevertheless, stable maps $N^2\to\R^4$ have a normal Euler class \cite{BJ}.

(c) If $m=2n-1$, $n>2$, similar techniques work to show that $f$ is stable if and only if $M$ contains
a finite subset $\Sigma$ such that $\lk(y,f)$ is a stable PL map $S^{n-1}\to S^{2n-2}$ (see (b)) for each 
$y\in\Sigma$, and $f|_{\dots}\:N\but f^{-1}(\Sigma)\to M\but\Sigma$ is an immersion with an embedded curve 
$\Delta$ of transverse double points (i.e.\ points $y\in M$ such that $\lk(y,f)$ is PL-left-right equivalent 
to the suspension over the Hopf link, $S^0*S^{n-2}\sqcup S^0*S^{n-2}\to S^0*S^{2n-3}$).
Let us note that stable PL maps $N^3\to M^5$ may be locally knotted at points of an embedded finite graph 
$G\subset N\but f^{-1}(\Delta\cup\Sigma)$.

(c$'$) If $(n,m)=(2,3)$, similarly $f$ is stable if and only if $M$ contains disjoint finite subsets $\Sigma$ 
and $T$ such that $\lk(y,f)$ is a stable PL map $S^1\to S^2$ for each $y\in\Sigma$ and is PL-left-right equivalent 
to the Borromean ornament (see \cite{M5}) for each $y\in T$; and 
$f|_{\dots}\:N\but f^{-1}(\Sigma\cup T)\to M\but(\Sigma\cup T)$ is an immersion with an embedded curve 
$\Delta$ of transverse double points.

(c$''$) If $(n,m)=(1,1)$, it is easy to see that $f$ is stable if and only if $N$ contains a finite subset 
$S$ such that $f|_S$ is an embedding and $f|_{N\but S}$ is an immersion.

(d) If $(n,m)=(2,2)$, it is not hard to see that $f$ is stable if and only if $N$ contains an embedded finite
graph $G$ with vertex set $V$ such that $f|_V$ is an embedding, $f|_{G\but V}$ is a stable PL map into
$M\but f(V)$ (see (b)), $f|_{N\but G}$ is an immersion, and $\lk(x,f)$ is a stable PL map $S^1\to S^1$ (see 
(c$''$)) for each $x\in V$ and is PL-left-right equivalent to the suspension over some map $S^0\to S^0$ 
for each $x\in G\but V$.
\end{example}

\begin{example}
Let $P^n$ be a polyhedron and $f\:P\to\R$ a PL map.
For each $x\in P$ let $L^+(x)=\lk\big(x,f^{-1}([y,\infty))\big)$ and
$L^-(x)=\lk\big(x,f^{-1}((-\infty,y])\big)$, where $y=f(x)$.
Thus $L^+(x)\cup L^-(x)=\lk(x,P)$ and $L^+(x)\cap L^-(x)$ coincides with
$L^0(x):=\lk\big(x,f^{-1}(y)\big)$.
Let us call $f$ {\it link-regular} at $x$ if $L^0(x)$ is collared in $L^+(x)$ and in $L^-(x)$, 
or in other words if $f|_{\lk(x,P)}$ is PL transverse to $\{f(x)\}$.
Let us call $f$ {\it regular} at $x$ if each of $L^+(x)$ and $L^-(x)$ is PL homeomorphic to the cone over 
$L^0(x)$ keeping $L^0(x)$ fixed, or equivalently (see \cite{M4}*{12.3}) if $f|_{\st(x,P)}$ is PL transverse 
to $\{f(x)\}$.
Let us call $f$ a {\it link-submersion} (resp.\ a {\it PL submersion}) if it is link-regular (resp.\ regular)
at all points $x\in P$.

It is not hard to see that the following conditions are equivalent for a PL map $P\to\R$, where $P$ is
a compact polyhedron:
\begin{enumerate}
\item $f$ is stable;
\item $f$ is a link-submersion and $P$ contains a finite set $S$ such that $f|_S$ is 
an embedding and $f|_{P\but S}$ is a PL submersion;
\item $f$ is PL-left-right equivalent to a linear map $K\to\R$ that embeds every $1$-simplex.
\end{enumerate}
Stable maps $P\to\R$ have been used in discrete differential geometry \cite{Ba} and in discrete 
Morse theory \cite{Bl}.
\end{example}

\begin{remark} The theory of stable PL maps, as outlined above, is clearly in a very early stage of
development.
In particular, implicit in the definition of a stable PL map are several questions about the status
of the obvious modifications of this definition.
One may also wonder about a multi-$0$-jet transversality theorem (cf.\ \S\ref{2multi0jet-transversality} below) 
for stable PL maps.
Presumably this situation has some similarity with the state of the theory of stable smooth maps in the times 
of Whitney and Thom, i.e.\ before Boardman, Mather and others who brought more clarity to that field.
\end{remark}

\section{Extended $2$-Multi-$0$-Jet Transversality} \label{sec-xx}

The main result of this Appendix, Corollary \ref{extended theorem}, addresses the case $s=2$, $r=0$ of the informal problem 
raised by C. T. C. Wall in 1971:
``Is there any natural way to `fill the hole along the diagonal' [in the $s$-multi-$r$-jet transversality theorem] and get
a transversality theorem in the completed space (so one could, for example,
combine injectivity [of a generic smooth map $N^n\to M^m$ in the range $m\ge 2n+1$]
with the immersion result [in the range $m\ge 2n$])?'' \cite{Wa1}*{p.~192}.

One special case of our Corollary \ref{extended theorem} was essentially proved by F. Ronga, but our proof 
of this case is somewhat different; in fact, we give two proofs somewhat different from each other and from Ronga's
(see Theorem \ref{xx-0} and subsequent remarks).

An application of another special case of Corollary \ref{extended theorem} is obtained in \S\ref{lifting lemma},
and is one of the results of Appendix \ref{sec-xx} that are applied in the main part of the paper.

The general case of Wall's problem will hopefully be addressed in a future paper by the author.
A related problem is also mentioned in a 1976 paper by Wall \cite{Wa2}*{p.~761}.

\subsection{Weakly generic maps}\label{w-gen}
A subset of a topological space $X$ is called {\it massive} if it is a countable intersection of dense open sets.
By Baire's category theorem, a massive subset of a complete metric space is dense.
In particular, since $C^\infty(N,M)$ with the metric (compact-open-like) topology is completely metrizable 
(see \cite{Hi}*{discussion following 3.4.4}), its massive subsets are dense in it.
Although $C^\infty(N,M)$ with the strong topology (also known as the Whitney topology or the Mather topology) 
is not metrizable if $N$ is non-compact, its massive subsets are also dense in it 
(see \cite{Hi}*{discussion following 3.4.4}, \cite{GG}*{II.3.3}).
Let us note that for compact $N$, the two topologies on $C^\infty(N,M)$ coincide (see \cite{Hi}).
A key technical advantage of the strong topology is that if $W$ is a closed subset of $M$, then
the set of smooth maps $N\to M$ that are transverse to $W$ is open and dense in the strong topology,
but only massive in the metric topology (see \cite{Hi}*{3.2.1}). 
The same applies to jet transversality (see \cite{Hi}*{3.2.8 and Exercise 3.8(b)}).

The assertion ``every generic smooth map $N\to M$ satisfies property $P$'' (or any logically equivalent 
assertion) will mean ``$C^\infty(N,M)$ with the strong topology contains an open dense subset whose 
elements satisfy property $P$''.
The assertion ``every weakly generic smooth map $N\to M$ satisfies property $P$'' (or any logically equivalent 
assertion) will mean ``$C^\infty(N,M)$ with the strong topology contains a massive subset whose elements satisfy 
property $P$''.%
\footnote{``Weakly generic'' corresponds to ``generic'' in the terminology of Gromov's book \cite{Gr}*{1.3.2(B)}.}
The choice of strong topology has only technical significance, because in the end we are only interested in 
the case of compact $N$.

\subsection{$2$-Multi-$0$-Jet Transversality Theorem} \label{2multi0jet-transversality}
Let $N$ be a closed smooth $n$-manifold, $M$ a smooth $m$-manifold, $m\ge n$, and let $f\:N\to M$ be 
a smooth map.
The graph $\Gamma_f\:N\to N\x M$, defined by $x\mapsto\big(x,f(x)\big)$, is an embedding.
Therefore the diagonal $\Delta_N=\{(x,x)\mid x\in N\}$ is the preimage of $\Delta_{N\x M}$ under 
$\Gamma_f\x\Gamma_f\:N^2\to (N\x M)^2$.
Let $\tilde\Gamma_f\:\tilde N\to\widetilde{N\x M}$ be the restriction of $\Gamma_f\x\Gamma_f$ to 
the deleted product $\tilde N=N\x N\but\Delta_N$.
Clearly, the image of $\tilde\Gamma_f$ lies in $\tilde N\x M\x M=\widetilde{N\x M}\but\Delta_N\x\tilde M$.

Let us recall the $2$-multi-$0$-jet Transversality Theorem \cite{GG}*{II.4.13}: if $L$ is a smooth submanifold of 
$\tilde N\x M\x M$ and $f\:N\to M$ is a weakly%
\footnote{If $L$ is closed and $L\subset K\x M$, where $K\subset\tilde N$ is compact, then ``weakly'' can be 
omitted (see \cite{GG}*{Proof of Lemma II.4.14}).
For every compact $N$ (in which case the metric and strong topologies coincide) it is easy to construct an $L$ 
(with ``non-asymptotic'' behavior near $\Delta_N\x M$) such that the set of all $f\:N\to M$ for which 
$\tilde\Gamma_f$ is transverse to $L$ is not open in $C^\infty(N,M)$.} 
generic smooth map, then $\tilde\Gamma_f$ is transverse to $L$.
An immediate consequence of this theorem is that $\Delta_f$ is a smooth $(2n-m)$-manifold; indeed, 
$\Delta_f=\tilde\Gamma_f^{-1}(L)$, where $L=\tilde N\x\Delta_M$.
Moreover, $f$ is {\it self-transverse}, that is, the restriction of $f\x f\:N^2\to M^2$ to $\tilde N$ 
is transverse to $\Delta_M$, or, equivalently, for any two distinct points $x,y\in N$ with $f(x)=f(y)$
the tangent space $T_{f(x)}M$ is generated by $df_x(T_xN)$ and $df_y(T_yN)$.

It is necessary to restrict $\Gamma_f\x\Gamma_f$ to $\tilde N$ in the $2$-multi-$0$-jet Transversality Theorem.
Indeed, $\Gamma_f\x\Gamma_f$ is not transverse to $\bar L:=N^2\x\Delta_M$ unless $f$ is an immersion, since 
$(\Gamma_f\x\Gamma_f)^{-1}(\bar L)=\Delta_f\cup\Delta_N$ is not a manifold unless $f$ is an immersion.
(Here $\Delta_f=\{(x,y)\in\tl N\mid f(x)=f(y)\}$.)
One way of explaining this failure is that $N^2$ is, in a sense, a wrong compactification of $\tilde N$. 

\subsection{Fulton--MacPherson and Axelrod--Singer compactifications of $\tilde N$}\label{FM-AS}
Let $N$ be a closed smooth manifold and let $\tau\:N^2\to N^2$ be the factor exchanging involution.
Let $R$ be a $\tau$-invariant tubular neighborhood of $\Delta_N$ in $N^2$, and let $\nu\:R\to\Delta_N$ be 
an equivariant normal bundle projection.
The closure $Q$ of $N^2\but R$ in $N^2$ is a manifold with boundary, and $\nu$ extends to an equivariant
smooth map $\phi\:N^2\to N^2$ that restricts to a homeomorphism between $Q\but\partial Q$ and $\tilde N$.
Hence $\tilde N$ is the interior of a $\Z/2$-manifold $(\check N,\check\tau)$ which is equivariantly 
homeomorphic to $Q$.
The quotient of $\check N$ by the restriction of $\check\tau$ to $\partial\check N$ is a closed manifold 
$\hat N$.

The manifolds $\check N$ and $\hat N$ are well-defined up to equivariant homeomorphism keeping $\tilde N$
fixed (indeed, $\hat N$ is nothing but the blowup of $N^2$ along $\Delta_N$) and are special cases of the
Axelrod--Singer and Fulton--MacPherson compactifications of configuration spaces (see \cite{FM}, \cite{Sin});
in our case of interest they were known long before (see \cite{Ro}).
Since $\nu$ is isomorphic to the tangent disk bundle of $N$, the Axelrod--Singer corona $\check N\but\tilde N$ 
is homeomorphic to the total space $SN$ of the spherical tangent bundle of $N$, and the Fulton--MacPherson 
corona $\hat N\but\tilde N$ is homeomorphic to the total space $PN$ of the projective tangent bundle of $N$.
In fact, using $\phi$, we obtain the commutative diagram
\[\begin{CD}
SN@>>>\check N\\
@VVV@VVV\\
PN@>>>\hat N\\
@VVV@VVV\\
N@>>>N^2.
\end{CD}\]
To avoid excessive notation, we will identify $\check N\but\tilde N$ with $SN$ and
$\hat N\but\tilde N$ with $PN$.

A continuous curve $\gamma\:[-1,1]\to N^2$ with $\gamma^{-1}(\Delta_N)=0$ lifts to a continuous curve 
$\hat\gamma\:[-1,1]\to\hat N$ if and only if $\gamma$ is differentiable at $0$; and 
$\hat\gamma(0)=\big(\gamma(0),\left<d\gamma_0(1)\right>\!\big)\in PN$.

If $N$ is a non-compact smooth manifold without boundary, all of the above applies, except that 
$\check N$ and $\hat N$ cannot be called ``compactifications'' (but they can be called {\it completions},
as long as we fix some complete Riemannian metric on $N$ or at least a uniform equivalence class
of such metrics).
If $N$ is a smooth manifold with boundary, one can define $\check N$ and $\hat N$ by considering the 
double of $N$, that is, $N\cup_{\partial N}N$.

\subsection{Complete self-transversality} Given a smooth map $f\:N\to M$ between smooth manifolds, let 
$\hat\Gamma_f\:\hat N\to\widehat{N\x M}$ be the extension of $\tilde\Gamma_f\:\tilde N\to\widetilde{N\x M}$ 
by means of the ``projective differential'' $PN\to P(N\x M)$, given by 
\[(x,\ell)\mapsto \Big(\big(x,f(x)\big),\big<v+df_x(v)\mid v\in\ell\big>\Big),\tag{$*$}\]
where $x\in N$ and $\ell$ is a $1$-subspace of $T_xN$.

The closure of our old friend $\tilde N\x\Delta_M$ in $\widehat{N\x M}$ can be written as $\hat N\x\Delta_M$
and intersects $P(N\x M)$ in $PN\x M$.
The preimage $\hat\Gamma_f^{-1}(PN\x M)$ coincides with the following subset of $PN$,
\[\hat\Sigma_f:=\big\{\left<v\right>\subset T_xN\mid x\in N,\,v\in\ker df_x\big\}.\]
Consequently $\hat\Gamma_f^{-1}\big(\hat N\x\Delta_M\big)$ coincides with the following subset of $\hat N$,
\[\hat\Delta_f:=\Delta_f\cup\hat\Sigma_f.\]

The map $f\:N\to M$ is called {\it completely self-transverse} if it is smooth and $\hat\Gamma_f$ is 
transverse to $\hat N\x\Delta_M$.
(The word ``completely'' may remind us of the completions.)
Thus if $f$ is completely self-transverse, $\hat\Delta_f$ is a submanifold of $\hat N$.

\begin{theorem} \label{xx-0} 
Let $N^n$ and $M^m$ be smooth manifolds, where $m\ge n$ and $N$ is compact.
Then the set of completely self-transverse maps is open and dense in $C^\infty(N,M)$.
\end{theorem}

Theorem \ref{xx-0} is an easy consequence of a result by Ronga \cite{Ro}*{2.5(i)} (see also \cite{Ro2}).
As the present author was unaware of Ronga's work, Theorem \ref{xx-0} was also announced in 
\cite{M1}*{comments after Proposition 1}.
Its proof given below is extracted from an unfinished 2004 manuscript by the author and is somewhat different 
from Ronga's proof (Ronga's proof is more explicit, whereas ours is coordinate free), although both proofs 
quote the same cases of the Jet Transversality Theorem.

In \S\ref{extended} below we also prove a generalization of Theorem \ref{xx-0}, whose proof on the one hand 
reuses much of the proof of Theorem \ref{xx-0}, but on the other hand provides alternative arguments elsewhere.
This resulting third proof of Theorem \ref{xx-0} is logically shorter in that it builds directly on 
the proof of the Jet Transversality Theorem instead of being content with sorting out its consequences.

\subsection{Full $1$-transversality}
Let $\P_N\:PN\to N$ denote the projectivization of the tangent bundle $\T_N\:TN\to N$, and let $\gamma$ 
be the tautological line bundle over $PN$, which is associated with its double covering by 
the total space $SN$ of the spherical tangent bundle.
Then $\gamma$ can be viewed as a subbundle of $\P_N^*\T_N$, and the differential $df\:\T_N\to\T_M$ 
yields a map of bundles $df^{PN}\:\P_N^*\T_N\to\P_N^*f^*\T_M$ over $PN$.
The restriction of $df^{PN}$ to $\gamma$ can be regarded as a section of the Hom-bundle 
$\eta:\Hom(\gamma,\P_N^*f^*\T_M)\to PN$.
The zero set of this section $s_f\:PN\to E(\eta)$ clearly coincides with $\hat\Sigma_f$.
If $s_f$ is transverse to the zero section, the map $f$ is called {\it fully $1$-transverse},
following Porteous \cite{Por}.
Thus if $f$ is fully $1$-transverse, $\hat\Sigma_f$ is a submanifold of $PN$.

\begin{lemma}\label{xx-1} $f$ is completely self-transverse if and only if it is self-transverse and 
fully $1$-transverse.
\end{lemma}

\begin{proof} Clearly, $f$ is self-transverse if and only if $\hat\Gamma_f$ is transverse 
to $\tilde N\x\Delta_M$.

On the other hand, let $p_1$, $p_2$ denote the projections onto the factors of $N\x M$.
Since $p_1\Gamma_f=\id_N$ and $p_2\Gamma_f=f$, we have $\Gamma_f^*p_1^*\P_N=\P_N$ and 
$\Gamma_f^*p_2^*\P_M=f^*\P_M$.
With these identifications, $U:=E(\Gamma_f^*\P_{N\x M})\but E(f^*\P_M)$ is an open tubular neighborhood of
$E(\P_N)$ in $E(\Gamma_f^*\P_{N\x M})$.
Since the normal bundle of $E(\T_N)$ in $E(\Gamma_f^*\T_{N\x M})=E(\T_N\oplus f^*\T_M)$ is isomorphic to 
$E(\T_N^*f^*\T_M)$, and the normal bundle of $\R P^n$ in $\R P^{n+m}$ at $\ell\in\R P^n$ is canonically 
identified with $\Hom(\ell,\R^m)$, the normal bundle $\nu\:U\to PN$ of $E(\P_N)$ in $E(\Gamma_f^*\P_{N\x M})$
can be identified with the Hom-bundle $\eta\:\Hom(\gamma,\P^*_Nf^*\T_M)\to PN$.

The bundle projection $\nu\:U\to PN$, which discards the $M$-component of the tangent line, has the zero 
section $PN$ as well as the section $\hat\Gamma_f|_{PN}\:PN\to i^*\big(P(N\x M)\but p_2^*PM\big)$, where 
$i\:\Gamma_f(N)\to N\x M$ is the inclusion.
Under the above identification of $\nu$ with $\eta$ this section $\hat\Gamma_f|_{PN}$, which is given 
by the formula ($*$), gets identified with $s_f$.
Thus $f$ is fully $1$-transverse if and only if
$\hat\Gamma_f|_{PN}\:PN\to i^*P(N\x M)$ is transverse to $i^*p_1^*PN$.
Since $\hat N\x\Delta_M$ meets $P(N\x M)$ transversely in $PN\x M$, which in turn meets $i^*P(N\x M)$ 
transversely in $i^*p_1^*PN$, the latter is equivalent to saying that $\hat\Gamma_f$ is transverse to
$\hat N\x\Delta_M$ at each point of $PN\x M$. 
\end{proof}

\subsection{1-Jet Transversality Theorem}\label{1-jets}
The space $J^1(N,M)$ of $1$-jets from $N$ to $M$ is the total space of the vector bundle
$\J^1(N,M)\:\Hom(p_1^*\T_N,p_2^*\T_M)\to N\x M$, where $p_1$, $p_2$ denote the projections onto the factors 
of $N\x M$.
The differential $df$ can be viewed as a section $s_{df}\:N\to J^1(N,M)$ of the bundle $J^1(N,M)\to N\x M\to N$.
The $1$-jet Transversality Theorem (see \cite{Hi}*{3.2.8}, \cite{GG}*{II.4.9}) says that if $L$ is 
a smooth submanifold of $J^1(N,M)$ and $f$ is a weakly%
\footnote{If $L$ is closed, then ``weakly'' can be dropped (see \cite{Hi}*{3.2.8}, 
\cite{GG}*{II.3.4 and II.4.5}).} 
generic smooth map $N\to M$, then $s_{df}$ is transverse to $L$.
An immediate consequence of this theorem is that for a weakly generic $f\:N\to M$, 
the set $\Sigma^i_f:=\{x\in N\mid\dim(\ker df_x)=i\}$ is 
a (not necessarily closed) smooth submanifold in $N$ of codimension $i(m-n+i)$, as long as $m\ge n$.
Indeed, $\Sigma^i_f=s_{df}^{-1}(\Sigma^i)$, where $\Sigma^i\subset J^1(N,M)$ is the (generally non-closed) 
submanifold of all linear maps $L\:T_xN\to T_yM$, $(x,y)\in N\x M$, of rank $r:=n-i$; and a fiber of 
the normal bundle of $\Sigma^i$ at some $L\in\Sigma^i\cap\Hom(T_xN,T_yM)$, can be identified 
(see \cite{GG}*{proof of II.5.3}) with the coset $C_L:=L+\Hom(\ker L,\coker L)\subset\Hom(T_xN,T_yM)$, 
which has dimension $(n-r)(m-r)$.
Following Porteous \cite{Por}, we will say that $f$ is {\it $i$-transverse} if $s_{df}$ is transverse to 
$\Sigma^i$.
Note that this condition is only non-vacuous for finitely many values of $i$, namely when $i(m-n+i)\le n$.

\begin{lemma}\label{xx-2} \cite{Ro}*{2.2} If $f$ is $i$-transverse for all $i$, then it is fully $1$-transverse.
\end{lemma}

Lemma \ref{xx-2} and the 1-jet Transversality Theorem imply that weakly generic maps are fully $1$-transverse.
One application of Lemma \ref{xx-2} is discussed in \cite{Mc}.

\begin{proof} The section $s_f\:PN\to\Hom(\gamma,\P_N^*f^*\T_M)$ is transverse to the zero section 
if and only if the section $\hat s_f\:PN\to\Hom(\P_N^*\T_N,\P_N^*f^*\T_M)$, given by $df^{PN}$, 
is transverse to the subbundle $\Pi:=\Hom(\P_N^*\T_N/\gamma,\P_N^*f^*\T_M)$.
On the other hand, $s_{df}\:N\to J^1(N,M)$ factors through a section of the bundle
$\Gamma_f^*\J^1(N,M)\:\Hom(\T_N,f^*\T_M)\to N$.
Now $\hat s_f$ is the induced section of the induced bundle $\P_N^*\Hom(\T_N,f^*\T_M)\to PN$, and 
it follows that $s_{df}$ is transverse to $\Sigma^i$ if and only if $\hat s_f$ is transverse to
$\hat\Sigma^i:=\P_N^*\Gamma_f^*(\Sigma^i)$.

We have $\Pi\subset\bigcup_i\hat\Sigma^i$, and each $\Pi^i:=\Pi\cap\hat\Sigma^i$ is a (not necessarily closed) 
submanifold of $\P_N^*\Hom(\T_N,f^*\T_M)$.
By the definition of transversality, to prove that $\hat s_f$ is transverse to $\Pi$ it suffices to show 
that $\hat s_f$ is transverse to each $\Pi^i$.
By the above, $\hat s_f$ is transverse to each $\hat\Sigma^i$ since $f$ is assumed to be $i$-transverse.
So it remains to show that $\hat s_f$ restricted to each 
$\hat\Sigma^i_f:=\hat s_f^{-1}(\hat\Sigma^i)=\P_N^{-1}(\Sigma^i_f)$ is transverse to $\Pi^i$ (as a map into 
$\hat\Sigma^i$).

Let $\ell$ be a $1$-subspace of $T_xN$ for some $x\in N$, and let $L=\hat s_f(\ell)$.
Assuming that $L\in\Pi^i$, we have $\dim(\ker df_x)=i$ and $\ell\subset\ker df_x$.
Let $K_\ell$ denote the tangent space at $\ell$ to the fiber of $\P_N$ over $x$, that is, $K_\ell$ is 
the kernel of $d(\P_N)_\ell\:T_\ell(PN)\to T_xN$.
We claim that $d(\hat s_f)_\ell(K_\ell)$ and $T_L\Pi^i$ span $T_L\hat\Sigma^i$.
Indeed, direct computation shows that $\Pi^i$ has codimension $r$ in $\hat\Sigma^i$.%
\footnote{Namely, the fiber of $\hat\Sigma^i$ over $(x,\ell)\in PN$ consists of $n\x m$ matrices of rank $r$, 
and so has dimension $mn-(m-r)(n-r)$.
The fiber of $\Pi^i$ over $(x,\ell)$ consists of $(n-1)\x m$ matrices 
of rank $r$, and so has dimension $m(n-1)-(m-r)(n-1-r)=mn-(m-r)(n-r)-r$.}
On the other hand, $K_\ell$, which is a tangent space of $\R P^{n-1}$, can be identified with 
$\Hom(\ell,T_xN/\ell)$.
If $v_1,\dots,v_r\in K_\ell$ are representatives of a basis for $K_\ell/\Hom(\ell,\ker df_x/\ell)$, 
where $r=n-i$ is the rank of $df_x$, we claim that the $r$ vectors $d(\hat s_f)_\ell(v_i)+T_L\Pi^i$
are linearly independent in $T_L\hat\Sigma^i/T_L\Pi^i$.
Indeed, every nontrivial linear combination $\sum_j\alpha_jd(\hat s_f)_\ell(v_j)$ can be written as
$d(\hat s_f)_\ell(v)$, where $v=\sum_j\alpha_jv_j\ne 0$.
This $v$ is the tangent vector to a great circle arc from $\ell$ to a nearby point $\ell'\in\R P^{n-1}$ such that 
$\ell'\not\subset\ker df_x$.
Then $\hat s_f(\ell')\notin\Pi^i$, and it follows that $d(\hat s_f)_\ell(v)\notin T_L\Pi_i$.
\end{proof}

\begin{proof}[Proof of Theorem \ref{xx-0}]
Let $S$ be the set of all maps $N\to M$ that are self-transverse and $i$-transverse for all $i$
such that $i(m-n+i)\le n$.
By the 1-Jet Transversality and the 2-Multi-0-Jet Transversality theorems (see \cite{GG}) $S$ is massive,
and in particular dense in $C^\infty(N,M)$.
By Lemmas \ref{xx-1} and \ref{xx-2} $S$ lies in the set $T$ of all completely self-transverse maps $N\to M$.
Hence $T$ is also dense in $C^\infty(N,M)$.

On the other hand, the set $U$ of all maps $\hat N\to\widehat{N\x M}$ that are transverse to 
$\hat N\x\Delta_M$ is open in $C^\infty(\hat N,\widehat{N\x M})$ (see e.g.\ \cite{GG}*{Proposition II.4.5}).
By definition, $T=\phi^{-1}(U)$, where $\phi\:C^\infty(N,M)\to C^\infty(\hat N,\widehat{N\x M})$
is given by $f\mapsto\hat\Gamma_f$.
It is easy to see that $\phi$ is continuous. Hence $T$ is open. 
\end{proof}

\subsection{Behavior of $\hat\Delta_f$ at $\hat\Sigma_f$}

If $f$ is a completely self-transverse map, then $\hat\Delta_f$ is a submanifold of $\hat N$; and also its
intersection $\hat\Sigma_f$ with $PN$ is a submanifold of $PN$ by Lemma \ref{xx-1}.
The following lemma shows that this is not merely a coincidence.

\begin{lemma} \label{blowup} $\hat\Gamma_f$ is transverse to $P(N\x M)$ for every smooth map $f\:N\to M$.
\end{lemma}

\begin{proof} We have $\hat\Gamma_f^{-1}\big(P(N\x M)\big)=PN$.
For each $(x,\ell)\in PN$ we can find a smooth curve $\gamma\:\R\to N$ such that $\gamma(0)=x$ and 
$\left<\gamma'(0)\right>=\ell$.
Then $\delta\:\R\to N^2$ given by $\delta(t)=\big(\gamma(t),\gamma(-t)\big)$ is a smooth curve which
lifts to a smooth curve $\hat\delta\:\R\to\hat N$ with $\hat\delta(0)=(x,\ell)$.
Since $\delta$ is not tangent to $\Delta_N$ at $\delta(0)$ and $\delta=p\Gamma_f^2\delta$, where 
$p\:(N\x M)^2\to N^2$ is the projection, $\Gamma_f^2\delta$ is not tangent to $\Delta_{N\x M}$ at 
$\Gamma_f^2\delta(0)$.
Consequently its lift $\widehat{\Gamma_f\delta}=\hat\Gamma_f\hat\delta\:\R\to\widehat{N\x M}$ 
is not tangent to $P(N\x M)$ at $\hat\Gamma_f\hat\delta(0)$.
This suffices, since $P(N\x M)$ has codimension one in $\widehat{N\x M}$.
\end{proof}

\begin{corollary}\label{identification} If $f\:N\to M$ is a completely self-transverse map, then 

(a) $\hat\Delta_f$ intersects $PN$ transversely;

(b) $\hat\Delta_f$ coincides with the closure of $\Delta_f$ in $\hat N$;

(c) the image of $\hat\Delta_f$ in $N^2$ coincides with $\bar\Delta_f$;

(d) the image of $\hat\Sigma_f$ in $N$ coincides with $\Sigma_f$.
\end{corollary}

Here $\bar\Delta_f$ is the closure of $\Delta_f$ in $N\x N$.

\begin{proof} Obviously, Lemma \ref{blowup} $\Rightarrow$ (a) $\Rightarrow$ (b) $\Rightarrow$ (c) 
$\Rightarrow$ (d).
\end{proof}

\begin{corollary} \label{smooth-delta} If $f\:N\to M$ is a corank one completely self-transverse map, 
then $\bar\Delta_f$ is a smooth submanifold of $N\x N$ and $\Sigma_f$ is a smooth submanifold of $N$.
\end{corollary}

Let us note that although $\bar\Delta_f$ intersects $\Delta_N$ in the submanifold $\Delta_{\Sigma_f}$, 
this intersection is never transverse for $n>1$, since $\hat\Sigma_f$ has codimension one in $\hat\Delta_f$.

It is a well-known folklore result that $\bar\Delta_f$ is a smooth submanifold of $N\x N$ and $\Sigma_f$ 
is a smooth submanifold of $N$ for every stable%
\footnote{Let us note that stable maps are completely self-transverse, since the latter are dense in $C^\infty$
by Theorem \ref{xx-0} and self-transversality is preserved by $C^\infty$-left-right equivalence.}
corank one map $f$.
It can be derived from Morin's canonical forms \cite{Mo} for stable corank one germs 
(cf.\ \cite{Hou}*{\S9}); see also \cite{MM}*{2.14(i) and 2.16} for a complex algebraic version with $m>n$.

\begin{proof} Since $f$ is completely self-transverse, $\hat\Delta_f$ is a smooth submanifold of $\hat N$.
Also, Lemma \ref{xx-1} implies that $\hat\Sigma_f$ is a smooth submanifold of $PN$ and 
$\P_N\:PN\to N$ restricts to a submersion on $\hat\Sigma_f$.
On the other hand, since $f$ is a corank one map, $\P_N\:PN\to N$ is injective on $\hat\Sigma_f$,
and consequently the blowdown map $\pi\:\hat N\to N^2$ is injective on $\hat\Delta_f$.
It follows that $\P_N$ restricts to a smooth embedding on $\hat\Sigma_f$, and, using Corollary 
\ref{identification}(a), that $\pi$ restricts to a smooth embedding on $\hat\Delta_f$.
Finally, by Corollary \ref{identification}(c,d) the images of these two embeddings coincide with 
$\Sigma_f$ and $\bar\Delta_f$.
\end{proof}

\begin{lemma} \label{smoothcurve}
If $f\:N\to M$ is a completely self-transverse map, then for each pair $(x,\ell)\in\hat\Sigma_f$ 
there exists a smooth curve $\gamma\:\R\to N$ such that $\gamma(0)=x$, $\left<\gamma'(0)\right>=\ell$
and $f\big(\gamma(t)\big)=f\big(\gamma(-t)\big)$ for each $t\in\R$.
\end{lemma}

\begin{proof} $\hat\Sigma_f$ is the fixed point set of the smooth involution on $\hat\Delta_f$,
which at each point of $\hat\Sigma_f$ is locally equivalent to the orthogonal reflection 
$(x_1,x_2,\dots,x_k)\mapsto (-x_1,x_2,\dots,x_k)$ in $\R^k$, where $k=2n-m$.
Hence there exists an equivariant (with respect to the involution $x\mapsto -x$ on $\R$) smooth curve 
$\beta\:\R\to\hat\Delta_f$ such that $\beta(0)=(x,\ell)$.
Let $\pi\:\hat N\to N^2$ be the blowdown map and $p\:N^2\to N$ the projection onto the first factor.
Let $\gamma=p\pi\beta$.
Then $\pi\beta(t)=\big(\gamma(t),\gamma(-t)\big)\in\bar\Delta_f$, so
$f\big(\gamma(t)\big)=f\big(\gamma(-t)\big)$ for each $t\in\R$.
Also $\ell=\left<\lim_{t\to 0}\frac{\gamma(t)-\gamma(-t)}{2t}\right>=\left<\gamma'(0)\right>$.
\end{proof}

\subsection{Extended Gauss map} \label{egauss}
If $f\:N\to M$ is a completely self-transverse map between manifolds without
boundary, then clearly $\check\Delta_f$ is a smooth manifold whose boundary is a double cover 
of $\hat\Sigma_f$.

\begin{lemma}\label{3.1} 
Let $N$, $M$ be smooth manifolds, $f\:N\to M$ a completely self-transverse map and $G\:N\to\R^k$
a smooth map such that $f\x G\:N\to M\x\R^k$ is a smooth embedding.
Then $\check G\:\check\Delta_f\to S^{k-1}\x[0,\infty)$, defined by
$\check G(x,y)=\Big(\frac{G(y)-G(x)}{||G(y)-G(x)||},\,||G(y)-G(x)||\Big)$ for 
$(x,y)\in\Delta_f$ and
by $\check G(x,v)=\Big(\frac{dG_x(v)}{||dG_x(v)||},\,0\Big)$ for $(x,v)\in\check\Sigma_f$,
is a smooth map.
\end{lemma}

\begin{proof} 
Let us first consider the case where $N=M\x\R^k$ and $f$ and $G$ are the projections 
$p\:M\x\R^k\to M$ and $q\:M\x\R^k\to\R^k$ so that $f\x G=\id_{M\x\R^k}$.
It is easy to see that $\check q$ (defined similarly to $\check G$) is the composition of the natural map 
$\check\Delta_p\to\check\Delta_c$, where $c\:\R^k\to\{0\}$ is the constant map,
and the obvious retraction of $(\R^k)\check{\,}$ onto its anti-diagonal.
In the general case, $f\x G$ induces a smooth embedding $\check N\emb (M\x\R^k)\check{\,}$, which 
in turn restricts to a smooth embedding $G_*\:\check\Delta_f\to\check\Delta_p$. 
Now $G$ factors into the composition $\check\Delta_f\xr{G_*}\check\Delta_p\xr{\check q}S^{k-1}\x[0,\infty)$.
\end{proof}

\subsection{The extended $2$-multi-$0$-jet transversality theorem} \label{extended}

\begin{theorem} \label{projective differential}
Let $N$ and $M$ be smooth manifolds and let $L$ be a smooth submanifold of $P(N\x M)\but N\x PM$.
Let $X_L$ be the set of smooth maps $f\:N\to M$ whose ``projective differential'' 
$\hat\Gamma_f|_{PN}\:PN\to P(N\x M)\but N\x PM$ is transverse to $L$.
Then $X_L$ is massive in $C^\infty(N,M)$ with the strong topology; if $L$ is closed, then ``massive'' can be replaced with
``open and dense''.
\end{theorem}

\begin{proof} By the proof of Lemma \ref{xx-1}, $\hat\Gamma_f|_{PN}$ is identified with the section $s_f$ 
of the bundle $\Hom(\gamma,\P_N^*f^*\T_M)\to PN$ given by the composition 
$\gamma\subset\P_N^*\T_N\xr{df^{PN}}\P_N^*f^*\T_M$.
In particular, $L$ is identified with a submanifold of the total space $\Hom(\gamma,\P_N^*f^*\T_M)$.
Let $\bar L$ be the preimage of $L$ in $\Hom(\P_N^*\T_N,\P_N^*f^*\T_M)$ under the restricting map.
Clearly, $s_f$ is transverse to $L$ if and only if the section $\hat s_f\:PN\to\Hom(\P_N^*\T_N,\P_N^*f^*\T_M)$, 
given by $df^{PN}$, is transverse to $\bar L$.
On the other hand, $s_{df}\:N\to J^1(N,M)$ can be identified with a section of the bundle
$\Gamma_f^*\J^1(N,M)\:\Hom(\T_N,f^*\T_M)\to N$, and $\hat s_f$ is the induced section of 
the induced bundle $\P_N^*\Hom(\T_N,f^*\T_M)\to PN$.

By the proof of the Jet Transversality theorem \cite{GG}*{II.4.9} (see also \cite{Hi}*{3.2.8}),
for each $w\in N$ there exists a neighborhood $B$ of the origin in the vector space $\Hom(\R^n,\R^m)$ 
of all linear maps $\R^n\to\R^m$ and a smooth homotopy $g_b\:N\to M$, $b\in B$, such that $g_0=f$,
each $g_b|_{N\but U}=f|_{N\but U}$ for some fixed compact neighborhood $U$ of $w$, and the smooth map 
$\Phi\:N\x B\to\Hom(\T_N,f^*\T_M)$ defined by $\Phi(x,b)=(dg_b)_x$ sends some neighborhood of $(w,0)$ 
by a diffeomorphism onto some neighborhood of $df_w$.
Then the induced map $P\Phi\:PN\x B\to\P_N^*\Hom(\T_N,f^*\T_M)$, defined by 
$P\Phi(x,\ell,b)=\big(x,\ell,\Phi(x,b)\big)=(dg_b)^{PN}_{x,\ell}$, sends some neighborhood of $(w,\ell,0)$ by 
a diffeomorphism onto some neighborhood of $df^{PN}_{w,\ell}$ for each $\ell\in P_wN$.
In particular, $P\Phi$ is transverse to $\bar L$ at $df^{PN}_{w,\ell}$ for each $\ell\in P_wN$.
Let us note that $s_{dg_b}(x)=\Phi(x,b)$ and $\hat s_{g_b}(x,\ell)=P\Phi(x,\ell,b)$.
By \cite{GG}*{II.4.7} we get that $B$ contains a dense subset $B_0$ such that
$\hat s_{g_b}$ is transverse to $\bar L$ for each $b\in B_0$.
Let $b_1,b_2,\dots\in B_0$ converge to $0$, and set $f_i=g_{b_i}$.
Then each $\hat s_{f_i}$ is transverse to $\bar L$ and $f_i\to f$ in the strong topology due to
$f_i|_{N\but U}=f|_{N\but U}$.
Thus $X_L$ is dense.

Let $S\subset C^\infty\big(N,\Hom(\T_N,f^*\T_M)\big)$ be the subspace consisting of all sections
of the bundle $\Hom(\T_N,f^*\T_M)\to N$.
The map $C^\infty(N,M)\to S$ given by $f\mapsto s_{df}$ is continuous (see \cite{GG}*{II.3.4}),
and one can check that so is the map $S\to C^\infty\big(PN,\P_N^*\Hom(\T_N,f^*\T_M)\big)$,
given by sending every section of $\Hom(\T_N,f^*\T_M)\to N$ to the induced section of the induced bundle 
$\P_N^*\Hom(\T_N,f^*\T_M)\to PN$.
The subset of $C^\infty\big(PN,\P_N^*\Hom(\T_N,f^*\T_M)\big)$ consisting of all maps that are transverse 
to $\bar L$ is open if $\bar L$ is closed (see \cite{GG}*{II.4.5}), and is massive in general 
(see \cite{GG}*{proof of II.4.9}).
Hence the set of smooth maps $f\:N\to M$ such that $\hat s_f$ is transverse to $\bar L$ is open in
$C^\infty(N,M)$ if $\bar L$ is closed, and massive in general.
As noted above, this set coincides with $X_L$; and $\bar L$ is closed if $L$ is closed.
\end{proof}

From Theorem \ref{projective differential}, Lemma \ref{blowup} and the $2$-Multi-$0$-Jet Transversality 
Theorem we immediately obtain

\begin{corollary} \label{extended theorem}
Let $N$ and $M$ be smooth manifolds and let $L$ be a smooth submanifold of 
$\widehat{N\x M}\but\Delta_N\x\hat M$.
Let $X_L$ be the set of smooth maps $f\:N\to M$ such that $\hat\Gamma_f$ is transverse to $L$.
Then $X_L$ is massive in $C^\infty(N,M)$ with the strong topology; if $L$ is closed, then ``massive'' can be replaced with
``open and dense''.
\end{corollary}

The case $L=\hat N\x\Delta_M$ was already covered in Theorem \ref{xx-0}.
Next we note an application which needs a different $L$.

\subsection{Taking $\Sigma_f$ off a polyhedron} \label{lifting lemma}

\begin{corollary} \label{generic lifting}
Let $N$ and $M^m$ be smooth manifolds and let $Q$ be a closed subpolyhedron of $PN$.
If $\dim Q<m$, then for every generic smooth map $f\:N\to M$ the manifold $\hat\Sigma_f$ is disjoint from $Q$.
\end{corollary}

\begin{proof}
Let us fix some triangulation of $Q$.
Then $Q$ is the union of its open simplexes $\mathring\sigma_i$, which are smooth submanifolds of $PN$. 
This union is countable (and even finite if $N$ is compact).
Hence by Theorem \ref{projective differential} the set $S$ of maps $f\:N\to M$ such that $\hat\Gamma_f|_{PN}$ 
is transverse to each $L_i:=\mathring\sigma_i\x M$ is massive, and in particular dense in $C^\infty(N,M)$.
Since each $\dim L_i\le\dim Q\x M<2m$ and $\dim P(N\x M)-\dim PN=2m$, the image of $\hat\Gamma_f|_{PN}$ 
is disjoint from each $L_i$ and hence from $Q\x M$ for each $f\in S$.
In fact, $f\in S$ if and only if the image of $\hat\Gamma_f|_{PN}$ is disjoint from $Q\x M$.
We have $Q\x M=PN\x M\cap R$, where $R$ is the preimage of $Q$ under the projection 
$P(N\x M)\but N\x PM\to PN$.
Hence $f\in S$ if and only if $\hat\Gamma_f^{-1}(PN\x M)=\hat\Sigma_f$ is disjoint from $\hat\Gamma_f^{-1}(R)=Q$.

It remains to show that $S$ is open in $C^\infty(N,M)$.
Since $Q\x M$ is a closed subset of $P(N\x M)\but N\x PM$, the set of maps $PN\to P(N\x M)\but N\x PM$ whose
image is disjoint from $Q\x M$ is an open subset of $C^\infty\big(PN,\,P(N\x M)\but N\x PM\big)$.
On the other hand, the map $C^\infty(N,M)\to C^\infty\big(PN,\,P(N\x M)\but N\x PM\big)$ given by
$f\mapsto\hat\Gamma_f|_{PN}$ is continuous by the proof of Theorem \ref{projective differential}.
Hence $S$ is open.
\end{proof}

\begin{theorem} \label{taking off}
Let $f\:N^n\to M^m$ be a smooth map between smooth manifolds, where $N$ is compact,
$m\ge n$, such that $f$ is $i$-transverse for all $i$.
Let $P^p$ be a closed subpolyhedron of $N$ contained in $\Sigma_f$ and suppose that $k\ge p+j$, where 
$j=j(p)$ is the maximal number such that $p\le (j+1)n-jm-j^2$.
Then for every generic smooth map $g\:N\to\R^k$ the set $\Sigma_{f\x g}$ 
is disjoint from $P$, where $f\x g\:N\to M\x\R^k$ is the joint map.
\end{theorem}

Let us note that $j(p)\le 1$ for all $p$ due to $p\le\dim\Sigma_f=2n-m-1$.
Thus $j(p)=1$ for each $p>3n-2m-4$.

\begin{proof} The projection $\hat\Sigma_{f\x g}\to\Sigma_{f\x g}$ is surjective and is a restriction of 
the projection $\hat\Sigma_f\to\Sigma_f$.
Hence $\Sigma_{f\x g}$ is disjoint from $P$ if and only if $\hat\Sigma_{f\x g}$ is disjoint from
the preimage $Q$ of $P$ under the projection $\hat\Sigma_f\to\Sigma_f$.
Clearly, $\hat\Sigma_{f\x g}=\hat\Sigma_f\cap\hat\Sigma_g$.
So $Q$ is disjoint from $\hat\Sigma_{f\x g}$ if and only if it is disjoint from $\hat\Sigma_g$.
Thus we need to show that for every generic smooth map $g\:N\to\R^k$ the manifold $\hat\Sigma_g$ 
is disjoint from $Q$.
To prove this, by Lemma \ref{generic lifting} it suffices to show that $\dim Q<k$.

Indeed, let $Q_i$ be the preimage of $P_i:=P\cap\Sigma^i_f$ under the same projection.
We have $\dim Q_i=\dim P_i+i-1$.
For each $i\le j$ we have $\dim Q_i\le p+i-1\le p+j-1<k$.
Since $f$ is $i$-transverse, $\dim\Sigma^i_f\le n-i(m-n+i)=(i+1)n-im-i^2$ (see \S\ref{1-jets}).
In particular, $\dim\Sigma^{i+1}_f<\dim\Sigma^i_f$.
Also, the definition of $j$ implies that $p>\dim\Sigma^{j+1}_f$.
Then for each $i>j$ we have $\dim P_i\le\dim\Sigma^i_f\le\dim\Sigma^{j+1}_f-j+1-i\le p+j-i$.
Hence $\dim Q_i\le p+j-1<k$.
Thus $\dim Q<k$.
\end{proof}

\section*{Acknowledgements}

I am grateful to P. Akhmetiev, L. Funar, A. Gorelov, M. Kazarian, M. Skopenkov and V. Vassiliev for stimulating 
conversations and useful remarks, the referees for useful feedback, and A. Papadopoulos for proofreading the text.

%


\begin{thebibliography}{00}

\bib{Ad}{book}{
   author={Adachi, M.},
   title={Embeddings and Immersions},
   publisher={Iwanami Shoten},
   place={Tokyo},
   date={1984},
   language={Japanese},
   translation={
      series={Transl. Math. Monogr.},
      volume={124},
      publisher={Amer. Math. Soc.},
      date={1993},
      }
}

\bib{A0}{article}{
   author={Akhmet'ev, P. M.},
   title={Prem-mappings, triple self-interesection points of an oriented surface, 
   and Rokhlin's signature theorem},
   journal={Mat. Zametki},
   volume={59},
   date={1996},
   number={6},
   pages={803--810, 959, \href{http://mi.mathnet.ru/mz1779}{MathNet}},
   translation={
      language={English},
      journal={Math. Notes},
      volume={59},
      date={1996},
      pages={581--585},
   },
}

\bib{A1}{article}{
   author={Akhmet'ev, P. M.},
   title={On isotopic and discrete realization of maps of $n$-dimensional sphere in Euclidean space},
   journal={Mat. Sb.},
   volume={187},
   date={1996},
   number={7},
   pages={3--34; \href{http://mi.mathnet.ru/msb142}{MathNet}},
   translation={
      journal={Sb. Math.},
      volume={187},
      date={1996},
      pages={951--980},
   },
}

\bib{A2}{article}{
   author={Akhmet'ev, P. M.},
   title={Embeddings of compacta, stable homotopy groups of spheres, and singularity theory},
   journal={Uspekhi Mat. Nauk},
   volume={55},
   date={2000},
   number={3},
   pages={3--62; \href{https://www.mathnet.ru/rus/rm290}{MathNet}},
   translation={
      language={English},
      journal={Russian Math. Surveys},
      volume={55},
      date={2000},
      pages={405--462},
   },
}

\bib{A3}{article}{
   author={Akhmetiev, P. M.},
   title={Pontryagin--Thom construction for approximation of mappings by embeddings},
   journal={Topol. Appl.},
   volume={140},
   date={2004},
   pages={133--149},
   note={\href{https://www.sciencedirect.com/science/article/pii/S0166864103002876}{Journal}}
}

\bib{A4}{article}{
   author={Akhmet'ev, P. M.},
   title={A remark on the realization of mappings of the 3-dimensional sphere into itself},
   journal={Trudy MIAN},
   volume={247},
   date={2004},
   pages={10--14, \href{https://www.mathnet.ru/eng/tm6}{MathNet}},
   translation={
      language={English},
      journal={Proc. Steklov Inst. Math},
      volume={247},
      date={2004},
      pages={4--8},
   },
}

\bib{AM}{article}{
   author={Akhmetiev, P. M.},
   author={Melikhov, S. A.},
   title={Projected and near-projected embeddings},
   journal={Zap. Nauch. Sem. POMI},
   volume={498},
   date={2020},
   pages={75--104},
   note={Reprinted in J. Math. Sci. {\bf 255} (2021), 155--174; \arxiv{1711.03520}},
}

\bib{ARS}{article}{
   author={Akhmetiev, P. M.},
   author={Repov{\v{s}}, D.},
   author={Skopenkov, A. B.},
   title={Obstructions to approximating maps of $n$-manifolds into $\R^{2n}$ by embeddings},
   journal={Topol. Appl.},
   volume={123},
   date={2002},
   pages={3--14},
   note={\href{http://www.sciencedirect.com/science/article/pii/S016686410100164X}{Journal}},
}

\bib{AFT}{article}{
   author={Akitaya, H. A.},
   author={Fulek, R.},
   author={T\'oth, C. D.},
   title={Recognizing weak embeddings of graphs},
   journal={ACM Trans. on Algor.},
   volume={15},
   date={2019},
   pages={No.\ 50 (27 pp.)},
   note={\arxiv{1709.09209}},
}

\bib{AGV}{book}{
   author={Arnol'd, V. I.},
   author={Guse\u{\i}n-Zade, S. M.},
   author={Varchenko, A. N.},
   title={Singularities of Differentiable Maps. Vol. I: The Classification of Critical Points, Caustics and Wave Fronts},
   language={Russian},
   publisher={Nauka},
   date={1982},
   note={2nd Russian ed., MCCME 2004},
   translation={
    language={English},
    series={Mono. in Math.},
    volume={82},
    publisher={Birkh\"{a}user},
    place={Boston},
    date={1985},
    }
}

\bib{Ba}{article}{
   author={Banchoff, T.},
   title={Critical points and curvature for embedded polyhedra},
   journal={J. Diff. Geom.},
   volume={1},
   date={1967},
   pages={245--256},
   note={\href{https://projecteuclid.org/euclid.jdg/1214428092}{ProjectEuclid}}
}

\bib{BJ}{article}{
   author={Banchoff, T. F.},
   author={Johnson, O.},
   title={The normal Euler class and singularities of projections for polyhedral surfaces in $4$-space},
   journal={Topology},
   volume={37},
   date={1998},
   pages={419--439},
   note={\href{https://www.sciencedirect.com/science/article/pii/S0040938397000116}{Journal}}
}

\bib{Bl}{article}{
   author={Bloch, E. D.},
   title={Polyhedral representation of discrete Morse functions},
   journal={Discrete Math.},
   volume={313},
   date={2013},
   pages={1342--1348},
   note={\arxiv{1008.3724}},
}

\bib{Bo}{article}{
   author={Boardman, J. M.},
   title={Singularities of differentiable maps},
   journal={Inst. Hautes \'{E}tudes Sci. Publ. Math.},
   number={33},
   date={1967},
   pages={21--57},
   note={\href{http://www.numdam.org/item/PMIHES_1967__33__21_0/}{NUMDAM}}
}

\bib{BRS}{book}{
   author={Buoncristiano, S.},
   author={Rourke, C. P.},
   author={Sanderson, B. J.},
   title={A Geometric Approach to Homology Theory},
   series={London Math. Soc. Lecture Note Series},
   volume={18},
   publisher={Cambridge Univ. Press},
   date={1976},
}

\bib{CS}{article}{
   author={Carter, J. S.},
   author={Saito, M.},
   title={Surfaces in $3$-space that do not lift to embeddings in $4$-space},
   conference={
      title={Knot theory},
      address={Warsaw},
      date={1995},
   },
   book={
      series={Banach Center Publ.},
      volume={42},
      publisher={Polish Acad. Sci. Inst. Math.},
      place={Warsaw},
   },
   date={1998},
   pages={29--47},
}

\bib{Fa}{article}{
   author={Farrell, F. T.},
   title={Right-orderable deck transformation groups},
   journal={Rocky Mountain J. Math.},
   volume={6},
   date={1976},
   pages={441--447},
   note={\href{https://projecteuclid.org/journals/rocky-mountain-journal-of-mathematics/volume-6/issue-3/Right-orderable-deck-transformation-groups/10.1216/RMJ-1976-6-3-441.full}{ProjectEuclid}},
}


\bib{FK}{article}{
   author={R. Fulek},
   author={J. Kyn\v cl},
   title={Hanani-Tutte for approximating maps of graphs},
   conference={
      title={34th International Symposium on Computational Geometry},
   },
   book={
      series={LIPIcs. Leibniz Int. Proc. Inform.},
      volume={99},
      publisher={Schloss Dagstuhl. Leibniz-Zent. Inform., Wadern},
   },
   date={2018, No. 39 (15pp.)},
   note={\arxiv{1705.05243}},
}

\bib{FM}{article}{
   author={Fulton, W.},
   author={MacPherson, R.},
   title={A compactification of configuration spaces},
   journal={Ann. of Math.},
   volume={139},
   date={1994},
   pages={183--225},
}

\bib{FP}{article}{
   author={L. Funar},
   author={P. G. Pagotto},
   title={Braided surfaces and their characteristic maps},
   journal={New York J. Math.},
   volume={29},
   date={2023},
   pages={580--612},
   note={\arxiv{2004.09174}},
}

\bib{GWPL}{book}{
   author={Gibson, C. G.},
   author={Wirthm\"{u}ller, K.},
   author={du Plessis, A. A.},
   author={Looijenga, E. J. N.},
   title={Topological Stability of Smooth Mappings},
   series={Lecture Notes in Math.},
   volume={552},
   publisher={Springer},
   date={1976},
}

\bib{Gi}{article}{
   author={Giller, C. A.},
   title={Towards a classical knot theory for surfaces in $\R^4$},
   journal={Illinois J. Math.},
   volume={26},
   date={1982},
   pages={591--631},
}

\bib{GG}{book}{
   author={Golubitsky, M.},
   author={Guillemin, V.},
   title={Stable Mappings and their Singularities},
   series={Grad. Texts in Math.},
   volume={14},
   publisher={Springer},
   date={1973},
   note={(Russian transl., Mir, 1977)},
}

\bib{Go}{article}{
author={A. A. Gorelov},
title={Lifting maps between graphs to embeddings},
book={
 title={Essays on Topology --- Dedicated to Valentin Po\'enaru},
 editor={L. Funar},
 editor={A. Papadopoulos},
 publisher={Springer},
 date={2025},
 },
 pages={327--364},
 note={\arxiv{2404.12287}}
}


\bib{Gr}{book}{
   author={Gromov, M.},
   title={Partial Differential Relations},
   series={Ergebn. der Math.},
   volume={9},
   publisher={Springer},
   place={Berlin},
   date={1986},
   translation={
    language={Russian},
    publisher={Mir},
    place={Moscow},
    date={1990},
    }
}

\bib{Hi}{book}{
   author={Hirsch, M. W.},
   title={Differential Topology},
   series={Grad. Texts in Math.},
   volume={33},
   publisher={Springer},
   date={1976},
}

\bib{Hou}{article}{
   author={Houston, K.},
   title={Topology of differentiable mappings},
   book={
   title={Handbook of Global Analysis},
      editor={D. Krupka}
      editor={D. Saunders}
      publisher={Elsevier},
   },
   date={2008},
   pages={493--532, 1213},
   note={\href{http://www.kevinhouston.net/pdf/globan.pdf}{Author's homepage}}
}

\bib{Ka}{article}{
author={M. Kazarian},
title={Morin maps and their characteristic classes},
note={Preprint. \href{https://nanopdf.com/download/morin-maps-and-their-characteristic-classes_pdf}{Nanopdf.com}}
}

\bib{KM}{article}{
   author={Kervaire, M. A.},
   author={Milnor, J. W.},
   title={Groups of homotopy spheres. I},
   journal={Ann. of Math.},
   volume={77},
   date={1963},
   pages={504--537},
   note={\href{https://www.maths.ed.ac.uk/~v1ranick/papers/kervmiln.pdf}{A. Ranicki's archive}}
}

\bib{MM}{article}{
   author={Marar, W. L.},
   author={Mond, D.},
   title={Multiple point schemes for corank $1$ maps},
   journal={J. London Math. Soc.},
   volume={39},
   date={1989},
   pages={553--567},
}

\bib{MaIV}{article}{
   author={Mather, J. N.},
   title={Stability of $C^\infty$ mappings. IV. Classification of stable germs by $R$-algebras},
   journal={Inst. Hautes \'{E}tudes Sci. Publ. Math.},
   number={37},
   date={1969},
   pages={223--248},
   note={\href{http://www.numdam.org/item/?id=PMIHES_1969__37__223_0}{NUMDAM}},
}

\bib{MaVI}{article}{
   author={Mather, J. N.},
   title={Stability of $C^\infty$ mappings. VI: The nice dimensions},
   conference={
      title={Proc. of Liverpool Singularities Symp. I},
   },
   book={
      series={Lecture Notes in Math.},
      volume={192},
      publisher={Springer},
   },
   date={1971},
   pages={207--253}
}

\bib{Mc}{article}{
   author={McCrory, C.},
   title={Cobordism operations and singularities of maps},
   journal={Bull. Amer. Math. Soc.},
   volume={82},
   date={1976},
   pages={281--283},
   note={\href{https://www.ams.org/journals/bull/1976-82-02/S0002-9904-1976-14023-2/}{Journal}}
}

\bib{M1}{article}{
   author={Melikhov, S. A.},
   title={Sphere eversions and realization of mappings},
   journal={Proc. Steklov Inst. Math.},
   date={2004},
   number={247},
   pages={143--163; \arxiv[GT]{0305158}},
   translation={
      language={Russian},
      journal={Tr. Mat. Inst. Steklova},
      volume={247},
      date={2004},
      pages={159--181},
      note={\href{http://mi.mathnet.ru/tm15}{MathNet}},
   },
}

\bib{M2}{article}{
   author={Melikhov, S. A.},
   title={The van Kampen obstruction and its relatives},
   journal={Tr. Mat. Inst. Steklova},
   volume={266},
   date={2009},
   pages={149--183},
   note={Reprinted in: Proc. Steklov Inst. Math. {\bf 266} (2009), 142--176; \arxiv[GT]{0612082}},
}

\bib{M3}{article}{
   author={Melikhov, S. A.},
   title={Transverse fundamental group and projected embeddings},
   journal={Proc. Steklov Inst. Math.},
   date={2015},
   number={290},
   pages={155--165; \arxiv{1505.00505}},
   translation={
      language={Russian},
      journal={Tr. Mat. Inst. Steklova},
      volume={290},
      date={2015},
      pages={166--177},
      note={\href{http://mi.mathnet.ru/tm3646}{MathNet}}
   },
}

\bib{M4}{article}{
   author    = {Melikhov, S. A.},
   title     = {Combinatorics of combinatorial topology},
   note={\arxiv{1208.6309v2}},
}

\bib{M5}{article}{
   author    = {Melikhov, S. A.},
   title     = {A triple-point Whitney trick},
   journal   = {J. Topol. Anal.},
   volume={12},
   number={4},
   pages={1041--1046},
   date={2020}
   note={\arxiv{2210.04016}}
}

\bib{partI}{article}{
   author={Melikhov, S. A.},
   title={Lifting generic maps to embeddings. Triangulation and smoothing},
   journal={Stud. Sci. Math. Hung.},
   volume={59},
   number={2},
   pages={124--141},
   date={2022}
   note={\arxiv{2011.01402}}
}

\bib{MZ}{article}{
   author    = {S. A. Melikhov},
   author    = {J. Zaj\c ac},
   title     = {Contractible polyhedra in products of trees and absolute retracts in products of dendrites},
    journal  = {Proc. Amer. Math. Soc.},
       volume   = {141},
       year     = {2013},
       pages    = {2519--2535},
   note    = {\arxiv{1102.0696}},
}

\bib{Mo}{article}{
   author={Morin, B.},
   title={Formes canoniques des singularit\'{e}s d'une application diff\'{e}rentiable},
   journal={C. R. Acad. Sci. Paris},
   volume={260},
   date={1965},
   pages={5662--5665},
   translation={
   language={Russian},
   note={Osobennosti differenciruemyh otobrazhenij (V. I. Arnold, ed.), pp. 153--158, Mir, Moscow, 1968},
   },
}

\bib{Mu}{book}{
   author={J. R. Munkres},
   title={Elementary Differential Topology},
   series={Ann. Math. Studies},
   volume={54},
   date={1966},
   publisher={Princeton Univ. Press},
   note={Russian transl. in J. Milnor, J. Stasheff, Harakteristicheskie klassy, Mir, Moscow, 1979, 
   pp. 270--359} 
}

\bib{Pe}{article}{
   author={Petersen, P., V},
   title={Fatness of covers},
   journal={J. Reine Angew. Math.},
   volume={403},
   date={1990},
   pages={154--165},
   note={\href{http://resolver.sub.uni-goettingen.de/purl?GDZPPN002207184}{GDZ}},
}


\bib{Po2}{book}{
   author={Po\'{e}naru, V.},
   title={Homotopie r\'eguli\`ere et isotopie},
   series={Publ. Math. Orsay},
   date={1974},
   publisher={Univ. Paris XI, U.E.R. Math.},
   place={Orsay},
   note={No. 106 (23), 54 pp.; \href{http://sites.mathdoc.fr/PMO/feuilleter.php?id=PMO_1974}{Mathdoc}}
}

\bib{Por}{article}{
   author={Porteous, I. R.},
   title={Simple singularities of maps},
   conference={
      title={Proc. Liverpool Singularities Symposium, I (1969/70)},
   },
   book={
      series={Lecture Notes in Math.},
      volume={192},
      publisher={Springer}, 
      date={1971},    
      },
   pages={286--307}
}

\bib{Ro}{article}{
   author={Ronga, F.},
   title={``La classe duale aux points doubles'' d'une application},
   journal={Compositio Math.},
   volume={27},
   date={1973},
   pages={223--232},
   note={\href{http://www.numdam.org/item/?id=CM_1973__27_2_223_0}{NUMDAM}},
}

\bib{Ro2}{article}{
   author={Ronga, F.},
   title={Desingularisation of the triple points and of the stationary points of a map},
   journal={Compositio Math.},
   volume={53},
   date={1984},
   pages={211--223},
   note={\href{http://www.numdam.org/item/?id=CM_1984__53_2_211_0}{NUMDAM}},
}

\bib{RS1}{article}{
author={C. Rourke},
author={B. Sanderson},
title={The compression theorem I},
journal={Geometry and Topology},
volume={5},
date={2001},
pages={399--429},
note={\arxiv[GT]{9712235}},
}

\bib{RS2}{article}{
author={C. Rourke},
author={B. Sanderson},
title={The compression theorem III: applications},
journal={Alg. Geom. Topol.},
volume={3},
date={2003},
pages={857--872},
note={\arxiv[GT]{0301356}},
}

\bib{Sa}{article}{
   author={Satoh, S.},
   title={Lifting a generic surface in 3-space to an embedded surface in 4-space},
   journal={Topology Appl.},
   volume={106},
   date={2000},
   pages={103--113},
   note={\href{https://www.sciencedirect.com/science/article/pii/S0166864199000747}{Journal}}
}

\bib{Shi}{article}{
   author={Shiota, M.},
   title={Thom's conjecture on triangulations of maps},
   journal={Topology},
   volume={39},
   date={2000},
   pages={383--399},
   note={\href{https://www.sciencedirect.com/science/article/pii/S0040938399000221}{Journal}}
}

\bib{Si}{article}{
   author={Siek\l ucki, K.},
   title={Realization of mappings},
   journal={Fund. Math.},
   volume={65},
   date={1969},
   pages={325--343},
   note={\href{https://eudml.org/doc/214113}{EuDML}}
}

\bib{Sin}{article}{
   author={Sinha, D. P.},
   title={Manifold-theoretic compactifications of configuration spaces},
   journal={Selecta Math.},
   volume={10},
   date={2004},
   number={3},
   pages={391--428},
   note={\arxiv[GT]{0306385}},
}

\bib{Sk}{article}{
   author={Skopenkov, M.},
   title={On approximability by embeddings of cycles in the plane},
   journal={Topol. Appl.},
   volume={134},
   date={2003},
   pages={1--22},
   note={\arxiv{0808.1187}}
}

\bib{Sz1}{article}{
   author={Sz\H{u}cs, A.},
   title={The Gromov--Eliashberg proof of Haefliger's theorem},
   journal={Studia Sci. Math. Hungar.},
   volume={17},
   date={1982},
   pages={303--318},
   note={\href{http://real-j.mtak.hu/5483/}{Real-J}}
}

\bib{Sz2}{article}{
   author={Sz\H{u}cs, A.},
   title={Multiple points of singular maps},
   journal={Math. Proc. Cambridge Philos. Soc.},
   volume={100},
   date={1986},
   pages={331--346},
}

\bib{Sz3}{article}{
   author={Sz\H{u}cs, A.},
   title={Topology of $\Sigma^{1,1}$-singular maps},
   journal={Math. Proc. Cambridge Philos. Soc.},
   volume={121},
   date={1997},
   pages={465--477},
}

\bib{Sz4}{article}{
   author={Sz\H{u}cs, A.},
   title={On the cobordism group of Morin maps},
   journal={Acta Math. Hungar.},
   volume={80},
   date={1998},
   pages={191--209},
}

\bib{Ve}{book}{
   author={Verona, A.},
   title={Stratified Mappings --- Structure and Triangulability},
   series={Lecture Notes in Math.},
   volume={1102},
   publisher={Springer},
   place={Berlin},
   date={1984},
}

\bib{Wa1}{article}{
   author={Wall, C. T. C.},
   title={Lectures on $C\sp{\infty }$-stability and classification},
   conference={
      title={Proceedings of Liverpool Singularities --- Symposium I (1969/70)},
   },
   book={
      series={Lecture Notes in Math.}
      volume={192},
      publisher={Springer},
      address={Berlin},
   },
   date={1971},
   pages={178--206},
}

\bib{Wa2}{article}{
   author={Wall, C. T. C.},
   title={Geometric properties of generic differentiable manifolds},
   conference={
      title={Geometry and Topology --- Rio de Janeiro, Brazil 1976},
   },
   book={
      series={Lecture Notes in Math.},
      volume={597},
      publisher={Springer},
      address={Berlin},
   },
   date={1977},
   pages={707--774},
}


\bib{Ze}{book}{
author={Zeeman, E. C.},
title={Seminar on Combinatorial Topology},
publisher={Inst. Hautes Etudes Sci. Publ. Math.},
place={Paris},
date={1966},
note={\href{http://www.maths.ed.ac.uk/~aar/papers/zeemanpl.pdf}{A. Ranicki's webpage}},
}

\end{thebibliography}
\end{document}